\documentclass[english]{smfart}
\usepackage{baskervald}
\usepackage[utf8x]{inputenc} 
\usepackage[english]{babel}
\usepackage{amsthm} 
\usepackage{amsmath} 
\usepackage{amsfonts}
\usepackage{amssymb}
\usepackage[disable]{todonotes}

\newcommand\NN{\mathbb{N}} 
\newcommand\ZZ{\mathbb{Z}} 
\newcommand\QQ{\mathbb{Q}} 
\newcommand\RR{\mathbb{R}} 
\newcommand\CC{\mathbb{C}} 

\newcommand\eps{\varepsilon} 
\newcommand\fleche{\longrightarrow} 

\DeclareMathOperator{\Ker}{Ker}
\DeclareMathOperator{\ord}{ord}

\DeclareMathOperator{\Cond}{Cond}

\DeclareMathOperator{\Coker}{Coker}
\DeclareMathOperator{\Frac}{Frac}

\DeclareMathOperator{\Isom}{Isom}

\DeclareMathOperator{\tr}{tr}
\DeclareMathOperator{\Tr}{Tr}

\DeclareMathOperator{\Ht}{ht}
\DeclareMathOperator{\Hom}{Hom}

\DeclareMathOperator{\Spec}{Spec}
\DeclareMathOperator{\Spf}{Spf}

\DeclareMathOperator{\Gal}{Gal}
\DeclareMathOperator{\Fil}{Fil}
\DeclareMathOperator{\im}{Im}

\DeclareMathOperator{\Ha}{Ha}
\DeclareMathOperator{\Frob}{Frob}

\DeclareMathOperator{\Diag}{Diag}
\DeclareMathOperator{\End}{End}
\DeclareMathOperator{\Res}{Res}

\DeclareMathOperator{\Deg}{Deg}

\DeclareMathOperator{\GL}{GL}
\DeclareMathOperator{\HT}{HT}
\DeclareMathOperator{\ha}{ha}
\DeclareMathOperator{\Sum}{S}

\DeclareMathOperator{\Spm}{Spm}

\DeclareMathOperator{\wt}{wt}

\DeclareMathOperator{\Ext}{Ext}

\usepackage{tikz}
\usetikzlibrary{matrix,arrows}
\definecolor{cqcqcq}{rgb}{0.752941176471,0.752941176471,0.752941176471}
\definecolor{ffqqqq}{rgb}{0.333333333333,0.333333333333,0.333333333333}
\definecolor{cqcqcq}{rgb}{0.752941176471,0.752941176471,0.752941176471}
\definecolor{qqqqff}{rgb}{0.,0.,1.}
\definecolor{cqcqcq}{rgb}{0.752941176471,0.752941176471,0.752941176471}
\definecolor{ffqqqq}{rgb}{1.,0.,0.}
\usepackage{mathabx,epsfig}

\usepackage[all]{xy}
 \def\dar[#1]{\ar@<2pt>[#1]\ar@<-2pt>[#1]}
 \def\tar[#1]{\ar@<4pt>[#1]\ar@<0pt>[#1]\ar@<-4pt>[#1]}

\theoremstyle{definition} 
\newtheorem{definen}{Definition}[section]
\newtheorem{defin}[definen]{Definition}

\theoremstyle{plain} 

\newtheorem{theor}[definen]{Theorem}
\newtheorem{lemma}[definen]{Lemma}  
\newtheorem{prop}[definen]{Proposition}

\newtheorem{cor}[definen]{Corollary}

\theoremstyle{remark} 
\newtheorem{rema}[definen]{Remark}

\begin{document}

\title{Families of coherent PEL automorphic forms.}
\author{Valentin Hernandez}
\date{}
\address{Bureau iC1, LMO, Orsay, France}
\email{valentin.hernandez@math.cnrs.fr}
\maketitle

\tableofcontents 

\section{Introduction}
%
%

Families of automorphic forms have proven to be a great tool in number theory in the last 30 years. Their construction dates back to Hida, \cite{Hida}, who first constructed families of 
\textit{ordinary} modular forms (for the group $\GL_2$). This construction was then improved by Coleman in the 1990's, for \textit{overconvergent, finite slope}, modular forms and rigid spaces 
over $\QQ_p$ (whereas Hida was able to construct his families integrally). One great and yet surprising achievement was the construction soon after by Coleman and Mazur of 
one rigid space, the Eigencurve, which parametrizes all possible families of overconvergent, finite slope modular forms, i.e. gluing all the families previously constructed.

Before motivating the construction of this spaces, let us say that these constructions have seen many generalisations in different directions. First dealing with level outside $p$ and quaternion algebra by Buzzard \cite{Buz}, or for other algebraic groups, unitary groups, compact at infinity by Chenevier \cite{Che1}, and to more general groups by \cite{Ash-Stevens} and \cite{urb} using families of (generalised) modular symbols. More recently, \cite{AIP} have been able to construct families and eigenvarieties for Siegel modular forms using families of automorphic sheaves on the Siegel moduli space. These families of sheaves live in the rigid world, they are Banach sheaves on certain strict neighborhoods of the ordinary locus, that interpolates (in some sense) the classical automorphic vector bundles. This strategy has been extended by \cite{Kassaei,BraS} in the 
case of Shimura curves, \cite{AIP2} for Hilbert modular forms, and \cite{Brasca} for PEL Shimura varieties for which the ordinary locus in non empty.

This spaces are particularly interesting ; through their local properties (see for example \cite{BC2} and \cite{CH} for applications to the Bloch-Kato conjecture, and to constructing Galois representation associated to automorphic representations), but also for their global geometry (see \cite{LWX} and the application to the parity conjecture), which remains completely mysterious in general.

In all cases, the construction goes by constructing huge Banach spaces $M$ together with an action of a (commutative) Hecke algebra $\mathbb T$ containing a distinguished compact operator $U$.
With this data, if $M$ is a projective Banach space, we can construct following \cite{ColBanach} a rigid space $\mathcal E$ which parametrises Hecke eigensystems for $\mathbb T$ acting on $M$, for which the eigenvalue for $U$ is non-zero.
In \cite{Ash-Stevens} and \cite{urb}, theses spaces $M$ are the sections on Shimura varieties of $p$-adic overconvergent modular symbols, which interpolate the etale 
cohomology of these varieties. In \cite{AIP} and its generalisations, one first construct varying Banach automorphic sheaves $\omega^{\kappa\dag}$, where $\kappa$ is a $p$-
adic weight, and take the sections of theses sheaves on strict neighborhoods of the ordinary locus. Theses spaces interpolate the coherent cohomology, but are constructed on 
PEL Shimura varieties (one needs the moduli interpretation), and need the non emptyness of the ordinary locus. Indeed, one central tool to construct $\omega^{\kappa\dag}$ is 
the theory of the canonical subgroup and its overconvergence (see \cite{Lubin,Far2} for example). 
In this article, we mainly remove the ordinariness assumption. Let $(G,X)$ be a PEL Shimura datum\footnote{We exclude factors of type D}, and $p$ a prime. 
Our main result is the following

\begin{theor}
Suppose that $G$ is unramified at $p$, and let $K^p$ be a level outside $p$, hyperspecial outside a finite set of primes $S$. 
Let $I$ be a Iwahori sugbroup at $p$ and $K = K^pI$.
There exists rigid spaces $\mathcal E$ and $\mathcal W$, called respectively the eigenvariety and the weight space, together with a locally finite map
\[ w : \mathcal E \fleche \mathcal W,\]
and $\mathbb T = \mathbb T^{Sp}\otimes A(p) \fleche \mathcal O(\mathcal E)$ such that, for all $\kappa \in \mathcal W$, $w^{-1}(\kappa)$ is in bijection with the eigenvalues for the Hecke algebra $\mathbb T$ acting on weight $\kappa$, overconvergent, locally analytic modular forms for $G$
which are finite slope for some $U \in \mathcal A(p)$. Here $\mathcal A(p)$ is a (commutative) Hecke algebra at $p$
\todo{Check si tout $\mathcal A(p)$ agit de manière compacte : non, mais de pente finie, oui (cf. Chenevier1 (these) 4.7.1} 
and $\mathbb T^{Sp}$ is the unramified Hecke algebra for $G$ outside $Sp$. $\mathcal E$ and $\mathcal W$ are equidimensionnal of the same dimension.
Moreover there is a Zariski dense subset $\mathcal Z \subset \mathcal E$ such that all $z\in \mathcal Z$ coïncide with a \textit{classical} Hecke eigensystem in the previous identification.
\end{theor}

Actually, we can only construct families at unramified primes, but we can weaken a bit the assumptions on $G$ and $p$, by only constructing deformations in the directions of primes above $p$ which are unramified for $G$, see remark \ref{remramif}. 

We now explain how we prove this theorem. A first step in generalising the construction of \cite{AIP} to the case when the ordinary locus is empty is to find a subsitute for the ordinary locus and the canonical subgroup. A good substitute is to consider the $\mu$-ordinary locus (see \cite{Wed1}, \cite{Moo}, and also \cite{Bijmu}), and the canonical filtration, which exists on it, and overconverges on strict neighborhoods (see \cite{Her2}). This strategy has been followed in \cite{Her3} for $U(2,1)$ when $p > 2$. Unfortunately, the results of \cite{Her2} rely on a stronger hypothesis on $p$ : being big enough (always $p \neq 2$ and for general unitary group for example the bound can be very large). In this article we choose another strategy to avoid any hypothesis on $p$, and use (integral) Shimura varieties with higher (Iwahori-like) level at $p$, constructed by normalisation in \cite{LanSigma}. On these Shimura varieties naturally live flags of finite flat subgroups, and if we restrict to strict neighborhoods of the $\mu$-ordinary locus (more precisely what we call the $\mu$-canonical locus, see definition \ref{defmucan}), theses groups behave as the canonical filtration (and actually coincides with it when we know it exists, see Theorem \ref{thrfiltcan}). In particular, we can follow the construction of \cite{AIP} and \cite{Her3} for all $p$ with theses groups, and construct automorphic banach sheaves by introducing level at $p$. All of this rely on the fact that we can find a basis of strict neighborhood $X(deg \geq N- \eps)$ where our subgroups have \textit{high degree}, and thus are well behaved. 

In the setting where the ordinary locus is non empty, by results of Fargues \cite{Far2} we can relate degree and the Hasse invariant. In our situation we also have an Hasse 
invariant (by \cite{GN}; see also \cite{Her1} and Definition \ref{def59}), but we can relate it to the degrees, using \cite{Her2}, only if $p$ is big enough... Thus we chose another 
strategy : we have a second basis of strict neighborhoods $X(\ha \leq v)$ where the (valuation of) Hasse invariant is small enough (it is invertible on the $\mu$-ordinary locus), 
and we use these two basis of neighborhoods. Using the degree function, we can control our (call them canonical) subgroups easily, and thus as it was already remarked in 
\cite{Bijmu}, the action of the Hecke operators. In particular, we can check that we have an operator $U$ which acts as a compact operator on sections of our sheaves over $X(\deg \geq N-\eps)$. 
Unfortunately, we can't prove that the global sections over the opens $X(\deg \geq N-\eps)$ of the automorphic Banach sheaves are projective, thus we can't a priori use Coleman-
Buzzard's construction. On the other basis $(X(\ha \leq v))_{v > 0}$, we can't prove even that our expected-to-be compact operator $U$ (which generalise the operator $U_p$ on the modular curve) will stabilise each neighborhood (and thus worse, that it acts compactly on 
sections on $X(\ha \leq v)$), but using that $X(\ha \leq v)$ is affinoïd in rigid fiber, we can prove that global section of our automorphic Banach sheaves on $X(\ha \leq v)$ are 
projective. Here to be precise we need to work on both the toroïdal and minimal compactifications of \cite{LanSigma}, the toroïdal compactification being needed to construct the 
automorphic sheaves, and the minimal to get the affinoïd result, together with a result of vanishing of higher cohomology due to Lan, see Appendix \ref{AppA}.
Thus we need to relate both these sections on the two basis of neighborhoods. Fortunately we can and do in section \ref{sect9} using complexes computing higher cohomology of our Banach sheaves, the action of the Hecke operators on these complexes, and
that we can always intertwine these opens,
\[ X(deg \geq N - \eps) \supset X(\ha \leq v) \supset X(deg \geq N - \eps') \supset X(\ha \leq v') \supset X^{\mu-can},\]
where $\eps'$ and $v'$ are well chosen small enough. Passing to finite slope parts, and using results of \cite{urb}, we get that $U$ acts as a compact operator on the finite slope part of sections of our Banach automorphic sheaves on any of our strict neighborhoods, and that theses spaces are projective (in some sense). Thus we can apply Coleman-Buzzard's machinery and get the theorem. 

As an application of these results, we can extend the result on the Bloch-Kato conjecture we had in \cite{Her3}, and prove the following. Let $E$ be a quadratic imaginary number field, and 
\[ \chi : \mathbb A_E^\times/E^\times \fleche \CC^\times,\]
which is polarised, meaning that $\chi^\perp:=(\chi^c)^{-1} = \chi |.|^{-1}$. Denote by $L(\chi,s)$ its $L$-function. If $p$ is a prime, denote
\[ \chi_p : \Gal(\overline E/E) \fleche \overline{\QQ_p}^\times,\]
the $p$-adic Galois character associated to $\chi$, and denote $H^1_f(E,\chi_p)$ the Bloch-Kato-Selmer group of $\chi_p$ (see \cite{BC2} chapter 5).
Then we prove

\begin{theor}
Let $p$ be a prime, unramified in $E$. If $L(\chi,0) = 0$ and $\ord_{s=0}L(\chi,s)$ is even, then
\[	H^1_f(E,\chi_p) \neq 0.\]
\end{theor}
In particular we remove the hypothesis that $p \neq 2$ when $p$ is inert in $E$ and $p \notdivides \Cond(\chi)$ that were in \cite{Her3}. Also, a version of the previous theorem is well-known to be due to Rubin (\cite{Rub}) but there it is necessary that $p \neq 2$ (and $p \neq 3$ if $E = \QQ(j)$, which we unfortunately also need to assume...). In particular, we get new cases of the Bloch-Kato conjecture when $p = 2$ is unramified in $E$ !

Of course this result relies heavily, as in \cite{Her3}, on works of Bellaïche and Chenevier, \cite{BC1} and \cite{BC2}. As in this last reference, we can even construct independent classes as predicted by the Bloch-Kato conjecture, under some assumption on the eigenvariety $\mathcal E$ for $U(2,1)$. The idea is to consider a specific \textit{Arthur point} $y \in \mathcal E$ (known to exists by results of Rogawski and a calculation of cohomology in \cite{Her3}), see Propositon \ref{propx0}. Denote by $t$ the dimension of the tangent space of $\mathcal E$ at $y$. As $\mathcal E$ is of dimension $3$, we have $t \geq 3$. There exists a subspace
\[ Ext_T(1,\overline{\chi}) \subset H^1_f(E,\chi_p),\]
and denote its dimension by $h$. The assumptions of $\chi$ implies that $\chi_\infty(z) = z^a \overline{z}^{1-a}$, for all $z \in \CC$, with $a \in \ZZ_{\geq1}$ (up to change $\chi$ by its conjugate $\chi^c$\footnote{Sometimes we denote it by $\overline\chi$ instead of $\chi^c$}). We can prove the following,

\begin{theor}
Recall that we assume $L(\chi,0) = 0$ and $\ord_{s=0}L(\chi,s)$ is even. Then, if $a \geq 2$ and $p$ is split,
\[ t \leq h\frac{h+3}{2} + 1.\]
\end{theor}

In particular this implies the previous theorem for $a \geq 2$ and $p$ split. Also in this case, if $\mathcal E$ is non regular at $y$, then $h > 1$, and there are thus at least 2 classes in $H^1_f(E,\chi_p)$. Remark that $a \geq 2$ is exactly the assumption so that the Hecke eigenvalues at $y$ actually appears in etale cohomology, when $a = 1$ we can only prove that these eigenvalues appear in the coherent cohomology, the motives becomes irregular and our argument breaks down. Also, the hypothesis that $p$ is split is necessary to choose a good refinement which is \textit{sufficently far} from the ordinary one (we need it to be anti-ordinary see \cite{BC1}) : when $p$ is inert, there is a unique accessible refinement, and it is \textit{not} anti-ordinary.
%

\subsection*{Acknowledgements}

I would like to first thank Stephane Bijakowski for many interesting discussions and suggesting to systematically use the degree function instead of the Hasse invariant. I would also warmly thank Vincent Pilloni for introducing me to this subject and explaining his recent results on higher Hida theory, suggesting that they could be usefull in the case at hand. Also I would like to warmly thank Gaetan Chenevier and Benjamin Schraen for taking time to explain me in details a lot of their knowledge on the subject. I would also like to thank Fabrizio Andreatta, Christophe Breuil, Laurent Fargues, Eyal Goren and Adrian Iovita for interesting discussions.

\section{Algebraic groups, Shimura Datum and weight spaces}
\label{sect2}
Let $p$ be a prime and let $\mathcal D = (B,\star,V, <,> ,\mathcal O_B, \Lambda,h)$ be an integral Shimura-PEL-datum.
Let $G$ be the associated algebraic group over $\QQ$, i.e.,
\[ G(R) = \{ (g,c(g)) \in GL_B(V\otimes R)\times R^\times | <gv,gw> = c(g) <v,w> \forall v,w \in V\otimes R\}.\]
$(G,h)$ defines a Shimura datum. Suppose that the datum is unramified at $p$ (see \cite{KotJams} or \cite{WedVieh}). This means that $B\otimes_\QQ \QQ_p$ is isomorphic to a product of matrix algebras over finite extensions of $\QQ_p$.
We can decompose $B = \prod_{i=1}^r B_i$ as a product of simple algebras and we suppose that no factor is of type D (orthogonal), see \cite{WedVieh} Remark 1.1.
As $p$ is unramified in $\mathcal D$, we can also consider $\mathbb G$ a reductive model at $p$ for $G$ (over $\ZZ_p$).

Every interesting object in this article will decomposed accordingly to the previous decomposition of $B$, and we can thus make our construction for each $B_i$. This simple algebras are classified into 2 types (as we excluded case D), the type A and the type C. In case C, the construction we are interested in is already made in \cite{Brasca} (which also do many cases of type A, but not all), and we thus assume for now on that $B_i$ is of type A.

As $p$ is unramified for $B$ (and thus $B_i$) we can further decompose. Let $F$ be the center of $B_i$, and $F_0 = (F)^{\star = 1}$. As we are in case A, $[F:F_0] = 2$.
Write $p = \pi_1 \dots \pi_{s_i}$ the decomposition of $p$ in primes of $F_0$. For $j \in \{1,\dots,s_i\}$, we say that $j$ (or $\pi_j$ or $(B_i,\pi_j)$ is in case AL is $\pi_j$ splits in $F$, and in case AU otherwise (compare \cite{WedVieh} Remark 1.3).

\begin{rema}
\label{remramif}
Actually we can allow a slightly larger class of Shimura datum than the unramified ones. Suppose that $B \otimes \QQ_p = B_1 \times B_2$ where 
\[ B_1 = \prod_{i=1}^r M_{n_i}(F_i),\]
where and $K/\QQ_p$ is a finite extension, and such that there is no factor of type D appearing in $B_{\QQ_p}$. For all $i$ denote again $F_0 = (F_i)^{\star = 1}$, and denote $\pi_1,\dots,\pi_{s_i}$ the prime over $p$. Let then $S_p^{full}$ be the set of couples $(i,j)$ such that $\pi^j$ is not ramified in $F_0$ and does not ramifies in $F_i$ either.
When $p$ is unramfied in the datum $\mathcal D$, we can take $S_p^{max}$ to be the set of all $(i,j)$. In general, for $S_p \subset S_p^{max}$, we will be able to construct $S_p$-families of automorphic forms for the datum $\mathcal D$, i.e. we are able to let the forms vary (only) along the unramified primes of $\mathcal D$.
\end{rema}

Let $T$ be the center of $\mathbb G_1 = \Ker c \subset \mathbb G$. We can decompose $T$ (over $\ZZ_p$) according to the previous decomposition,
\[ T = \prod_{i=1}^r \prod_{j=1}^{s_i} T_{i,j},\]
(remark that if $B_i$ is of type C, we can also decompose according to primes over $p$).

\begin{defin}
The full weight space associated to the previous PEL datum is the rigid space over $\QQ_p$
\[ \mathcal W^{full} = \Hom_{cont}(T(\ZZ_p),\mathbb G_m^{rig}),\]
which associate to any Banach $\QQ_p$-algebra $R$ the set of continuous characters $\Hom_{cont}(T(\ZZ_p),R^\times)$. It is represented by the Banach algebra $\ZZ_p[[T(\ZZ_p)]]$.

If $S_p$ is a subset of the couples $(i,j)$ (that we see as places over $p$) and if we denote $T_{S_p}$ the torus over $\ZZ_p$;
\[ T_{S_p} = \prod_{(i,j) \in S_p} T_{i,j},\]
we can define the ($S_p$-)weight space \[\mathcal W_{S_p} = \Hom_{cont}(T_{S_p}(\ZZ_p),\mathbb G_m^{rig}).\]
It is also represented by the Banach algebra $\ZZ_p[[T_{S_p}(\ZZ_p]]$, and when $S_p$ contains all couples $(i,j)$, we have $\mathcal W_{S_p} =\mathcal W^{full}$.
\end{defin}

On $\mathcal W_{S_p}$ there is a universal character $\kappa^{univ} : T_{S_p}(\ZZ_p) \fleche \ZZ_p[[T_{S_p}(\ZZ_p)]]$. 
We have the following results,

\begin{prop}
The space $\mathcal W_{_p}S$ is geometrically a finite disjoint union of dimension the rank of $\mathbb G_1$. Moreover there exists a admissible covering by increasing affinoids,
\[ \mathcal W_{S_p} = \bigcup_{w >0} \mathcal W_{S_p}(w),\]
such that $\kappa^{univ}_{|\mathcal W_{S_p}(w)}$ is $w$-analytic. 
\end{prop}

\begin{proof}
See \cite{urb} 3.4.2 and Lemma 3.4.6. See \cite{AIP}, section 2.2 for a definition of $\mathcal W_{S_p}(w)$.
\end{proof}

We can decompose $\mathcal W_{S_p} = \prod_{(i,j) \in S_p} \mathcal W_{i,j}$ according to the decomposition of $B$. In the following we will construct families parametrized by $\mathcal W_{S_p}$, as their construction is not more difficult than the case of the full weight space, and following construction can be done on $\mathcal W_{S_p}$ when $p$ ramified at some places of $\mathcal D$, but not at other places. To my knowledge, this is usefull mainly for a trick used by Chenevier (\cite{Chetrick}) to control $p$-adic properties of families of Galois representations.

\section{Classical coherent Automorphic forms}
\label{sect3}
Associated to $(G,h)$ there is a tower of Shimura Varieties of the reflex field $E$. 
Because of the assumption of $p$ in $\mathcal D$, these Shimura varieties have good reduction at $p$ when the level at $p$ is hyperspecial (see \cite{KotJams}). Suppose this is the case in this section (otherwise all we say here remains true after inverting $p$, and we will explain how to extends this integrally in section \ref{sect5}).
We will describe their integral models as moduli space of Abelian varieties.
Let $K^p \subset G(\mathbb A_{\QQ,f}^p)$ be sufficiently small level outside $p$. Denote $X_{K^p}$ the functor,
\[ X_{K^p} : S \in Sch/\Spec(\mathcal O_{E,p}) \fleche \{ (A,i,\lambda,\eta)\}/\sim,\]
that associated the set of quadruple $(A,i,\lambda,\eta)$ modulo equivalence where,
\begin{itemize}
\item $A/S$ is an abelian scheme
\item $i : \mathcal O_B \fleche \End(A) \otimes \ZZ_{(p)}$ is a $\ZZ_{(p)}$-algebra endomorphism.
\item $\lambda$ is a $\ZZ_{(p)}^\times$ equivalence class of $\mathcal O_B$-linear polarisation of order prime to $p$ which identifies Rosatti involution and $\star$ through $i$.
\item $\eta$ is a $K^p$-level structure on $A$ (see \cite{KotJams} section 5, or \cite{Lan}\footnote{Recall that such a level structure includes a (class of) isomorphism $\ZZ/p^N\ZZ \simeq \mu_{p^N}$ for some $N$, see \cite{Lan} Definition 1.3.6.1}).
\end{itemize}
where $i$ is subject to the determinant condition and the equivalence is by prime to $p$ quasi-isogeny (see also \cite{WedVieh} and for all details \cite{Lan}). As $K^p$ is \textit{sufficently small} $X_{K^p}$ is representible by a quasi-projective smooth scheme.

We choose $\nu$ a place of $E$ over $p$, and denote $\mathcal O_{E,\nu}$ the completion of $\mathcal O_E$ through $\nu$ and denote $X = X_{K^p,\nu}$ the base change to 
$\mathcal O_{E,\nu}$.

According to the decomposition of $B$, we can decompose $A = \prod_{i=1}^r A_i$ (and the other datums) as a product of abelian schemes (with additional structure associated to $B_i$). Moreover, we can further decompose the associated $p$-divisible group, writing
 $O_{B_i} \simeq \prod_{j=1}^{s_i} M_{n_i}(\mathcal O_{K_{i,j}})$, and using Morita-equivalence,
\[ A_i[p^\infty] = \prod_{j=1}^{s_i} \mathcal O_{K_{i,j}}^{n_i} \otimes _{O_{K_{i,j}}} A_i[\pi_j^\infty].\]
Moreover for a $(i,j)$ of type AL (i.e. $\pi_j$ splits in $F=F_i$), we can further decompose,
\[ A_i[\pi_j^\infty] = H_{i,j} \times H_{i,j}^D,\]
such that $\lambda$ is given by $(x,y) \mapsto (y,x)$ and $i$ preserve each factor. 

Denote $\omega$ the conormal sheaf of $A$, it is a locally free sheaf on $X$ which decompose as previously, and for all $(i,j)$ we get $\omega_{i,j} = \omega_{A_i[\pi_j^\infty]}$
locally free sheaf of rank $\dim A_i[\pi_j^\infty]$.
Let $P$ be the parabolic in $G_1$ fixing the cocharacter $\mu$ and $M$ the levi of $P$. $T$ can be seen as a torus in $M$ and fix a Borel $B$ of $M$. For $\kappa \in X^+(T)$ a dominant weight for this choice, there exists a locally free sheaf $\omega^\kappa$ on $X$. This sheaf can be described this way. 
Let \[ \mathcal T^\times = \Isom_{X,\mathcal O_B}((\Gamma\otimes_{\ZZ_p} \mathcal O_X)^\vee,\omega) \simeq \Isom_{X,\mathcal O_B}(\Gamma\otimes_{\ZZ_p} \mathcal O_X,Lie(A/Y)),\]
where $\Gamma$ is a $\mathcal O_B$-invariant $\mathcal O_{E,(p)}$-lattice in $V_1$ (where $V = V_0 \oplus V_1$ under the weight decomposition of $\mu$, see \cite{WedVieh} p10)  and $\pi : \mathcal T^\times \fleche X$, 
the space of trivialisations of $\omega$. This is a $M$-torsor.

\begin{defin}
Let $\kappa$ be a dominant algebraic character of $T$ and $\kappa^\vee$ its dual, i.e. $-w_0(\kappa)$ where $w_0$ is the longest element of the Weyl group. We see this characters as characters of B, extending them trivially on the unipotent.
The coherent automorphic sheaf $\omega^\kappa$ is the locally free sheaf over $X$ defined by,
\[ \omega^\kappa = \pi_* \mathcal O_{\mathcal T^\times}[\kappa^\vee].\]
\end{defin}


Let $X^{tor}$ be a toroïdal compactification\footnote{A priori the following definition depends on this choice, however by \cite{Lan} Lemma 7.1.1.3, this is independant of the choice of a \textit{toroïdal} compactification, and in most cases we don't even need to specify any compactification, by Koecher's principle, see \cite{LanKoecher} Theorem 2.3} of $X$ (see \cite{Lan}) and $D$ its boundary.

\begin{defin}
The space of (respectively cuspidal) modular (or coherent automorphic) forms of weight $\kappa$, and level $K^pG(\ZZ_p)$ is the space,
\[ H^0(X^{tor},\omega^\kappa), \quad (\text{respectively} \ H^0(X^{tor},\omega^\kappa(-D)).\]
\end{defin}

\begin{rema}
\label{remkappa}
The goal of this article is to deform $p$-adically the previous spaces of automorphic forms. Unfortunately, we can check that in some cases the duality $\kappa \mapsto \kappa' = \kappa^\vee$ does not extend to $p$-adic weights. This is the case for $U(2,1)_{E/\QQ}$ when $p$ is inert in $E$ where $T = \mathcal O_{E,p}^\times \times \mathcal O_{E,p}^1$. We can see an algebraic weight $(k_1 \geq k_2,k_3)$ as $(x,y) \mapsto \tau(x)^{k_1}\tau(y)^{k_2}\sigma\tau(x)^{k_3}$ and duality sends $(k_1 \geq k_2,k_3)$ on $(-k_2,-k_1,-k_3)$. This does not come from a natural algebraic map on $\mathcal W$.
This is harmless for us as we can work with normalisation of the weight given by $\kappa^\vee$ instead of $\kappa$, and thus we will embed weight $\kappa$ classical autormorphic forms into weight $\kappa^\vee$ overconvergent ones.
\end{rema}

\section{Local models and Jones Induction result}
\label{sect4}
To construct families of automorphic forms, we will first construct families of autormorphic sheaves, i.e. we will construct automorphic sheaves $\omega^{\kappa\dag}$ for $\kappa$ not only a dominant algebraic weight but a $p$-adic one, and theses sheaves will interpolate the coherent sheaves $\omega^\kappa$ (actually to be more precise the sheaves $\omega^{\kappa^\vee}$, see remark \ref{remkappa}).
This has been done previously in analogous settings (see \cite{AIP,AIS,PiFou,Brasca,Her3}), and all these works adapt geometrically constructions that were first developped in the case of compact at infinity groups (see \cite{Buz,Che1,urb}) using interpolations of algebraic representations by locally analytic ones. As our sheaves will be modeled on these construction, let us review the theory. It will be usefull in analysing classicity questions in section \ref{sect8}.

\subsection{Inductions}
Let us fix some notations. We will be interrested in representations of a $p$-adic groups attached to $\mu$. $\mu$ gives rise to a parabolic in $G$, and denote $M$ the Levi 
subgroup of this parabolic. The group $M_{\QQ_p}$ splits over the couples $(i,j)$ introduced before. As explained in the previous section, $(i,j)$ of type (C) are ordinary and thus 
have be treated in \cite{Brasca}, thus we focus on type (A). In this cases, $M_{(i,j)}$ is isomorphic to a Levi of the group $\Res_{F_i^+/\QQ_p}U(n_i)_{F_i/F_i^+}$. 

Denote $\mathcal T^+=\mathcal T^+_{(i,j)}$ the set of embeddings of $F_i^+$ into $\overline{\QQ_p}$ and $\mathcal T$ the corresponding set for $F_i$. $M_{(i,j)}$ is thus up to extending scalars isomorphic to some 
$P = \prod_{\tau \in \mathcal T^+} \GL_{p_\tau} \times \GL_{q_\tau}$ say over $K$ a $p$-adic field.  The integers $p_\tau,q_\tau$ are determined by $\mu$, the co-character associated to the Shimura Datum $(G,h)$, and verify that
\[ p_\tau + q_\tau = n_i, \forall \tau \in \mathcal T^+.\]
 For now on, we drop the index $(i,j)$ in the notations, thus set $n_i = n = h = p_\tau + q_\tau$. Still denote by $P=\prod_{\tau \in \mathcal T^+} \GL_{p_\tau} \times \GL_{q_\tau}$ an integral model over $\mathcal O = \mathcal O_K$.
Let $T$ be the maximal (diagonal) torus of $P$, $B$ the upper Borel, and for each $\kappa \in X^+(T)$, denote the (algebraic, non-normalized) induction,
\[ V_\kappa = \{ f : P \fleche \mathbb A^1 \text{ algebraic } | f(gb) = \kappa(b)f(g) \text{ for all } g,b \in P \times B\}.\]
This is a finite dimensional $K$-vector space endowed with an action of $P(K)$ by $(g .f)(z) = f(g^{-1}z)$.

The algebraic induction is a local model of the automorphic sheaves $\omega^\kappa$ in the sense that etale locally the later is isomorphic to the former.
We will now describe another representation that will \textit{interpolate} the previous ones and which will be local models of the coherent Banach sheaves constructed later in the paper.

Let $I = I_1$ be the Iwahori subgroup of $P$, i.e. $I = \Ker(P(\mathcal O) \fleche P(\mathcal O/p)/B(\mathcal O/p))$. Denote more generally $I_n$ the level-n Iwahori, i.e. 
elements that are upper triangular modulo $p^n$. We have a Iwahori decomposition $I = B(\mathcal O) \times N^0$, and we can identify $N^0$ with
\[ (p\mathcal O)^N \subset \mathbb A_{an}^N, \quad N = \sum_{\tau\in \mathcal T^+} \frac{p_\tau(p_\tau -1) + q_\tau(q_\tau-1)}{2}.\]
For any $\eps \geq 0$, we define $N_0^{\eps}$ as the subspace,
\[ \bigcup_{x \in(p\mathcal O)^N} B(x,p^{-\eps}) \subset \mathbb A_{an}^N\]
and for $L$ a $p$-adic field, denote $\mathcal F^{\eps-an}(N^0,L)$ the function that are restriction to $N^0$ of analytic functions on $N^0_\eps$. 
Now we can define the $\eps$-analytic induction. Let $\kappa \in \mathcal W(L)$ be $\eps$-anaytic,
\[ V_{\kappa,L}^{\eps-an} = \{ f : I \fleche L : f(gb) = \kappa(b)f(g) \forall g,b \in I\times B(\mathcal O), f_{N^0} \in \mathcal F^{\eps-an}(N^0,L)\}.\]
Denote $V_{\kappa,L}^{loc-an} = \bigcup_{\eps >0} V_{\kappa,L}^{\eps-an}$ and $V_{\kappa,L}^{an} = \bigcap_{\eps \geq 0} V_{\kappa,L}^{\eps-an}$.
This spaces won't be local models of our Banach-automorphic sheaves, but they will have the same finite slope eigenvalues. 

Choose an ordering, $\{ p_\sigma : \sigma \in \mathcal T\} = \{p_\tau,q_\tau : \tau \in \mathcal T^+\} = \{p_1 \leq \dots \leq p_{2f}\}$, and let $P^0_{(p_\sigma)} \subset \GL_{p_{2f}}$ be the standard parabolic with (ordered) blocs of size $(p_i - p_{i-1})_{i=1,\dots,2f}$, where $p_0 = 0$. Denote $I_{(p_\sigma)}$ the Iwahori subgroup of $P^0_{(p_\sigma)}(\mathcal O)$ with respect to the standard upper triangular Borel, and $N^0_{(p_\sigma)}$ the opposite unipotent in $P^0_{(p_\sigma)}(\mathcal O)$.
For every $\sigma \in \mathcal T$, every matrix $M \in P^0_{(p_\sigma)}(\mathcal O)$ can be written of the form,
\[ M = 
\left(
\begin{array}{cc}
 A_\sigma & B_\sigma  \\
 0 &  D_\sigma  
\end{array}
\right), \quad A_\sigma \in M_{(p_{2f}-p_\sigma)\times (p_{2f}-p_\sigma)}(\mathcal O), D_\sigma \in M_{p_\sigma \times p_\sigma}(\mathcal O).
\]
In particular, we get for each $\tau \in \mathcal T^+$ a map,
\[ 
\begin{array}{ccc}
 P_{(p_\sigma)} &\fleche& \GL_{p_\tau} \times \GL_{q_\tau},\\
 M  & \longmapsto & (D_\sigma,D_{\overline \sigma})
  \end{array}, \quad \text{for }\] where $\sigma,\overline\sigma$ are the two embeddings over $\tau$ such that $p_\sigma = p_\tau, p_{\overline\sigma} = q_\tau$
and $D_\sigma$ refers to the previous decomposition.

Denote again $P^0_{(p_\sigma)}$ the image of the diagonal embedding $P^0_{(p_\sigma)} \fleche \prod_\tau \GL_{p_\tau} \times \GL_{q_\tau}$ (it is injective, looking at a $\tau$ such that $p_\tau = p_{2f}$ or $q_\tau = p_{2f}$). Thus, we can see $I_{(p_\sigma)}$ as a subgroup of $P$, and consider, for every $\kappa \in W(L)$ which is $\eps$-analytic,
\[ V_{\kappa,L}^{0,\eps-an} = \{ f : I_{(p_\sigma)}B(\mathcal O) \fleche L : f(gb) = \kappa(b)f(g), f_{N^0_{(p_\sigma)}} \in \mathcal F^{\eps-an}(N^0_{(p_\sigma)},L)\}.\]
Everything makes sense as $N^0_{(p_\sigma)}$ can be seen as a subset of $N^0$ and we can define $\eps$-analytic functions on it (using balls in $\mathbb A_{an}^N \supset N^0$ centered on points of 
$N^0_{(p_\sigma)}$). It is slightly complicated, but now $V_{\kappa,L}^{0,\eps-an}$ will be local models of our forthcoming Banach-automorphic sheaves.

\subsection{$U_p$-operator}

Define for all $i \leq \frac{h}{2}$ an integer,
\[ d_i = 
\left(
\begin{array}{ccc}
  p^{-2}I_{\max(p_{2f} - (h-i),0)} & & \\
 &p^{-1}I_{\min(h-2i,p_{2f}-i)}  &   \\
 & &  I_{i}  \\
\end{array}
\right) \in P_{(p_\sigma)}^0(K).
\]
We sometimes see $d_i$ in $\GL_{p_\tau} \times \GL_{q_\tau}$ using the previous embedding. Denote for each $\sigma \in \mathcal T$, $a_\sigma = \max(p_\sigma - (h-i),0), b_\sigma =\max( \min(h-2i,p_\sigma - i),0)$ and $c_\sigma = \min(i,p_\sigma)$ (thus $a_\sigma + b_\sigma+c_\sigma =p_\sigma$). This is respectively the number of $p^{-2},p^{-1},1$ appearing in $D_\sigma$ in the previous decomposition for $d_i$.
We can define an operator $\delta_i$ on $V^{0,\eps-an}_{\kappa,L}$ by $\delta_if(j) = f(d_ind_i^{-1}b)$ where $j = nb$ is the Iwahori decomposition.
\begin{prop}
\label{propdeltarad}
Let $f \in V_{\kappa,L}^{0,\eps-an}$ that we see as a function in $\mathcal F^{\eps-an}(N^0_{(p_\tau)},L)$ of variable 
$(x_{k,l}^\tau,y_{m,n}^\tau)_{1 \leq l < k \leq p_\tau, 1 \leq n < m \leq q_\tau,\tau}$. Then,
\[ \delta_i : 
\begin{array}{ccc}
  \mathcal F^{\eps-an}(N^0_{(p_\tau)},L) & \fleche  & \mathcal F^{\eps-an}(N^0_{(p_\tau)},L)  \\
  f  &  \longmapsto &   ((x_{k,l}^\tau,y_{m,n}^\tau) \mapsto f(p^{v_{k,l}^\tau}x_{k,l}^\tau,p^{w_{m,n}^\tau}y_{m,n}^\tau))
\end{array}
\]
where, if we denote $\tau = \sigma\overline \sigma$ in $F$, with $p_\tau = p_\sigma$, \[
v_{k,l}^\tau = 
\left\{
\begin{array}{cc}
2 & \text{if } k > a_\sigma + b_\sigma \text{ and } l \leq a_\sigma\\
 1 & \text{if }  (b_\sigma + a_\sigma \geq k > a_\sigma  \text{ and } l \leq a_\sigma) \text{ or } (b_\sigma + a_\sigma \geq l > a_\sigma \text{ and } k > a_\sigma + b_\sigma) \\
 0 &  \text{otherwise}  
\end{array}
\right. \]\[ w_{m,n}^\tau = 
\left\{
\begin{array}{cc}
 2 & \text{if } m > a_{\overline\sigma} + b_{\overline\sigma} \text{ and } n \leq a_{\overline\sigma}\\
 1 & \text{if }  (b_{\overline\sigma} + a_{\overline\sigma} \geq m > a_{\overline\sigma}  \text{ and } n \leq a_{\overline\sigma}) \text{ or } (b_{\overline\sigma} + a_{\overline\sigma} \geq n > a_{\overline\sigma} \text{ and } m> a_{\overline\sigma} + b_{\overline\sigma}) \\
 0 &  \text{otherwise}  

\end{array}
\right.
\]
In particular, $\prod_i \delta$ is completely continuous.
\end{prop}

\begin{rema}
It is not a mistake that $f$ has "as much variables as entries in $N^0$" instead of $N^0_{(p_{\sigma})}$. The reason is that $f$ is seen as a function (even a locally analytic one) in a neighborhood of the image of $N^0_{(p_\sigma)}$ in the analytic space associated to $N^0$ (and not to $N^0_{(p_\sigma)}$). Indeed, such $f$ can't be defined on $N^0$ a priori, except if we know that it is $1$-analytic (as the neighborhood of $N^0_{(p_\sigma)}$ of radius $\frac{1}{p}$ in $\mathbb A^N = (N^0)^{an}$ is $N^0 = (p\mathcal O)^N$ itself.
\end{rema}

\begin{proof}
This is a direct calculation on matrices of $N^0$.
\end{proof}

\subsection{Jones's BGG and a fiberwise classicity result}

Let $P$ our previous algebraic group $T$ its torus and $B$ its upper Borel that defines $\Delta$ a set of positive roots. Then for every dominant weight $\kappa \in X^+(T)$,
Jones's \cite{Jones} proved the exacteness of the following sequence,
\begin{equation}\label{seqJones}
0 \fleche V_{\kappa,L} \fleche V_{\kappa,L}^{an} \overset{d}{\fleche} \bigoplus_{\alpha \in \Delta} V_{\alpha\bullet\kappa,L}^{an}\end{equation}
where $d$ is an explicite map (see for example \cite{AIP} for $GSp_{2g}$ ($P = \GL_g$) and \cite{Brasca} for a similar case to ours).
Then the following proposition is \cite{Brasca} proposition 6.5

\begin{prop}
Write $\kappa = (k_{\sigma,i}) \in X^+(T)$ according to the decomposition $P = \prod_{\tau\in \mathcal T^+} \GL_{p_\tau}\times\GL_{q_\tau} = \prod_{\sigma\in\mathcal T} \GL_{p_\sigma}$ a dominant weight. Set 
\[ \nu_{i}^\sigma =  \inf \{ k_{\sigma,i} - k_{\sigma,i+1} : i  < p_{\sigma} \}.\]
Then, \[ V_{\kappa,L}^{0,\eps-an, < \underline \nu} \subset V_{\kappa,L}.\]
(The same proposition is true with $V_{\kappa,L}^{\eps-an, < \underline \nu}$).
\end{prop}

\begin{proof}
The first thing to check is that if $f \in V_{\kappa,L}^{0,\eps-an}$ is of non-zero slope, then $f \in V_{\kappa,L}^{an}$ (if $\eps \leq 1, N^0_{\eps} = N^0_{(p_\sigma),\eps}$).
But as $\prod_i \delta_i$ is increasing the analytic radius, by proposition \ref{propdeltarad} we get the claim.
Now, we can use Jones's BGG result as in \cite{AIP} Proposition 2.5.1, or \cite{Brasca} section 6.1, and we get the result.
\end{proof}

\begin{rema}
The previous calculation is made completely explicitely for $G = (G)U(2,1)$ in \cite{Her3}.
\end{rema}

\section{Integral models}
\label{sect5}
\subsection{Isogeny Graphs}

\begin{defin}
Fix $h \in \NN^*$ and $n \in \NN^*$, and denote $\Gamma_n^h$ the subset of $M_{n\times h}(\CC)$ such that $M = (m_{i,j})_{1 \leq i \leq n,1 \leq j \leq h} \in M_{n\times h}(\CC)$ satisfies,
\begin{enumerate}
\item For all $(i,j)$, $m_{i,j}  \in \{0,1\}$,
\item For all $(i,j)$, if $m_{i,j} = 1$, then $m_{i-1,j} =1$ and $m_{i,j-1}=1$ (when defined).
\end{enumerate}
Let $\underline{\Gamma_n^h} = (\Gamma_n^h,v)$ be the graph whose points are $M \in \Gamma_n^h$, and there is an arrow from $M = (m_{i,j})$ to $M' = (m'_{i,j})$ if
\[ \{ (i,j) | m_{i,j} \neq m'_{i,j} \} = \{ (i_0,j_0)\} \quad \text{and} \quad m_{i_0,j_0} = 0, \quad m'_{i_0,j_0} = 1.\]

When $n=0$, define $\Gamma_0^h$ as $\{\star\}$, and the map $\pi_{0,1} : \Gamma_0^h \ni \star \mapsto (0,\dots,0) \in \Gamma_1^h$. When $n \geq 2$, we have a natural map,
\[ \pi_{n-1,n} : 
\begin{array}{ccc}
\underline{\Gamma_{n-1}^h} &\fleche& \underline{\Gamma_{n}^h}\\
M = (m_{i,j})_{1 \leq i \leq h,1 \leq j \leq n-1} & \longmapsto   &   (m'_{i,j}) \in M_{n}(\CC), m'_{i,j} = m_{i,j} \text{ if } j <n, 0 \text{ otherwise}.
\end{array}
\]
This map preserves vertices, it is an embedding of graphs.
\end{defin}

\begin{rema}
If $n=1$, the possible matrices are simply given by 
\[ M_i = (\underbrace{1,\dots,1}_{i \text{ times}},0,\dots,0).\]
They parametrizes the lattices appearing in a periodic lattice chain inside $\GL_h(\ZZ_p)$ as in \cite{RZ}.
\end{rema}

\subsection{Some integral models}
\label{sectNorm}

Let $p$ be a prime, and let $\mathcal D$ be an integral Shimura-PEL-datum as in section \ref{sect2}. 

Denote by $\mathcal P = \{ (i,v) | v \text{ place of } F_i^+\}$ the set of \textit{places} of $G$, where $F_i$ is the center of $B_i$. 
Fix $S_p \subset \{(i,j) | i \in \{1,\dots,r\}, j \in\{1,\dots,s_i\}\}$ a set of unramified places 
over $p$. With our assumptions on $\mathcal D$, for all $v =(i,j) \in S_p$, $v$ is unramified, and 
$B \otimes_{F^+} {F^+_v}$ is split and isomorphic to $M_n(F_v)$. 
Fix $S$ a finite set of places of $\mathcal P$ such that $S \cap S_p = \emptyset$, and $S$ contains all places such that $B$ doesn't split or is ramified.

Fix then a compact $K^{S,S_p}$ outside $SS_p$ such that, $K_v$ is maximal hyperspecial for all $v \notin S \cup S_p$.

For all $(i,j) \in S_p$ we can associate an integer $h_{i,j} = \Ht_{\mathcal O_{i,j}} A[\pi_j]$ in case (AU) and $h_{i,j}= \Ht_{\mathcal O_{j,j}} A[\pi^+_j]$ in case (AL). These integers are defined for example by looking at the characteristic 0 moduli space as explained in section 3 (or could be read directly on $G$, and even defined by the integral moduli space of Kottwitz if $G$ is unramified at $p$). 
We set $\underline\Gamma_n = \prod_{(i,j) \in S_p} \underline \Gamma_n^{h_{i,j}}$.

Fix once and for all a compact subgroup $K_S \subset G(F_S)$ and for all $v \in S_p$, consider $K_v^{Sph}$ the maximal hyperspecial compact open subgroup.
We will study some covering of 
\[ X^{Sph} = X_{G,K_0}, \quad K_0 = K^{S,S_p}K_S\prod_{v \in S_p} K^{Sph}_v.\]
The Shimura variety associated to $X^{Sph}$ has a good integral model $\mathfrak X^{Sph}$ over $\Spec(\mathcal O_K)$, for $K/\QQ_p$ a well chosen finite extension (\cite{Lan} if $K_v$ is hyperspecial for all $v | p$, and \cite{LanSigma} if $p$ is unramified in $\mathcal D$ for example by normalisation of the hyperspecial level. In general, we fix any integral model $\mathfrak X^{Sph}$ given by \cite{LanSigma}.
If  $K_v$ is hyperspecial for all $v | p$, $\mathfrak X^{Sph}$ is smooth).

We will define our base space, and its integral model following \cite{LanSigma}. Let for all $v \in S_p$, $I_v$ be a Iwahori subgroup at $p$ of $G(F_v)$.
Define first its generic fiber,
\[ X = X_{G,K}, K = K^{S,S_p}K_S \prod_{v \in S_p} I_v.\]
This space, over some extension $K$ of $\QQ_p$, classifies quintuples $(A,\iota,\lambda,\eta_S,H_\cdot)$ modulo isomorphisms where $(A,\iota,\lambda,\eta_S)$ is a point of $X^{Sph}$, and $H_\cdot$ is a full flag of $A[v]$, for $v \in S_p$.
Explicitely, for every $(i,j)$ as before, $H_\cdot$ induces,
\begin{enumerate}
\item In case $(AL)$, a filtration 
\[ 0 \subset H_1 \subset \dots \subset H_r = A_i[\pi_j^+],\]
by finite flat $\mathcal O_{i,j}$-group schemes such that $H_k$ is of rank $p^k$.
\item In case $(AU)$, a filtration,
\[ 0 \subset H_1 \subset \dots \subset H_r = A[\pi_j],\]
by finite flat $\mathcal O_{i,j}$-groups schemes such that $H_k$ is of rank $p^k$ and $H_i^\perp = H_{r-i}^{(\sigma)}$.
\end{enumerate}

This Shimura variety with Iwahori level at $S_p$ has a natural integral model over $\Spec(\mathcal O_K)$. When all the prime $v | p$ satisfies that $K_v$ is parahoric (this is only a condition outside $S_p$ here), then this is defined by the lattice chain introduced in \cite{RZ}. See for example \cite{Lan}. In general, this can be seen as explained in \cite{LanSigma}, example 2.4 and 13.12.
The abelian scheme $A$ and the subgroups $H_k^{i,j}$ gives rise to isogenies (precisely, we need to use Zarhin's trick, see remark \ref{remZarhin}),
\[ A \fleche A_k^{i,j} = A/H_k^{i,j}.\]
In particular we get a map,
\[ X \fleche \prod_{\gamma \in \Gamma_1} \mathfrak X^{Sph},\]
sending $(A,\iota,\lambda,\eta,H_\cdot)$ to $(A_k^{i,j},\iota,\eta,\lambda)$ (see remark \ref{remZarhin}). 
Then the integral model $\mathfrak X$ is defined as the normalisation of $\mathfrak X^{Sph}$ in $X$. This is flat over $\Spec(\mathcal O_K)$. 

The same thing applies to compatible choices of toroïdal compactification\footnote{In all the this text, we always assume the rational cone decompositions to be smooth and projective without further comment}, and we get spaces, flat, proper over $\Spec(\mathcal O_K)$ (see \cite{LanSigma} Lemma 7.9),
\[ \mathfrak X^{tor} \quad \text{and} \quad \mathfrak X^{Sph,tor}.\]

\begin{rema}
In the following, we will be interested mainly by $A$ (as opposed to the collection of all the $A_\gamma$) and the subgroups $H_{i,j}^k[p^\ell]$. Thanksfully, there is a "universal semi-abelian scheme" (more precisely, a degenerating family) on $\mathfrak X^{tor}$ and its covers extending $A$ on $X$. If $p$ is unramified in the PEL datum and we are at hyperspecial level this is \cite{Lan} Theorem 6.4.1, in general this is \cite{LanSigma} Theorem 11.2.

But we will need slightly more, as for a semi-abelian scheme $G$, $G[p^n]$ need not to be finite flat. Fortunately, we can find an etale covering $\mathfrak U$ of $\mathfrak X^{Sph,tor}$ such that $G$ is approximated on each open of this covering by a Mumford 1-motive $M$, i.e. $G[p^n] = M[p^n]$ (see \cite{Stroh} section 2.3 (more precisely Proposition 2.3.3.1) and \cite{LanSigma} 
Theorem 11.2). This etale covering is an isomorphism on the boundary (see \cite{Stroh} section 2.4). In particular, there is a semi-abelian scheme of constant rank 
$\widetilde G$ such that $\widetilde G[p^n] \subset G[p^n]$ is finite flat, and such that $\omega_{\widetilde G[p^n]} = \omega_{G[p^n]}$. We can thus by pullback find also an etale covering of $\mathfrak X^{tor}$ on which we have the finite flat group scheme $\widetilde G[p^n]$.
Thus, the ($\mu$-ordinary) Hasse invariant or the degree function extends on this covering of $\mathfrak X^{Sph,tor}$, but we can descend them : see in subsection \ref{subsectBord}.
\end{rema}

We have similarly for any $n$, a Shimura variety with Iwahori level $p^n$, $X^{tor}_0(p^n)$, over $\Spec(K)$, classifying, outside the boundary, $(A,\iota,\lambda,\eta_S,H_\cdot)$, with $H_\cdot$ a full flag of $A[p^n]$. More precisely, we have for all $(i,j)$,
\begin{enumerate}
\item In case $(AL)$, a filtration 
\[ 0 \subset H_1 \subset \dots \subset H_r = A_i[\pi_j^+],\]
by finite flat $\mathcal O_{i,j}$-group schemes such that $H_k$ is of $\mathcal O_{i,j}$-rank $p^{nk}$ with cyclic graded pieces.
\item In case $(AU)$, a filtration,
\[ 0 \subset H_1 \subset \dots \subset H_r = A[\pi_j],\]
by finite flat $\mathcal O_{i,j}$-groups schemes such that $H_k$ is of $\mathcal O_{i,j}$-rank $p^{nk}$ with cyclic graded pieces such that $H_i^\perp = H_{r-i}^{(\sigma)}$.
\end{enumerate}
Once again, by \cite{LanSigma} (here we are in caracteristics zero, so this is easier) there is a natural map\footnote{Cone decomposition must be chosen appropriately, but we suppose so, without further comment, as it is always possible to refine the choices in order to get the compatibility} (again, see remark \ref{remZarhin}),
\[ X_0(p^n)^{tor} \fleche \prod_{\gamma \in \Gamma_n} \mathfrak X^{Sph,tor},\]
sending $(A,\iota,\lambda,\eta_S,H_\cdot)$ to $(A/(H_{i,j}^k[p^\ell]),\iota,\lambda,\eta_S,)$ away from the boundary.

There is moreover a map 
\[X^{tor}_0(p^n) \overset{\pi_{n,n-1}}{\fleche} X^{tor}_0(p^{n-1}),\]
given by sending the flag $(H_{i,j}^k[p^\ell] \subset A_i[\pi_j^\ell])_{\ell \leq n}$ to the flag $(H^k_{(i,j)}[p^{\ell}] \subset A_i[\pi_j^{\ell}])_{\ell \leq n-1}$. In particular, the diagram,
\begin{center}
\begin{tikzpicture}[description/.style={fill=white,inner sep=2pt}] 
\matrix (m) [matrix of math nodes, row sep=3em, column sep=2.5em, text height=1.5ex, text depth=0.25ex] at (0,0)
{ 
X_0(p^n)^{tor} & &\prod_{\gamma \in \Gamma_n} \mathfrak X^{Sph,tor} \\
X_0(p^{n-1})^{tor} & &\prod_{\gamma \in \Gamma_{n-1}} \mathfrak X^{Sph,tor}\\
};
\path[->,font=\scriptsize] 
(m-1-1) edge node[auto] {$$} (m-1-3)
(m-2-1) edge node[auto] {$$} (m-2-3)
(m-1-1) edge node[auto] {$\pi_{n,n-1}$} (m-2-1)
(m-1-3) edge node[auto] {$\Gamma_{n,n-1}$} (m-2-3)
;
\end{tikzpicture}
\end{center}
is commutative.

\begin{defin}
Define $\mathfrak X_0(p^n)^{tor}$ to be the normalisation of $\prod_{\Gamma_n} \mathfrak X^{Sph,tor}$ in $X_0(p^n)^{tor}$. It is proper and flat over $\Spec(\mathcal O_K)$.
By normalisation, the map $\pi_{n,n-1}$ extends as a map,
\[ \pi_{n,n-1} : \mathfrak X_0(p^n)^{tor} \fleche \mathfrak X_0(p^{n-1})^{tor}.\]
\end{defin}

In particular (see also \cite{FC} Chap I Prop. 2.7), over $\mathfrak X_0(p^n)^{tor}$ we have by pullback natural isogeny graphs,
\[ (A_\gamma)_{\gamma \in \Gamma_n},\]
such that the Kernel of $A_i[\pi_j^\infty] \fleche A_{k,m}^{i,j}$, is a finite flat, at least away from the boundary, $\mathcal O_{i,j}$-subgroup of $A_i[\pi_j^n]$ of $\mathcal O_{i,j}$-rank $p^{km}$.
We denote it by $H^{i,j}_{k,m}$, or, if $(i,j)$ is understood, $H_k[p^m]$. This makes sense as $H_{k,m} = H_{k,n}[p^m]$. 

\begin{rema}
\label{remZarhin}
Actually the construction is slightly more evolved as what than been said, as the abelian varieties $A_\gamma = A/H_{i,j}^k[p^\ell]$ appearing in the isogeny graph might not be principally polarized, thus need not to give a map $X_0(p^n) \fleche X^{Sph}$. But as explained in \cite{LanSigma} Proposition 4.12 and Proposition 6.1, we have a map to an auxiliary moduli problem where $A/H_{i,j}^k[p^\ell]$ is modified to be principally polarized by Zarhin's trick, extends to the integral model $\mathfrak X_0(p^n)$ (all this works on the compactifications), and we can then deduce the extension of $A/H_{i,j}^k[p^\ell]$ itself.
\end{rema}


\subsection{Results on the canonical filtration and the Hodge-Tate map}
\begin{theor}
\label{thrfiltcan}
Let $L$ be a valued extension of $\QQ_p$, and $G$ be a truncated level $n$ $p$-divisible group over $\Spec(\mathcal O_L)$ with action of $\mathcal O$ and signature $(p_\tau,q_\tau)$.
\begin{enumerate}
\item \label{can1}Then there exist at most one  sub-$\mathcal O$-module $H_\tau$ of height $np_\tau$ such that,
\[ \deg H_\tau > \sum_{\tau'} \min(np_\tau',np_\tau) - \frac{1}{2}.\]
We call it the canonical subgroup of height $p_\tau$ if it exists.
\item Moreover, if two sub-$\mathcal O$-modules $H_\tau, H_{\tau'}$ of respective height $np_\tau, np_{\tau'}$ as before exists, then,
\[ p_\tau \leq p_{\tau'} \quad \text{if and only if} \quad H_\tau \subset H_{\tau'}.\]
\item If moreover $G$ is polarized, then $H_\tau$ is polarized, i.e. 
\[ H_\tau^\perp := (G/H_\tau)^D \hookrightarrow G^D,\]
is identified with the canonical subgroup of height $q_\tau$ of $G^D$.
\item \label{can4} The group $H_\tau$ verifying the previous hypothesis is a step of the Harder-Narasihman filtration of $G$, it also coincide with the kernel of the Hodge-Tate map,
\[ \alpha_{G,\tau,n-\eps} : G(\mathcal O_L) \fleche \omega_{G^D,\tau,n-\eps},\]
where $\eps = \deg_\tau(G/H_\tau)$.
\item\label{can5} \label{HTbound}  Suppose that $H_\tau$ as in \ref{can1}. exists. The cokernel of the  Hodge-Tate map,
\[ \alpha_{G,\tau}\otimes 1:  G(\mathcal O_L) \otimes \mathcal O_L \fleche \omega_{G^D,\tau},\]
is of degree $p^{\frac{\Deg_\tau(G[p]/H_\tau[p])}{p^f-1}}$. In particular, write $\eps_{\tau'} = n\min(p_{\tau'},p_\tau) - \deg_{\tau'} H_\tau$, then the cokernel of the Hodge-Tate map 
is killed by $p^{\frac{K_\tau(p_\bullet) + \Sum_\tau(\eps_\bullet)}{p^f-1}}$, where
\[ K_\tau( p_\bullet) =  \sum_{i=1}^f p^{f-i} \max(p_{\sigma^i\tau}-p_\tau,0) \quad \text{and} \quad \Sum_\tau(\eps_\bullet) = \sum_{i=1}^f p^{f-i} \eps_{\sigma^i\tau}.\]
\end{enumerate}
\end{theor}

\begin{proof}
The first three assertions are Bijakowski's result, \cite{Bijmu} Proposition 1.24,1.25, 1.30 (see for example \cite{Her2} Proposition A.2 for something written for the $p^n$-torsion).
Assertion \ref{can4}. is proposition 7.8 and 7.7 of \cite{Her2} (appliying 7.8 we get a step $H'_\tau$ and by 7.7 $H_\tau$ and $H_\tau'$ coincide with the Kernel of the Hodge-Tate map).
It it sufficient to prove \ref{can5}. for $n=1$. Remark that our hypothesis for $H_\tau \subset G$ implies the same for $H_\tau[p] \subset G[p]$.
Indeed denote $\deg_{\tau'} H_\tau = n\min(p_{\tau'},p_\tau) - \eps_{\tau'}$, and write the sequence,
\[ 0 \fleche H_\tau[p] \fleche H_\tau \fleche pH_\tau \fleche 0\]
which is exact in generic fibre, where $pH_\tau$ is the schematic adherence of $pH_\tau(\mathcal O_C)$ in $H_\tau$. Then $pH_\tau \subset G[p^{n-1}]$, and we have,
\[ \deg_{\tau'} H_\tau \leq \deg_{\tau'} H_\tau[p] + \deg_{\tau'} pH_\tau,\]
and because $pH_\tau$ is of height $(n-1)p_\tau$ and inside $G[p^{n-1}]$, $\deg_{\tau'}pH_\tau \leq (n-1)\min(p_\tau,p_{\tau'})$. Thus, 
\begin{equation}\label{eqdeg} \deg_{\tau'} H_\tau[p] \geq \min(p_{\tau'},p_\tau) - {\eps_\tau'}.\end{equation}

Then denote $E = G[p]/H_\tau[p]$. The hypothesis on the degree of $H_\tau$, and thus of $H_\tau[p]$ implies
\[ \omega_{H_\tau^D,\tau,\eps} = 0\]
for all $\eps < 1 - \deg_\tau(H_\tau^D)$, in particular, $\eps < 1/2$. Using the same devissage of $E$ as in \cite{Her2}, proof of theorem 6.1 implies that
 \[\deg \Coker(\alpha_{E,\tau,\eps} \otimes 1) = \deg \Coker(\alpha_{G,\tau,\eps} \otimes 1) = \frac{\Deg_\tau(E)}{p^f-1}.\]
Using the properties of various $\deg_{\tau'}$, and equation (\ref{eqdeg}) we get the result.
\end{proof}

\begin{rema}
\begin{enumerate}
\item The principal difference of the previous theorem with \cite{Her2} is that we don't a priori have the existence of such groups $H_\tau$. In \cite{Her2}, up to taking $p$ big enough to relate the ($\mu$-ordinary) Hasse invariant to the Hodge-Tate map, we have a condition for the existence in terms of the Hasse invariant. In this article, we assure the existence by increasing the level at $p$ in the integral model.
\item The bound given in \ref{HTbound} is interesting in general only when $p$ is big enough compared to $(p_\tau)$. If $p$ is small and $(p_\tau)$ is too big, then it is more interresting to use the bound given by Fargues (\cite{Far2}) which states that (in full generality) the cokernel of the Hodge-Tate map is killed by $p^{\frac{1}{p-1}}$. Note this is because the degree involve taking some determinant.
\end{enumerate}
\end{rema}

\subsection{Degree function, $\mu$-ordinary locus and Hasse invariants.}


As usual, fix a couple $(i,j) \in S_p$, and suppress it from the notation. 
For each $\tau \in \mathcal T$ is associated an integer $p_\tau$, and thus a subgroup of height $np_\tau$ over $\mathfrak X_0(p^n)$, $H_{\tau} \subset G = A[\pi_j^\infty]$, which is finite flat and killed by $p^n$.
We can thus, following \cite{Bijmu}, define for each $\tau$ a real-valued function,
\[ \deg(H_{\tau}) : 
\begin{array}{ccc}
\mathcal X_0(p^n)  & \fleche   & [0,n\sum_{\tau'} \min(p_\tau,p_{\tau'})]  \\
 (A,i,\lambda,\eta,H_\cdot) & \longmapsto  & \deg(H_{\tau})  
\end{array}
\]
and another one, $\prod_\tau \deg(H_{\tau})$. Then we have the following result of Bijakowski \cite{Bijmu} Proposition 1.34,
\begin{prop}
\label{Pro58}
The locus where the previous function is maximal in $\mathcal X_0(p^n)$, i.e.
\[\prod_{j=0}^{f-1} \deg(H_{{\sigma^{j}\tau}})^{-1}(\{n\sum_{\tau'} \min(p_\tau,p_{\tau'})\}\times \dots \times \{n\sum_{\tau'} \min(p_{\sigma^{f-1}}\tau,p_{\tau'})\}),\] 
is included in $\mathcal X_0(p^n)^{full-\mu-ord}$, the $\mu$-ordinary locus of $\mathcal X_0(p^n)$ .
\end{prop}

\begin{rema}
To be precise, as we have fixed the prime $(i,j) \in S_p$, the $\mu$-ordinary locus above, and in the rest of the text (until the conclusion at the end of section \ref{sect9}) if not stated otherwise, is with respect to the prime $(i,j)$.
\end{rema}

\begin{defin}
\label{def59}
On $\mathfrak X^{Sph}$ we can define a $\mu$-ordinary Hasse invariant $\Ha$ (cf. \cite{Her1}, see also \cite{GN,KW}) which is a section of the sheaf 
$\det\omega_G^{\otimes(p^f-1)} \pmod{p}$. This defines a function
\[ v(\Ha) : \mathcal X^{Sph} \fleche [0,1],\]
which sends a $\mathcal O_K$-point to the (truncated by 1) valuation of the $\mu$-ordinary Hasse invariant of the reduction of the corresponding point of $\mathfrak X^{Sph}$.
In particular we can define by pullback an analogous function on $\mathcal X_0(p^n)$, and define
\[ \mathcal X_0(p^n)^{\mu-full}(v) = v(\Ha)^{-1}([0,v]).\]
\end{defin}

\begin{rema}
\begin{enumerate}
\item In the previous definition, the valuation is normalized by $v(p) = 1$, and $\mathcal X_0(p^n)^{\mu-full}(0)$ is the $\mu$-ordinary locus of $\mathcal X_0(p^n)$.
\item Actually by construction we have many maps from $\mathcal X_0(p^n)$ to $\mathcal X^{Sph}$ (and as much for their integral model), namely one for each $\gamma \in \Gamma_n$. The one we consider above is the canonical one corresponding to the zero-matrix $\gamma$ (which sends $A$ on $A$).
\end{enumerate}

\end{rema}

\begin{defin}
\label{defmucan}
Define $\mathcal X_0(p^n)(v)$ as the (union of) connected components of $\mathcal X_0(p^n)^{\mu-full}(v))$ which contains a point of maximal degree for the previous function (equivalently, the components where the subfiltration of $H_\cdot$ of height $np_\tau$ coincide where the canonical filtration in sense of theorem \ref{thrfiltcan}).
We will call $\mathcal X^{}_0(p^n)(0) =: \mathcal X^{can-\mu-ord}_0(p^n)$ the $\mu$-ordinary-canonical locus of $\mathcal X_0(p^n)$. It is an open and closed subset of $\mathcal X_0(p^n)^{\mu-full}(0)$ and coincide with the locus of maximal degree of proposition \ref{Pro58}.
\end{defin}

\begin{rema}
$\mathcal X^{}_0(p^n)(v)$ is the analogue of the strict neighborhoods of the ordinary-multiplicative part of the modular curves of level $\Gamma_0(p)$.
\end{rema}

\begin{defin}
For $(\eps_\tau)_\tau$, define the rigid analytic open,
\[ \mathcal X_0(p^n)((\eps_\tau)_\tau) = \prod_{j=0}^{f-1} \deg(H_{{\sigma^{j}\tau}})^{-1}(\prod_{j=0}^{f-1}[n\sum_{\tau'} \min(p_{\sigma^j\tau},p_{\tau'})-\eps_{\sigma^j\tau},n\sum_{\tau'} \min(p_{\sigma^j\tau},p_{\tau'})]).\]
This is a strict neighborhood of the $\mu$-ordinary-canonical locus $\mathcal X_0(p^n)(0)$ in $\mathcal X_0(p^n)$.
\end{defin}

\begin{rema}
The map $\pi_{n,n-1}$ sends $\mathcal X_0(p^n)(\underline \eps)$ into $\mathcal X_0(p^{n-1})(\underline \eps)$.
Indeed, if $\deg H_\tau > n\sum_{\tau'} \min(p_\tau,p_{\tau'}) - \eps_\tau$, then because of the generic exact sequence,
\[ 0 \fleche H_\tau[p^{n-1}] \fleche H_\tau \fleche K \fleche 0,\]
and the fact that $K$ is killed by $p$, thus $\deg K \leq \sum \min(p_\tau,p_{\tau'})$ 
we have that $\deg H_\tau[p^{n-1}] \geq (n-1)\sum_{\tau'} \min(p_\tau,p_{\tau'}) - \eps$.

\end{rema}

\subsection{Extension to the boundary}
\label{subsectBord}

We want to extend the previous opens to all $\mathcal X_0(p^n)^{tor}$, thus we will need to extends the functions $\deg$ and $^\mu\Ha$. The function $^\mu\Ha$ can be extended to all $\mathfrak X_0(p^n)^{tor}$ (as a section of some $\det (\omega_G)^{\otimes N} \otimes \mathcal (\mathcal O_{\mathfrak X_0(p^n)^{tor}}/p)$) by \cite{LanIMRN} Theorem 8.7. For the functions $\deg$, we can also extend it. The group $H_{p_\tau}$ is the Kernel of an isogeny of semi-abelian schemes
\[ \pi : A \fleche A_\gamma,\]
on $\mathfrak X_0(p^n)^{tor}$. Thus, looking at the corresponding map on conormal sheaves we get
\[ \pi^* : \omega_{A_\gamma} \fleche \omega_A,\]
and taking determinants gives $\det\pi^* \in H^0(\mathfrak X_0(p^n)^{tor},\det\omega_A\otimes \det\omega_{A_\gamma}^{-1})$. Over $\mathfrak X_0(p^n)$, the valuation at every point of $\det\pi^*$, which can be seen as an element of $\RR_+$, coincides with the degree of $H_{p_\tau}$. Thus, we have extended the degree map to,
\[ \deg(H_{\tau}) : \mathcal X_1(p^n)^{tor} \fleche \RR_+.\]
To check that this map is actually bounded by $n\sum_{\tau'}\min(p_\tau,p_\tau')$ as on the open Shimura variety $\mathcal X_0(p^n)^{tor}$, we can do the following.
Let $x \in \mathcal X_0(p^n)^{tor}(K)$, and let $\widetilde G/\mathcal O_K$ be the constant toric rank semi-abelian scheme such that $A_x$ is a quotient by some etale sheaf $Y$ of $\widetilde G$ by Mumford's construction. Then by \cite{FC90} Corollary 3.5.11, we have an exact sequence, and taking schematic adherence $H_n$ of $\widetilde G[n] \otimes K$ in $A_x[n]$, we have that $H_n$ is isomorphic to $\widetilde G[n]$ and whose quotient in $A_x[n]$ is etale. Decompose accordingly $A_{\gamma,x}$ together with the isogeny $\pi$ (see for example \cite{FC90} Corollary III.7.2), and decompose $\pi_H$ as $\widetilde \pi_H$ along $\widetilde G$. Then the degree of $\pi_H$ is the same as $\widetilde \pi_H$ as its quotient is etale. But $\ker \widetilde\pi_H$ (which is now finite flat) is of signature smaller than $(n\min(p_\tau,p_\tau'))_{\tau'}$, thus the assertion on its degree.

In particular we can define $\mathcal X_0(p^n)^{tor}((\eps_\tau)_\tau)$ and  $\mathcal X_0(p^n)^{tor}(v)$ as before.

\subsection{Two collections of strict neighborhoods} 

The previous opens $\mathcal X_0(p^n)^{tor}((\eps_\tau)_\tau)$ and  $\mathcal X_0(p^n)^{tor}(v)$ both define stricts neighborhoods of the $\mu$-ordinary-canonical locus 
$\mathcal X_0(p^n)^{tor}(0)$. Thus we get the following proposition,

\begin{prop}
For all $v >0$ there exists $(\eps_\tau)_\tau >0$ such that,
\[ \mathcal X_0(p^n)^{tor}((\eps_\tau)_\tau) \subset \mathcal X_0(p^n)^{tor}(v),\]
and for all $(\eps_\tau)_\tau >0$ there exists $v >0$ such that,
\[\mathcal X_0(p^n)^{tor}(v) \subset \mathcal X_0(p^n)^{tor}((\eps_\tau)_\tau).\]
\end{prop}

\begin{proof}
Fix $\mathcal V$ a strict neighborhood of $\mathcal X^{can-\mu-\ord,tor}_0(p^n) = \mathcal X_0(p^n)^{tor}(0)$. As $(\mathcal V,\mathcal X_0(p^n)^{tor}\backslash \mathcal X_0(p^n)^{tor}(0))$ 
is an admissible covering, $\mathcal V$ contains $\mathcal X_0(p^n)^{tor}(v))$ for some $v >0$. The same applies for $\mathcal X_0(p^n)^{tor}((\eps_\tau)_\tau)$.
\end{proof}

\begin{defin}
We say that $\eps = (\eps_\tau)$ and $v$ are $n$-compatible, and we say that they satisfies $(\eps,v,n)$, if,
$$
\mathcal X_0(p^n)^{tor}(v) \subset \mathcal X_0(p^n)^{tor}(\eps).
$$
\end{defin}

Let us explain quickly why we chose this two collections of strict neighborhoods. Classically, we use the stricts neighborhoods $\mathcal X(v)$ given by the Hasse invariant to construct eigenvarieties because this is the classical definition of Katz, and as the Hasse invariant is a section of an ample line bundle on the minimal compactification, we get that the ordinary (or $\mu$-ordinary) locus and its strict neighborhoods in the minimal compactification are affinoids. This is a crucial part of the construction described in \cite{AIP}. In many case (see \cite{Brasca} or \cite{Her3} in the Picard case, and using \cite{Her2} in all PEL unramified case when $p$ is big enough), we manage to construct on the opens $\mathcal X(v)$
a canonical filtration and control the degree of the subgroups of this filtration explicitely in terms of $v$. Thus the choice of the strict neighborhoods $\mathcal X(v)$ is enough to do all the constructions in these cases. But the classicity results as in \cite{Buz,Kas,PiDuke,PSdep,BPS} and in the $\mu$-ordinary case \cite{Bijmu} relies on the stricts neighborhoods in terms of the degree. So in the unramified PEL case when $p$ is not big enough, it is not clear a priori how to relate the degrees in terms of the Hasse invariant. Nevertheless, the previous proposition will allow us to use the best of both worlds.

We will need to understand the behavior of the strict neighborhoods along etale maps.
\begin{lemma}
\label{lemmastricetale}
Let $\pi : X \fleche Y$ a rigid-etale map. Let $U \subset X$ be a open subset and $V = \pi(U)$ the corresponding open in $Y$.
Let $U_w \subset X$ be a strict neighborhood of $U$, then $\pi(U_w)$ is a strict neighborhood of $V$.
\end{lemma}

\section{Canonical filtration, Hodge-Tate map and overconvergent modular forms}
\label{sect6}

As before, fix a couple $(i,j) \in S_p$ that will be understood until the rest of this section. 
In the previous section we defined a rigid open denoted $\mathcal X_0(p^n)^{tor}(v)$.

\begin{defin}
We denote by $\mathfrak X_0(p^n)^{tor,\mu-full}(v)$ the normalisation of the greatest open in the blow-up of $\mathfrak X_0(p^n)^{tor}$ by the ideal $I = (p^v,^{\mu}\Ha)$ where this ideal is generated by $^{\mu}\Ha$. As this scheme is normal; it has the same connected components than its rigid fiber, and we thus denote $\mathfrak X_0(p^n)^{tor}(v)$ the one whose generic fiber is $\mathcal X_0(p^n)(v)$.
\end{defin}

From now on, fix $\underline \eps < \underline{\frac{1}{2}}$. 
Recall that over $\mathfrak X_0(p^n)^{tor}$ we have subgroups $H_{p_\tau}^m$ for $m \leq n$ (which are finite flat on $\mathfrak X_0(p^n)$ of $\mathcal O$-rank $mp_\tau$), but a priori only quasi-finite flat over the boundary.
 
\begin{prop}
If $\eps < \frac{1}{2}$, for every $v > 0$ such that
\[ \mathcal X_0(p^n)^{tor}(v) \subset \mathcal X_0(p^n)^{tor}(\eps),\]
the groups $H_{\tau}^m$ are finite flat over $\mathfrak X_0(p^n)^{tor}(v)$.\end{prop}

\begin{proof}
Over $\mathfrak X_0(p^n)^{tor}$ there is a isogeny
\[ A \fleche A^{i,j}_\gamma,\]
of semi-abelian schemes whose Kernel is the group $H_{\tau}^m$ (a priori only quasi-finite flat), and this group is finite flat over $\mathfrak X_0(p^n)$.
Moreover, by a classical construction, there is an etale covering $\mathfrak U^{tor}$ of $\mathfrak X_0(p^n)^{tor}$ over which the semi abelian schemes $A$ and $A_\gamma$ can be approximated by a 1-motive of Mumford $M$ and $M_\gamma$ 
(concretely these $\mathfrak U^{tor}$ exists for $\mathfrak X^{Sph}$ by construction, see e.g. \cite{Stroh} section 3, and we can moreover assure that $M[p^n]$ and 
$M_\gamma[p^n]$ are isomorphic to $A[p^n]$ and $A_\gamma[p^n]$, by the arguments of \cite{Stroh} section 2.3, and take the pull-back via 
$\mathfrak X_0(p^n)^{tor} \fleche \prod_{\gamma} \mathfrak X^{Sph,tor}$).
We only need to check that $H_{p_\tau}^m$ is finite flat over $\mathfrak U^{tor}(v) := \mathfrak U^{tor} \times_{\mathfrak X_0(p^n)^{tor}} \mathfrak X_0(p^n)^{tor}(v)$. But there is an isogeny over $\mathfrak U^{tor}(v)$
\[ \pi : A \fleche A_\gamma,\]
such that $\Ker \pi$ is $H_{p_\tau}^m$. Thus for every $\mathcal O_K$-point of $\mathfrak U^{tor}(v)$, $H_{p_\tau}^m$ is of high degree (in the sense of theorem \ref{thrfiltcan}).
But over $\mathfrak U^{tor}(v)$, $A$ and $A_\gamma$ are associated to Mumford 1-motives $M$ and $M_\gamma$ by Mumford construction. Thus there exists semi abelian schemes $G$ and $G_\gamma$, of constant toric ranks, in the datum of $M$ and $M_\gamma$, such that the isogeny $\pi$ reduces to
\[ \pi' : G \fleche G_\gamma.\]
Call $H' = \ker \pi'$. It is finite flat as $G$ and $G_\gamma$ have constant toric ranks. As $\omega_G \simeq \omega_A$ and $\omega_{G_\gamma} \simeq \omega_{A_\gamma}$, the degree of $H'$ is the same as the one of 
$H_{\tau}^m = \Ker \pi$. Thus, away from the boundary, over $\mathfrak U(v) := \mathfrak U^{tor}(v) \times_{\mathfrak X_0(p^n)^{tor}(v)} \mathfrak X_0(p^n)(v)$,
by unicity in $A[p^n]$ of Theorem \ref{thrfiltcan}, we have $H' = H_{\tau}^m$ (it is true for every $O_K$-point, thus on $\mathfrak U(v)$ by normality). In particular, the semi-abelian schemes
\[ A/H_{\tau}^n \quad \text{and} \quad A/H',\]
are isomorphic over $\mathfrak U(v)$. But by \cite{FC90} Prop. I.2.7, this implies by normality of $\mathfrak X_0(p^n)^{tor}(v)$, and thus of $\mathfrak U^{tor}(v)$,
that they are isomorphic over $\mathfrak U^{tor}(v)$. Thus $H_{p_\tau}^m$ is finite flat.
\end{proof}

\subsection{The sheaves $\mathcal F$ and integral automorphic sheaves}
\label{secttriv}
We denote \[\{ p_\tau | \tau \in \mathcal I\} \cup \{0,h\} = \{ 0 =:p_0 \leq p_1 < p_2 < \dots < p_r \leq p_{r+1} := h\}.\]

We define for every $v>0$ such that $\mathcal X_0(p^n)(v) \subset \mathcal X_0(p^n)(\underline{\eps})$, a cover of $\mathcal X_0(p^n)(v)$.
If $p_r = h$ (in which case $p_1 = 0$ by duality and thus on $\mathfrak X$ the universal $p$-divisible group $A_i[\pi_j^\infty]$ has no multiplicative nor etale part), we set
\[\mathcal X_1(p^n)^{tor}(v) := \prod_{k=1}^{r+1} \Isom_{\mathcal X_0(p^n)^{tor}(v),pol,\mathcal O}(H_{p_k}/H_{p_{k-1}},\mathcal O/p^n\mathcal O^{p_k - p_{k-1}}),\]
where\footnote{We now write $H_{p_\tau}$ instead of $H_\tau$. Thus $H_{p_k} = H_\tau$ if $p_\tau = p_k$ and of $\mathcal O$-height $np_k$.} the condition \textit{pol} is trivial in case (AL), and in case (AU) means that we are also given an isomorphism,
\[ \nu' : (\mathcal O/p^n\mathcal O)^D \simeq (\mathcal O/p^n\mathcal O)^{\sigma},\]
i.e. \[\mathcal X_1(p^n)^{tor}(v) \subset \prod_{k=1}^{r+1} \Isom_{\mathcal X_0(p^n)^{tor}(v),\mathcal O}(H_{p_k}/H_{p_{k-1}},\mathcal O/p^n\mathcal O^{p_k - p_{k-1}})\times \Isom(\mathcal O/p^n\mathcal O)^D, (\mathcal O/p^n\mathcal O)^{\sigma}),\] satisfying the following.
There are fixed isomorphisms,
\[ \phi_k : (H_{p_k}/H_{p_{k-1}})^D \simeq (H_{p_{r-k+2}}/H_{p_{r-k-1}})^{(\sigma)},\]
induced by $H_{p_k}^\perp \simeq H_{p_{r-k+1}}^{(\sigma)}$, itself induced by the prime-to-$p$ polarisation on $\mathfrak X^{tor}.$\footnote{To be precise, we have on $\mathfrak X_0(p^n)^{tor}(\eps)$ a semi-abelian scheme $A$ and $H_{p_\tau}^m$ inside its $p$-torsion. The group homorphism $\lambda : A \fleche A^{\vee,(\sigma)}$, is a polarisation on $\mathfrak X_0(p^n)$, and this polarisation, which identifies $H_{p_k}^\perp$ with $H_{p_{r-k}}^{(\sigma)}$, induces an isomorphism $(H_{p_k}/H_{p_{k-1}})^D \simeq (H_{p_{r-k+1}}/H_{p_{r-k}})^{(\sigma)}$ everywhere. Indeed, it is enough to check it locally and introduce the formal-etale covering $\mathfrak U^{tor}(v)$ of subsection \ref{subsectBord}. Over $\mathfrak U^{tor}(v)$, the polarization extend as $\lambda$ an isogeny of 1-motives, thus induces an isogeny $\lambda^{ab}$ of their abelian parts on which the asserted isomorphism follows from theorem \ref{thrfiltcan} and normality of $\mathfrak U^{tor}(v)$.}

We require that for all $k$, the two isomorphisms,
\[ \psi_k^D : (\mathcal O/p^n\mathcal O)^{D,p_k - p_{k-1}} \fleche (H_{p_k}/H_{p_{k-1}})^D,\]
and 
\[ \psi_{r-k+2} : (H_{p_{r-k+2}}/H_{p_{r-k+1}}) \fleche \mathcal O/p^n\mathcal O^{D,p_k - p_{k-1}},\]
satisfies $\psi_k^D = \psi_{r-k+2}^{(\sigma),-1}$, after identifying source and  target via $\nu'$ and $\phi_k$. 

In this definition we have extended slightly the definitions of the (canonical) subgroups $H_k$ : for $k=0$ we set $H_0 = \{0\}$ and for $k = r+1$ we set $H_{r+1} = G[p^n]$.
If $p_r < h$ (in which case $p_1 > 0$ and on $\mathfrak X$ the universal $p$-divisible group $A_i[\pi_j^\infty]$ has a non-zero multiplicative and etale part), we set
\[\mathcal X_1(p^n)^{tor}(v) := \prod_{k=2}^{r} Isom_{\mathcal X_0(p^n)^{tor}(v),pol,\mathcal O}(H_{p_k}/H_{p_{k-1}}\mathcal O/p^n\mathcal O^{p_k - p_{k-1}}) \times \Isom_{\mathcal X_0(p^n)^{tor}(v)}(H_{p_1},(\mathcal O/p^n\mathcal O)^{p_1}).\]

\begin{rema}
\label{rem63}
\begin{enumerate}
\item This is because if $p_r = h$, the group $A_i[\pi_j^m]$ is finite flat and polarised on the all toroïdal compactification, but not if $p_r < r$, because $A_i[\pi_j^m]/H_{p_r}$, which is generically finite etale, is only quasi-finite on the boundary.

\item The point is that $\mathcal X_1(p^n)^{tor}(v)$ is a rigid open in (a toroïdal compactification of) the Shimura variety for $G$ of some level (which we could make explicite).
Indeed, if we use the definition of \cite{Lan} Definition 1.3.7.4. at our prime $(i,j)$, we see that it amounts to the previous definition. Indeed, the morphism 
\[ \nu : \ZZ/p^n\ZZ \overset{\sim}{\fleche} \mu_{p^n},\]
there induces a perfect pairing, 
\[ \mathcal O/p^n\mathcal O \times (\mathcal O/p^n\mathcal O)^{(\sigma)} \overset{tr(<,>)}{\fleche} \ZZ/p^n\ZZ \overset{\nu}{\fleche} \mu_{p^n},\]
where $\tr(<a,b>) := \tr(a\overline b)$ is a perfect pairing, and thus induces an isomorphism of $\mathcal O$-group schemes
\[ \nu' : (\mathcal O/p^n\mathcal O)^D \overset{\sim}{\fleche} (\mathcal O/p^n\mathcal O)^{(\sigma)}.\]
Let $\psi_k$ and $\psi_{r-k+2}$ be the isomorphism induced by a Level structure in the sense of \cite{Lan}, then let $r = p_k - p_{k-1}$,
\begin{center}
\begin{tikzpicture}[description/.style={fill=white,inner sep=2pt}] 
\matrix (m) [matrix of math nodes, row sep=3em, column sep=2.5em, text height=1.5ex, text depth=0.25ex] at (0,0)
{ 
(H_{p_k}/H_{p_{k-1}}) \times (H_{p_{r-k+2}}/H_{p_{r-k+1}})^{(\sigma)}  & & \mu_{p^n} \\
(\mathcal O/p^n\mathcal O)^r \times (\mathcal O/p^n\mathcal O)^r & & \ZZ/p^n\ZZ\\
};
\path[->,font=\scriptsize] 
(m-1-1) edge node[auto] {$Weil$} (m-1-3)
(m-1-1) edge node[auto] {$\psi_k \times \psi_{r-k+2}^{(\sigma)}$} (m-2-1)
(m-2-1) edge node[auto] {$\tr(<,>)_L$} (m-2-3)
(m-1-3) edge node[auto] {$\nu^{-1}$} (m-2-3)
;
\end{tikzpicture}
\end{center}
must be commutative, and by compatibility between the polarisations on $A[p^n]$ and $L$, the two pairings $\tr(<,>)_L \circ (\psi_k \times \psi_{r-k+2}^{(\sigma)})$ and 
$\tr(<,>)_L \circ (\psi_k \times (\nu'\circ(\psi_k^D)^{-1}\circ\phi_k^{-1}))$ must coincide, thus $\psi_{r-k+2}^{(\sigma)} = \nu'\circ(\psi_k^D)^{-1}\circ\phi_k^{-1}$ by \cite{Lan} Corollary 1.1.4.2.
\end{enumerate}
\end{rema}

\begin{defin}
Let $\mathfrak X_1(p^n)^{tor}(v)$ be the normalisation of $\mathfrak X_0(p^n)^{tor}(v)$ in $\mathcal X_1(p^n)^{tor}(v)$. 
It is flat, proper and normal over $\Spec(\mathcal O_K)$, and moreover we have maps
\[ \pi_{n,n-1} : \mathfrak X_1(p^n)^{tor}(v) \fleche \mathfrak X_1(p^{n-1})^{tor}(v),\]
by normalisation of the map sending $(\psi_k)$ to $(\psi_k[p^{n-1}])$.
\end{defin}

\begin{prop}
\label{prop63}
Suppose $(\eps,v,n)$.
For every $\tau$, there is on $\mathfrak X_1(p)^{tor}(v)$ a locally free $\mathcal O_{\mathfrak X_1(p)^{tor}}$-module of rank $p_\tau$ $\mathcal F_\tau \subset \omega_{G,\tau}$, (respectively in case (AL) also a sheaf $\mathcal F_\tau^\perp \subset \omega_{G^D,\tau}$) containing 
\[ p^{\frac{K_\tau(q_\bullet) + S_\tau(\eps_\bullet)}{p^f-1}} \omega_\tau \quad (\text{respectively } p^{\frac{K_\tau(p_\bullet) + S_\tau(\eps_\bullet)}{p^f-1}} \omega_{G^D,\tau}).\]
For all $n$, it induces by pullback by $\pi_n = \pi_{n,1}$ a sheaf $\mathcal F_\tau$ (resp. and $\mathcal F_\tau^\perp$) on $\mathfrak X_1(p^n)^{tor}(v)$, 
endowed with a compatible map for all $w_\tau < n-\eps_\tau$, for all $\Spec(R) \subset \mathfrak X_1(p^n)^{tor}(v)$,
\[ \HT_{\tau,w_\tau} : H_{p_\tau,n}^D(R) \fleche \mathcal F_\tau \otimes R_{w_\tau},\]
(resp. 
\[ \HT^\perp_{\tau,w_\tau} : (H_{p_\tau,n}^\perp)^D(R) = (G[p^n]/H_{p_\tau,n})(R) \fleche \mathcal F^\perp_\tau \otimes R_{w_\tau},\]
which induces an isomorphism,
\[ H_{p_\tau,n}^D(R) \otimes R_{w_\tau} \fleche \mathcal F_\tau \otimes R_{w_\tau},\]
(resp. $\HT_{\tau,w_\tau}^\perp \otimes R_{w_\tau}$ is also an isomorphism).
\end{prop}

\begin{proof}
Indeed, we can work locally over $S = \Spec(R)$. We  have isomorphisms $(H_{p_k}/H_{p_{k-1}})^D(R) \simeq (\mathcal O/p^n\mathcal O)^{p_k - p_{k-1}}$ but as $H_{p_\tau}^D(R)$ is a $\mathcal O/p^n\mathcal O$-module killed by $p^n$ and of finite presentation, there exists an isomorphism $H_{p_\tau}^D(R) \simeq (\mathcal O/p^n\mathcal O)^{p_\tau}$. We can thus work as in \cite{AIP} proposition 4.3.1 (see \cite{Her3} Proposition 6.1), where the analogs of the proposition are assured by Theorem \ref{thrfiltcan}, and the construction of $\mathfrak X_0(p^n)^{tor}(v)$.
\end{proof}

\begin{prop} 
Suppose we are given an isogeny on $\mathfrak X_1(p)^{tor}(v)$, $\phi : G' \fleche G$ where $G'$ is a $p$-divisible group, together with subgroups $H'_{p_\tau} \subset G'[p]$ satisfying the properties of Theorem \ref{thrfiltcan}. We can thus define $\mathcal F'$ for $G'$ similarly. Then the induced map,
\[ \phi^* : \omega_{G'} \fleche \omega_G,\]
sends $\mathcal F'$ in $\mathcal F$.
\end{prop}

\begin{proof}
As the groups in Theorem \ref{thrfiltcan} are steps of some Harder-Narasihman filtration, and this filtration is functorial, $\phi$ induces a map
\[ \phi : H'_{p_\tau} \fleche H_{p_\tau}.\]
The rest follows easily (see e.g. \cite{AIP} Proposition 4.4.1).
\end{proof}

\subsection{Constructing Banach sheaves}

Out of the universal isomorphisms \[\psi_k^D : (\mathcal O/p^n\mathcal O)^{p_k - p_{k-1}} \fleche (H_{p_k}/H_{p_{k-1}})^D,\]
on $\mathcal X_1(p^{n})^{tor}(v)$, we get a (full) flag of $(H_{p_k}/H_{p_{k-1}})^D$, and thus (inductively) of $H_{p_s}^D$ for all $s$\footnote{By first taking the full flag of $(H_{p_s}/H_{p_{s-1}})^D$ given previously, and then lifting the one of $(H_{p_s}/H_{p_{s-2}})^D/(H_{p_s}/H_{p_{s-1}})^D \simeq (H_{p_{s-1}}/H_{p_{s-2}})^D$ and so on...} by writing for all $i$, $e^k_1,\dots,e^k_{p_k-p_{k-1}}$ the natural basis of 
$(\mathcal O/p^{n}\mathcal O)^{p_k-p_{k-1}}$, and we thus denote $x^k_i$ the corresponding images in $(H_k/H_{k-1})^D$ through $\psi_k$. 
Choose a lift of this basis,
\[(x_1,\dots,x_{p_s})\]
of $H_s^D$,
and denote $\Fil_m^\psi$ the subgroup of $H_s^D$ generated by $x_1,\dots,x_m$. These subgroups does not depends on the lifts. 
From now on, fix $v >0$ such that $\mathcal X_0(p^n)^{tor}(v) \subset \mathcal X_0(p^n)^{tor}(\underline \eps)$ (i.e. such that $(\eps,v,n)$ is satisfied). In particular, we have the sheaves $\mathcal F_\tau$ on $\mathfrak X_0(p^n)^{tor}(v)$ and the compatibilities with $HT_\tau$ of the proposition \ref{prop63}. 
For simplicity, in case (AL) we call $\mathcal T$ the set of embeddings of $\mathcal O$ together with their conjugate, and represent its elements by $\tau,\overline \tau$.
For all $\overline \tau$, we mean by $\omega_{\overline\tau}$ the sheaf $\omega_{G^D,\tau}$, for $\mathcal F_{\overline \tau}$ the sheaf $\mathcal F_\tau^\perp$ and $\HT_{\overline\tau} = \HT_\tau^\perp$. We hope it will not lead to any confusion.

\begin{defin}
For all $\tau$ let $\mathcal Gr_\tau$ be the Grassmanian parametrizing all complete Flags of $\mathcal F_\tau$, and $\mathcal Gr_\tau^+$ which parametrizes same flags, together with a basis of the graded pieces.

Let $w \leq n-\eps_\tau$. For all $R'$ in $R-Adm$, an element $\Fil_\bullet \mathcal F_\tau$ of $\mathcal Gr_\tau(R')$ 
(respectively $(\Fil_\bullet \mathcal F_\tau,w_\bullet)$ of $\mathcal Gr_\tau^+(R')$)
is said to be $w$-compatible with $\psi$ if $\Fil_\bullet \mathcal F_\tau \equiv \HT_\tau(\Fil_\bullet^\psi) \pmod{p^wR'}$ (respectively if moreover $w_i \equiv \psi(x_i) \pmod{p^wR' + \Fil_{i-1}\mathcal F_\tau}$). This definition does not depends on the choice of the lifts $(x_i)$.
\end{defin}

Of course $\Fil_\bullet \mathcal F_\tau$ and $\Fil_\bullet^\psi$ are not always defined for the same index set for $\bullet$. It is understood that we restrict $\bullet$ to the smallest of the two index sets.

\begin{prop}
For each $\tau \in \mathcal T$ and $w_\tau < n_\tau - \eps_\tau$, there exists formal schemes,
\[ \mathfrak{IW}_{\tau,w_\tau}^+ \overset{\pi_1}{\fleche} \mathfrak{IW}_{\tau,w_\tau} \overset{\pi_2}{\fleche} \mathfrak X_1(p^{n})^{tor}(v),\] 
where $\pi_1$ is a $\mathfrak T_{\tau,w_\tau}$-torsor, and $\pi_2$ is affine.
\end{prop}

\begin{proof}
We set, following \cite{AIP},
\[
\mathfrak{IW}_{\tau,w}: \begin{array}{ccc}
R-Adm  &  \fleche & Sets  \\
  R' & \longrightarrow  &   \{w-\text{compatible } \Fil_\bullet \mathcal F_\tau \in \mathcal Gr_\tau(R')\}
\end{array}
\]
\[
\mathfrak{IW}_{\tau,w}^+ : \begin{array}{ccc}
R-Adm &  \fleche & Sets  \\
  R' & \longrightarrow  &   \{w-\text{compatible } (\Fil_\bullet \mathcal F_\tau,w_\bullet^\tau) \in \mathcal Gr^+_\tau(R')\}
\end{array}
\]
These are representable by affine formal schemes (some admissible open in an admissible formal blow-up of the previous Grassmanians).

\end{proof}

We denote by,
\[ \mathfrak{IW}_w^+ = \prod_\tau \mathfrak{IW}_{\tau,w}^+ \overset{\pi_1}{\fleche} \mathfrak{IW}_{w}=\prod_\tau \mathfrak{IW}_{\tau,w},\]
and $\mathcal{IW}_{\tau,w}^+,\mathcal{IW}_{\tau,w},\mathcal{IW}_{w}^+,\mathcal{IW}_{w}$ the corresponding generic fibers.

Denote by $\pi : \mathfrak{IW}_w^+ \fleche \mathfrak X_{0}(p^n)(v).$

\begin{defin}
Let $\kappa$ be a $w$-analytic character in $\mathcal W$. The formal sheaf,
\[ \mathfrak w_w^{\kappa\dag} := \pi_*\mathcal O_{\mathfrak{IW}_w^+}[\kappa],\]
is a small formal Banach sheaf on $\mathfrak X_0(p^n)^{tor}(v)$.
\end{defin}

\begin{proof}
Denote $\kappa^0$ its restriction to $\mathcal W^0$ of $\kappa$. It is analytic on $\mathcal T_w$. The map
\[ \pi_1 :\mathfrak{IW}_w^+ \fleche \mathfrak{IW}_w,\]
is a torsor over $\mathfrak T_w$, thus $(\pi_1)_*\mathcal O_{\mathfrak{IW}_w^+}[\kappa^0]$ is invertible, and 
\[ \pi_2 : \mathfrak{IW}_w \fleche \mathfrak X_1(p^n)^{tor}(v),\]
is affine, thus $(\pi_2\circ\pi_1)_*\mathcal O_{\mathfrak{IW}_w^+}[\kappa^0]$ is a small formal Banach sheaf. As $\mathfrak X_0(p^n)^{tor}(v)$ is quasi excellent (finite-type over $\mathcal O_K$), thus Nagata,
the map $\mathfrak X_1(p^n)^{tor}(v) \fleche \mathfrak X_0(p^n)^{tor}(v)$ is finite, and we can use \cite{AIP} with the action of $T(\mathcal O/p^n\mathcal O)$.
Thus, \[\mathfrak w_w^{\kappa\dag} = \left((\pi_2\circ\pi_1)_*\mathcal O_{\mathfrak{IW}_w^+}[\kappa^0](\kappa^{-1})\right)^{T(\mathcal O/p^n\mathcal O)},\]
is a small Banach sheaf on $\mathfrak X_0(p^n)^{tor}(v)$.
\end{proof}

We would like to descend further to $\mathfrak X^{tor}(v)$, i.e. at Iwahori level, unfortunately the map $\mathfrak X_0(p^n)^{tor}(v) \fleche \mathfrak X^{tor}(v)$ is not finite in general...

Let $\omega_w^{\kappa\dag}$ be the associated rigid sheaf (\cite{AIP} appendice) on $\mathcal X_0(p^n)^{tor}(v) \subset \mathcal X_0(p^n)^{tor}(\underline \eps)$. 

\subsection{Descent to Iwahori level}
\label{sect63}

In order to get an action of Hecke operators at $p$, which are defined at Iwahori level, we will need to descend our construction at this level. Fortunately, this is possible in rigid 
fiber.

We haven't really defined $\mathcal X_1(p^n)^{tor}$ (but see remark \ref{rem63}), and it will not be usefull for us, but in general 
$\mathcal X_1(p^n)^{tor} \fleche \mathcal X_0(p)^{tor}$ is a torsor over the group\footnote{This is for $p_r = h$, there is an analogous description when $p_r < h$.}
\[ I(p^n) := 
\left\{
\left(
\begin{array}{cccc}
A_1  &  \star & \star & \star  \\
  & A_2  & \star & \star  \\
  &   &  \ddots & \star\\
 p\mathcal O/p^n & & & A_{r+1} 
\end{array}
\right) :  \begin{array}{c}
A_i \in I_p(\mathcal O/p^n\mathcal O) \subset \GL_{p_i-p_{i-1}}(\mathcal O/p^n)  \\
A_i^t = A_{r-i+2}^{(\sigma)}
\end{array}
\right\} \pmod{U_P},
\]

where we chose an ordering $\{ p_\tau,q_\tau | \tau \in \mathcal T (= \mathcal T_{(i,j)})\} \cup \{0,h\} = \{ 0 \leq p_1 < p_2 < \dots < p_r \leq p_{r+1} = h\}$, and $h = h_{(i,j)}$ is the 
$\mathcal O_{(i,j)}$-height of $A_i[\pi_j^\infty]$, $I_p$ denote the standard Iwahori subgroup, 
and $U_P$ is the standard upper-bloc-diagonal unipotent associated to $p_1 \leq p_2 \dots \leq p_{r+1} = h$ 
(remember that we fixed a couple $(i,j)$ at this point so here everything is related to the group $G = G_{(i,j)}$ at the place $(i,j)$). Of course, here we chose a specific pairing so that this parabolic is upper-triangular.

The group $I(n)$ does not preserve $\mathcal X_1(p^n)^{tor}(v)$ : the reason is that the condition on $\mathcal X_1(p^n)^{tor}(v)$ to be "close to the $\mu$-ordinary canonical locus" (i.e. that the group of height $p_\tau$ have big enough degree) fixes the group of height $np_\tau$ to be equal to the canonical (and thus unique) corresponding one.
In particular $\mathcal X_1(p^n)^{tor}(v) \fleche \mathcal X_0(p)^{tor}(v)$ is a torsor over,
\[ I^0(n) :=\im\left(\left\{
\left(
\begin{array}{cccc}
A_1  &  \star & \star & \star  \\
  & A_2  & \star & \star  \\
  &   &  \ddots & \star\\
 0 & & & A_r 
\end{array}
\right) :   \begin{array}{c}
A_i \in I_p(\mathcal O/p^n\mathcal O) \subset \GL_{p_i-p_{i-1}}(\mathcal O/p^n)  \\
A_i^t = A_{r-i+2}^{(\sigma)}
\end{array}
 \right\} \fleche I(n) \right),
\]

\begin{rema}
The group $I^0(n)$ is related to the group $P^0_{(p_\tau)}$ of section \ref{sect4}
\end{rema}

Denote by $U^0(n)$ the (diagonal, not just bloc-diagonal) subgroup,

\[
\left(
\begin{array}{cccc}
1  &  \star &  & \star \\
  & 1  & & \vdots  \\
  &   &  \ddots &\star \\ 
    &   &   & 1 \\ 
\end{array}
\right) \subset I^0(n)
\] 
Define $\mathcal X_0^+(p^n)^{tor}(v)$ as the quotient of $\mathcal X_1(p^n)^{tor}(v)$ by $U^0(m)$, which doesn't parametrizes trivialisations of the groups $H_k/H_{k-1}$ but only full flags of subgroups of this quotients, together with a basis of the graded pieces. Actually we can also define the same way $\mathcal F_\tau$ over $\mathcal X_0(p^n)^+(v)^{tor}$ (i.e. the sheaves $\mathcal F_\tau$ descend). As the action of $U^0(m)$ on $\mathcal X_1(p^n)^{tor}(v)$ lifts to $\mathcal{IW}_w^+$, denote also by
$\mathcal{IW}_w^{0,+}$ the quotient of by $U^0(m)$. 
As $\mathcal F_\tau \simeq \omega_\tau$ over $\mathcal X_0^+(p^n)^{tor}(v)$ (i.e. after inverting $p$), we thus have an injection, \[\mathcal{IW}_w^{0,+} \subset (\mathcal T/U)_{\mathcal X_0^+(p^n)^{tor}(v)}.\]

\begin{prop}
\label{propdescente}
If 
$n - \eps_\tau > w > n-1$, then the composite,
\[ \mathcal{IW}_w^{0,+} \hookrightarrow (\mathcal T/U)_{\mathcal X_0^+(p^n)^{tor}(v)} \fleche (\mathcal T/U)_{\mathcal X^{tor}(v)},\]
is an open immersion.
\end{prop}

\begin{proof}
Denote by $V \subset \mathcal X^{tor}(v) = \mathcal X_0(p)^{tor}(v)$ the image by $\pi_n$ of $\mathcal X_0^+(p^n)^{tor}(v)$. Up to reducing to a suitable affinoide $U \subset \mathcal X_0^+(p^n)^{tor}(v)$, the previous composite map 
$h$ is given by, 
\[ h : \coprod_\tau \coprod_{\gamma \in S} 
M_\tau\left(
\begin{array}{cccc}
 1+p^wB(0,1) &   &  & \\
 p^wB(0,1) &  1+p^wB(0,1)  & &  \\
  &   &   \ddots & \\
  & & &  1+p^wB(0,1)
\end{array}
\right) \gamma
 \fleche \coprod_\tau (\GL_{p_\tau}/U)_{\pi_n(U)},\]
 where $S$ is a set of representative of $I^0(n)/U^0(n)$ in $I(\mathcal O) \subset \GL_{p_\tau}(\mathcal O)$, and $M_\tau$ is the matrix relating the basis of $H_{p_\tau}^D$ to the fixed one of $\omega_{\tau}$, which is thus related to the Hodge-Tate map (or equivalently relating a fixed basis of $\mathcal F_\tau$ to a fixed one of $\omega_\tau$).
 In particular there exists $M^*_\tau$ such that, $M^*_\tau M_\tau = p^c Id_{p_\tau}$ for some $c$ (which we could bound in terms of the Hasse invariant or $\frac{1}{p-1}$, but it is not even necessary). From this, we deduce that $M^* \circ h$ is injective, thus the same thing for $h$.
\end{proof}


Thus we have a map $g_n : \mathcal{IW}_w^{0,+} \fleche \mathcal X^{tor}(v)$. It is not clear that the map,
\[ \pi_n : \mathcal X_0^+(p^n)^{tor}(v) \fleche \mathcal X^{tor}(v) := \mathcal X_0(p)^{tor,\mu-can}(v),\]
is surjective. But still, having $n$ fixed, $\pi_n(\mathcal X_0^+(p^n)^{tor}(v))$ 
describe a basis of strict neighborhoods of $\mathcal X^{tor,\mu-can}$ by lemma \ref{lemmastricetale}.

%

\begin{defin}
Suppose $(\eps,n,v)$ is satisfied. The open $\pi_n(\mathcal X_0^+(p^n)^{tor}(v))$ is a strict neighborhood of the ($\mu$-canonical) ordinary locus $\mathcal X^{tor}(0) =\mathcal X^{tor}(0)^{\mu-can}$ included in $\mathcal X^{tor}(\underline \eps)$.
On $\pi_n(\mathcal X_0^+(p^n)^{tor}(v))$, if $w \in ]n-1,n-\eps_\tau[$ for all $\tau$, for all $\kappa$ $w$-analytic,
we define the following sheaf,
\[ \omega^{\kappa\dag}_{w} = ((g_n)_*\mathcal O_{\mathcal{IW}_w^{0,+}})[\kappa].\]
It is called the sheaf of overconvergent, $w$-analytic modular forms of weight $\kappa$. For every $v' > 0$ small enough such that $\mathcal X^{tor}(v') \subset \pi_n(\mathcal X_0^+(p^n)^{tor}(v))$, the module
\[ M_w^{\kappa\dag}(v') = H^0(\mathcal X^{tor}(v'),\omega_w^{\kappa\dag}), \]
(respectively $M_{cusp,w}^{\kappa\dag}(v') = H^0(\mathcal X^{tor}(v'),\omega_w^{\kappa\dag}(-D))$) is called the module of $v'$-overconvergent, $w$-analytic (respectively cuspidal) modular forms of weight $\kappa$.
\end{defin}

\begin{rema}
In the previous compatibilities, if $(\eps,n,v)$ is satisfied, $(\eps,n,v')$ is for all $v' < v$. Also, because of the compatibility between $w$ and $n$, $n$ is uniquely defined (and is thus suppressed form the notation of $\omega_w^{\kappa\dag}$).
Thus, we can choose $v$ arbitrarily close to $0$ in the previous definition. Also, for every $w$ and $\kappa$, there exists $n_0$ such that for all $n \geq n_0$, there is $w' > w$, and $\kappa$ is $w'$-analytic with $n-1< w' < n-\eps_\tau$ for all $\tau$. In particular, there exists constants $v_0,w_0$ such that $M_w^{\kappa\dag}(v)$ is defined for all $v < v_0$ 
and $w > w_0$ such that $w \in ]n-1,n-\eps[$ (for some integer $n$ large enough).
\end{rema}

Suppose $n'-\eps_\tau > w' > w$ with $w \in ]n -1 ,n - \eps[$ and $n \leq n'$. As a flag with graded basis which is $(n',w')$-compatible is also $(n,w)$-compatible, there is an injective map,
\[ \mathcal{IW}^+_{n',w'} \hookrightarrow \mathcal{IW}^+_{n,w} \times_{\mathcal X_1(p^n)^{tor}(v)} \mathcal X_1(p^{n'})^{tor}(v).\]
In particular, we have a natural map, for every $w'$-analytic $\kappa$,
\[ \omega_w^{\kappa\dag} \hookrightarrow \omega_{w'}^{\kappa\dag},\]
over $\pi_{n'}(\mathcal X_0^+(p^{n'})^{tor}(v))\subset \pi_n(\mathcal X_0^+(p^n)^{tor}(v))$.

\begin{defin}
For $w >0$, the module,
\[ M^{\kappa\dag} = \varinjlim_{v\rightarrow 0,w\rightarrow \infty} M_w^{\kappa\dag}(v) \quad (\text{respectively } M_{cusp}^{\kappa\dag} = \varinjlim_{v\rightarrow 0,w\rightarrow \infty} M_{cusp,w}^{\kappa\dag}(v))\]
is the module of overconvergent locally analytic (respectively cuspidal) modular forms of weight $\kappa$.
\end{defin}

\begin{rema}
In the previous definition, it is understood that the limit is taken on $v,w$ such that $w \in ]n-1,n-\eps[$ for some $n = n_w$,  $(\eps,n,v_0)$ is satisfied for some $v_0$ and $\mathcal X^{tor}(v) \subset \pi_n(\mathcal X_0^+(p^n)^{tor}(v_0)) \subset \mathcal X^{tor}(\eps)$. Thus in particular $(\eps,n,v_0)$ is satisfied and $v \leq v_0$. 
\end{rema}

\subsection{Some complexes}

For compatibilities reasons with Hecke operators and to control the structure on the previous modules, we will need to define complexes overconvergent sections.
Recall that on $\mathcal X^{tor} = \mathcal X_0(p)^{tor}$, our rigid toroïdal Shimura variety with fixed Iwahori level, we have defined two basis of strict neighborhoods of $\mathcal X^{tor}(0)$ (the \textit{canonical} $\mu$-ordinary locus whose points have maximal degree), one given by
\[ \mathcal X^{tor}(\underline \eps),\]
for $\underline\eps = (\eps_\tau)_\tau$ (of points with degrees bigger than the maximal one minus $\eps_\tau$), and
\[ \mathcal X^{tor}(v),\]
for $v \in (0,1]$ (describing the connected components containing $\mathcal X^{tor}(0)$ of points with Hasse invariant of valuation smaller than $v$). Because we will need to let the neighborhood described in terms of the degree vary, we from now on call $\underline \eps_0$ a number fixed to be able to define the sheaves $\omega_w^{\kappa\dag}$, and we will always consider small enough opens $\mathcal X(v)$ and $\mathcal X(\underline \eps)$ so that there exists the sheaves $\omega_w^{\kappa\dag}$ on them. In particular, once $w$ is fixed this just implies that $v$ or $\eps$ are small enough (depending on $w$).

Ultimately we are interested by (finite slope) overconvergent cuspidal modular forms, that is, (finite slope) elements of
\[ \varinjlim_{v \fleche 0,w \fleche \infty} H^0(\mathcal X^{tor}(v),\omega^{\kappa\dag}_w(-D)) = \varinjlim_{\underline\eps \fleche 0,w \fleche \infty} H^0(\mathcal X^{tor}(\underline \eps),\omega^{\kappa\dag}_w(-D)).\]
We temporarily introduce the following complexes,

\begin{defin}
Let $w >0$, $\mathcal U = \Spm(A) \subset \mathcal W$ an affinoid such that $\kappa_\mathcal U$ is $w$-analytic, and define for $v,\underline \eps$ small enough \footnote{such that the sheaves $\omega_w^{\kappa_{\mathcal U}\dag}$ is defined on $\mathcal X^{tor}(v)\times \mathcal U$ resp. $\mathcal X^{tor}(\underline\eps)\times \mathcal U$.},
\[C_{cusp}(v,w,\kappa_U) = R\Gamma(\mathcal X^{tor}(v)\times \mathcal U,\omega^{\kappa_\mathcal U\dag}_w(-D)),\]
and
\[ C_{cusp}(\underline \eps,w,\kappa_U) = R\Gamma(\mathcal X^{tor}(\underline \eps)\times \mathcal U,\omega^{\kappa_\mathcal U\dag}_w(-D)).\]
We can analogously define the non-cuspidal versions of these complexes.

We also define
\[ H^i_{cusp,\dag}(\kappa_U) = \lim_{v \fleche 0,w\fleche \infty}H^i(\mathcal X^{tor}(v)\times \mathcal U,\omega^{\kappa_\mathcal U\dag}_w(-D)),\]
and \[H^{i'}_{cusp,\dag}(\kappa_U) = \lim_{\underline\eps \fleche 0,w\fleche \infty}H^i(\mathcal X^{tor}(\eps)\times \mathcal U,\omega^{\kappa_\mathcal U\dag}_w(-D)).\]
\end{defin}

In particular $H^0_{cusp,\dag}(\kappa) = H^{0'}_{cusp,\dag}(\kappa)$ is simply the space of overconvergent locally analytic cuspidal modular forms of weight $\kappa$, and we will see that the higher cohomology groups vanishes (their finite slope part at least).

\begin{prop}
\label{pro617}
The previous complexes are represented by bounded complexes of projective $A[1/p]$-modules (i.e. perfect complexes in the sense of \cite{urb}).
\end{prop}

\begin{proof}
This is the same proof as \cite{PilHigher}. We have maps,
\[ \mathcal{IW}_w^+\times_{\mathcal X(v)}\mathcal U \overset{\pi_1}{\fleche} \mathcal {IW}_w\times_{\mathcal X(v)}\mathcal U \fleche \mathcal X_1(p^n)^{tor}(v)\times \mathcal U,\]
and sheaves $\mathcal L^{\kappa_U} = (\pi_{1*}\mathcal O_{\mathcal{IW}^+_w})[\kappa_{\mathcal U}^0]$ (for the action of ${\mathcal T_w}$), this is a line bundle on $\mathcal{IW}_w\times_{\mathcal X(v)}\mathcal U$), and $\omega^{\kappa_U^0\dag}_w = \pi_{2*}\mathcal L^{\kappa_U}$.
Moreover
\begin{eqnarray*} R\Gamma(\mathcal X^{tor}(v)\times U,\omega^{\kappa_U\dag}_w(-D)) &=& R\Gamma(I^0(n),R\Gamma(\mathcal X_1(p^n)^{tor}(v)\times U,\omega^{\kappa_U^0\dag}_w(-D))(-\kappa)) \\
&=& R\Gamma(I^0(n),R\Gamma(\mathcal{IW}_w \times U,\mathcal L^{\kappa_U}(-D))(-\kappa)).\end{eqnarray*}

The last equality is because $\mathcal{IW}_w \fleche \mathcal X_1(p^n)^{tor}(v)$ is locally affinoid.
Now if you choose $\mathfrak U$ a covering of $\mathcal{IW}_w\times\mathcal U$ which is $I^0(n)$-stable (by adding all translates by $I^0(n)$ if necessary), then the Cech complex of this covering is perfect and calculates $R\Gamma(\mathcal{IW}_w \times U,\mathcal L^{\kappa_U}(-D))$, and twisting the action of $I^0(n)$, and looking at the direct factor of invariants by $I^0(n)$ (we are in characteristic 0), this is still perfect and calculates $R\Gamma(\mathcal X^{tor}(v)\times U,\omega^{\kappa_U\dag}_w(-D))$.
The same remains true with $\mathcal X^{tor}(\underline \eps)$ (for $\eps$ small enough) instead of $\mathcal X^{tor}(v)$.
\end{proof}

\section{Hecke Operators}
\label{sect7}
In this section we will construct Hecke operators, both at $p$ and outside $p$. As noted in \cite{AIP,Brasca}, it is not true that the Hecke correspondances will extend to a fixed choice of a toroidal compactification, nevertheless we can adapt the choice of toroïdal compactifications and use results of Harris (\cite{HarCor} Proposition 2.2).

\begin{lemma}[Harris]
Let $\Sigma,\Sigma'$ be two smooth projective polyhedral cone decomposition, and $\mathcal X_1(p^n)^{tor}_\Sigma,\mathcal X_1(p^n)^{tor}_{\Sigma'}$ the associated toroïdal compactifications.
Then there is a canonical isomorphism 
$H^*(\mathcal X_1(p^n)^{tor}_\Sigma,\mathcal O_{\mathcal{IW}^+}) \simeq H^*(\mathcal X_1(p^n)^{tor}_{\Sigma'},\mathcal O_{\mathcal {IW}^+})$.
\end{lemma}

\begin{proof}
To simplify notation, denote $X_{\Sigma} = \mathcal X_1(p^n)^{tor}_\Sigma$. Up to choosing a common refinement of $\Sigma$ and $\Sigma'$, we can suppose that $\Sigma'$ refines $\Sigma$ and look at the map
\[ \pi : X_{\Sigma'} \fleche X_\Sigma.\]
By results of Harris we have $\pi^*\omega_G = \omega_G'$. Moreover, we can take $\Sigma'$ small enough (which we do) such that it corresponds to a refinement of the auxiliary datum we chose in section \ref{sectNorm}. In particular, on the integral model $\mathfrak X_1(p^n)^{tor}_{\Sigma'}$, the groups $H_k$ are given by pullback of those on $\mathfrak X_1(p^n)^{tor}_\Sigma$. Thus we have a carthesian square,
\begin{center}
\begin{tikzpicture}[description/.style={fill=white,inner sep=2pt}] 
\matrix (m) [matrix of math nodes, row sep=3em, column sep=2.5em, text height=1.5ex, text depth=0.25ex] at (0,0)
{ 
\mathcal IW^+_{\Sigma'} & & \mathfrak X_{\Sigma'} \\
\mathcal IW^+_\Sigma& &\mathfrak X_{\Sigma}\\
};
\path[->,font=\scriptsize] 
(m-1-1) edge node[auto] {$i'$} (m-1-3)
(m-2-1) edge node[auto] {$i$} (m-2-3)
(m-1-1) edge node[auto] {$\pi'$} (m-2-1)
(m-1-3) edge node[auto] {$\pi$} (m-2-3)
;
\end{tikzpicture}
\end{center}
Also, by results of Harris (\cite{Har} (2.4.3)-(2.4.6)), we have quasi-isomorphisms $\pi_*\mathcal O_{X_{\Sigma'}} \overset{\sim}{\fleche} R\pi_*\mathcal O_{X_{\Sigma'}} = \mathcal O_{X_{\Sigma}}$. As $\mathcal IW^+ \fleche X$ is flat and $\pi$ is proper, we have thus by base change (see e.g. \cite{Stacks} 29.5.2)
\[ R\pi'_*\mathcal O_{IW^+_{\Sigma'}} \simeq \mathcal O_{IW^+_{\Sigma}}.\]

\end{proof}

\subsection{Hecke Operator outside $p$}

Let $\lambda$ be a place where our fixed level $K^p$ is hyperspecial, and fix $\gamma \in G(F^+_\lambda)$. Denote $C_\gamma$ the (analytic space associated to the) 
moduli space classifiying tuples
$(A_k,\iota_k,\lambda_k,\eta_k)$, $k=1,2$, of the type $G$, together with an isogeny $f: A_1 \fleche A_2$ which respects the addictionnal structure. It is endowed with two maps,
\[ C_\gamma \overset{p_1}{\underset{p_2}{\rightrightarrows}} \mathcal Y_{Iw(p)}(v),\]
where $p_k(f : A_1 \fleche A_2) = A_k$. Denote $C_\gamma(p^n) = C_\gamma \times_{\mathcal Y_{Iw(p)}(v)} \mathcal Y_{1}(p^n)(v)$.
But we can find choices of smooth projective polyhedral cone decompositions (see \cite{Lan} proposition 6.4.3.4) $\Sigma$ and $\Sigma'$ and associated toroïdal compactifications $X_{Iw(p),\Sigma}, C_{\gamma,\Sigma}, X_{Iw(p),\Sigma'}$ and maps $p_1 : C_{\gamma,\Sigma} \fleche X_{Iw(p),\Sigma}, p_2 : C_{\gamma,\Sigma} \fleche X_{Iw(p),\Sigma'}$ which extends the previous ones. As $v$ is away from $p$, this correspondance preserves $\mathcal X_{Iw(p)}(v)$, and the universal isogeny induces an isomorphism,
\[ f^* : p_2^*\mathcal F_{\mathcal X_1(p^n)_{\Sigma'}(v)} \overset{\sim}{\fleche} p_1^*\mathcal F_{\mathcal X_1(p^n)_\Sigma(v)},\]
and we can thus construct,
\begin{eqnarray*}
H^0(\mathcal X_1(p^n)_{\Sigma'}(v),\mathcal O_{\mathcal{IW}^+}) \overset{p_2^*}{\fleche} H^0(C_{\gamma,\Sigma}(p^n),p_2^*\mathcal O_{\mathcal{IW}^+}) \overset{f^*}{\fleche} \\
H^0(C_{\gamma,\Sigma}(p^n),p_1^*\mathcal O_{\mathcal{IW}^+}) \overset{\Tr p_1^*}{\fleche}H^0(\mathcal X_1(p^n)_\Sigma(v),\mathcal O_{\mathcal{IW}^+})
\end{eqnarray*}
As by the previous lemma, $H^0(\mathcal X_1(p^n)_{\Sigma'}(v),\mathcal O_{\mathcal{IW}^+})=H^0(\mathcal X_1(p^n)_{\Sigma}(v),\mathcal O_{\mathcal{IW}^+})$, we get an operator $T_\gamma$ on $H^0(\mathcal X_1(p^n)_{\Sigma}(v),\mathcal O_{\mathcal{IW}^+})$.
Similarly $T_\gamma$ acts on $C_{cusp}(v,w,\kappa_U)$ and $C_{cusp}(\underline \eps,w,\kappa_U)$ (as the isogeny is outside $p$) and their non-cuspidal analogues. We can thus forget about the choice of $\Sigma$ in the notations.

 \begin{rema}
 Here we made a slight abuse, as we used the notations with a fixed $(i,j)$. Of course, taking tensor products over the $(i,j)$ of the correspondings $\mathcal{IW}^+$ (which depends of the choice of $(i,j)$) solves this abuse of notation.
 \end{rema}

\begin{defin}
Let $\kappa \in \mathcal W(w)(L)$ with $ w \in]n-1+c,n-\eps_\tau[$. Restricting the previous operator to homegenous functions on $\mathcal X_{Iw(p)}(v)$  for $\kappa$, we get the Hecke operator,
\[ T_\gamma : M^{\kappa\dag}_w \fleche M^{\kappa\dag}_w.\]
Working over $\mathcal X_{Iw(p)}(v) \times \mathcal W(w)$, considering $\kappa = \kappa^{univ}$, we get an operator
\[ T_\gamma^{univ} : M^{\kappa^{univ}\dag}_w \fleche M^{\kappa^{univ}\dag}_w,\]
which is $\mathcal O_{\mathcal W(w)}$-linear, and an operator on $C_{cusp}(v,w,\kappa_{\mathcal W(w)})$ and $C_{cusp}(\underline\eps,w,\kappa_{\mathcal W(w)})$.
\end{defin}

Denote by $\mathcal H_{K^p}$ the spherical Hecke algebra of level $K^p$, the previous construction endow for each $w$ the modules $M^{\kappa^{univ}\dag}_w$ (respectively the module $M^{\kappa\dag}_w$ with $\kappa \in \mathcal W_w$) with an action of 
$\mathcal H_{K^p}$.

\subsection{Hecke operators at $p$}

At $p$, the construction of Hecke Operators is much more subtle than outside $p$, and even more subtle than in the ordinary case, as already remarked in \cite{Her3}.
Indeed, when the ordinary locus is non empty, only one operator, $U_{p,g}$ in \cite{AIP}, is compact on classical forms (it improves the "Hodge"-radius, i.e. the Hasse invariant), 
but does not improves the analycity radius for overconvergent forms, whereas the other operators, $U_{p,i}, i < g$ in \cite{AIP}, improves only the analycity radius (a priori).
Already for $U(2,1)$ with $p$ inert in the quadratic imaginary field the situation is different. Indeed, there is only one interesting operator, $U_p$, that improves at the same time both the Hodge-radius and the analicity radius.

Following \cite{Bijmu}, we define operators at $p$.

\subsubsection{Linear case} This is actually easier than the unitary case, and can be adapted from \cite{Bijmu} on $\omega^\kappa$ to $\omega^{\kappa\dag}_w$. But as this case can be recovered from the general Unitary case (considering $G \times G^D$ with canonical polarisation instead of $G$), we just write the details in the unitary case.

\subsubsection{Unitary case}
Fix as before $(i,j)$ a prime, that we supress from now on from the notations, and we can thus use $i$ as a variable. Let $G$ be the associated $p$-divisible group. Let $1 \leq i \leq \frac{h}{2}$ an integer, and define $C_i$ the moduli space $(A,\iota,\lambda,\eta,H_\bullet,L)$ where $(A,\iota,\lambda,\eta,H_\bullet) \in \mathcal Y_{Iw(p)}(v)$
and $L \subset G[p^2]$ be a totally isotropic subgroup such that $H_i \oplus L[p] = G[p]$ and $H_i^\perp \oplus pL = G[p]$, and denote the two maps,
 \[ C_i \overset{p_1}{\underset{p_2}{\rightrightarrows}} \mathcal Y_{Iw(p)}(v),\]
where $p_1(A,L) = A$ and $p_2(A,L) = A/L$.
Denote $C_i(p^n) = C_i \times_{\mathcal Y_{Iw}(v)}\mathcal Y_1(p^n)(v)$, and denote $f : A\fleche A/L$ the universal isogeny. As we are in characteristic zero, we can find smooth projective polyhedral cone decompositions $\Sigma,\Sigma',\Sigma''$ such that the previous correspondance extends to $p_1 : C_{i,\Sigma'} \fleche X_{Iw(p),\Sigma}, p_2 : C_{i,\Sigma'} \fleche X_{Iw(p),\Sigma''}$.
In \cite{Bijmu} Proposition 2.11, Bijakowski verifies that the previous correspondance stabilizes the open $\mathcal Y(\eps)$. More precisely, he verifies that the Hecke correspondance $C_i$ satisfies $\deg H_j' \geq \deg H_j$ with equality for $i = j$ if and only if $\deg H_i$ is an integer. Its proof extend to the case of a 1-motive in case of bad reduction, and thus extend to the boundary.

In particular, if $\eps_\tau <1$, by quasi-compacity of 
\[ \mathcal X_{Iw}(\forall \tau, \deg(H_\tau) \in [\lambda_\tau,\nu_\tau]), \]with\[ \sum_{\tau'} \min(p_{\tau'},p_\tau) - 1 < \lambda_\tau < \nu_\tau < \sum_{\tau'} \min(p_{\tau'},p_\tau), \lambda_\tau,\nu_\tau \in \QQ,\]
we can thus prove the following,

\begin{prop}
\label{propepsqc}
For all $w>0$, for all $\underline\eps >0$ sufficiently small, there exists $\underline\eps' < \underline\eps$ such that the Hecke correspondance $\prod(U_i)$ sends 
$\mathcal X_{Iw,\Sigma}(\underline\eps)$ into $\mathcal X_{Iw,\Sigma''}(\underline\eps')$.
Also, for all $\underline\eps >0$, and all $0 < \underline\eps' < \underline\eps$, there exists $N >0$ such that $\prod_i U_i^N$ send $\mathcal X_{Iw,\Sigma}(\underline\eps)$ in $\mathcal X_{Iw,\Sigma''}(\underline\eps')$.
\end{prop}

The universal isogeny $f$ induces a map,
\[ f^* : p_2^*\mathcal T_{an} \fleche p_1^*\mathcal T_{an},\]
which is an isomorphism, and denote $\widetilde{f}^*  = \bigoplus_\sigma \widetilde{f}_\sigma^*$ (using the decomposition $\omega = \bigoplus_\sigma \omega_{G,\sigma}$)
such that $\widetilde{f}_\sigma^*$ sends a basis $w_1',\dots,w_{p_\sigma}'$ of $\omega_{A/L,\sigma}$ to 
\[p^{-2}f^*w_1',\dots,p^{-2}f^*w_{p_\sigma - h + i}',p^{-1}f^*w'_{p_\sigma - h + i + 1},\dots,p^{-1}f^*w'_{p_\sigma-i},f^*w'_{p_\sigma - i +1},\dots,f^*w'_{p_\sigma},\]
(being understood that the terms on the left with $p^{-2}$ only appears if $p_\sigma > h-i$ and terms with $p^{-1}$ only if $p_\sigma > i$).
Another way to write it is to set, $a_\sigma = \max(p_\sigma - (h-i),0), b_\sigma =\max( \min(h-2i,p_\sigma - i),0)$ and $c_\sigma = \min(i,p_\sigma)$ (thus $a_\sigma + b_\sigma+c_\sigma =1$). Then $\widetilde{f}_\sigma^*$ sends  $w_1',\dots,w_{p_\sigma}'$ to,
\[p^{-2}f^*w_1',\dots,p^{-2}f^*w_{a_\sigma}',p^{-1}f^*w'_{a_\sigma+ 1},\dots,p^{-1}f^*w'_{a_\sigma+b_\sigma},f^*w'_{a_\sigma+b_\sigma+1},\dots,f^*w'_{p_\sigma},\]

\begin{rema}
This normalisation is made in order to make the operator $U_i$ vary in a family (it corrects the multiplication by $p$ that appears on $\omega$ if we do the quotient by $L$). 
It coincides with normalisation of \cite{Bijmu} for classical sheaves.
\end{rema}

%

\begin{defin}
Let $\underline w = (w^\tau_{i,j})_\tau$, such that for all $(i,j,\tau)$, $w_{i,j}^\tau \in ]c;n-\eps_\tau[$
and define $\mathcal{IW}^{0,+}_{\underline w}$ to be the open subspace of $\mathcal T^\times/U$ over $\pi_n(\mathcal X_0^+(p^n)^{tor}(v))$ such that its $L$-points, for all $L$ over $K$, is the datum of a $\mathcal O_L$-point of $\pi_n(\mathcal X_0^+(p^n)^{tor}(v))(\mathcal O_L)$, thus in particular an abelian scheme $A/\Spec(\mathcal O_L)$ for which $G[p^n]$ has (canonical by Theorem \ref{thrfiltcan}) filtration by subgroups $H_\tau$, together with a flag $\Fil_\bullet \mathcal F_\tau$ for all $\tau$ with graded pieces $w^\tau_\bullet$, such that there exists a (polarized) trivialization $\psi$ (as in section \ref{secttriv}), such that
\[ \omega_i^\tau \pmod{\Fil_{i-1}\mathcal F_\tau + p^{w_0^\tau}\mathcal F_\tau} = \sum_{j \geq i} a_{j,i} \HT_{\tau,w_0}(e_j),\]
where $w_0^\tau = n - \eps_\tau$ and with $a_{j,i} \in \mathcal O_L$ such that, $v(a_{j,i}) \geq w^\tau_{j,i}$ if $j > i$ and $v(a_{i,i} -1)  \geq w_{i,i}^\tau$. We then define as before $\omega^{\kappa\dag}_{\underline w}$ on $\pi_n(\mathcal X_0^+(p^n)^{tor}(v))$.
\end{defin}

\begin{rema}
In the previous definition, if we take $n' \geq n$, $w_{i,j}^\tau \in ]c,n-\eps_\tau[$ and we make the previous construction over $\pi_{n'}(\mathcal X_0^+(p^n)(v))$ for $w_0 = n-\eps_\tau$ or $w_0 = n' - \eps_\tau$, we get the same space. Thus, up to reducing the strict neighborhood, we suppress $n$ from the notation. When $\underline w$ is parallel and $0 < w < 1 - \eps_\tau$, then $\omega^{\kappa\dag}_{\underline w} = \omega^{\kappa\dag}_{ w}$.
\end{rema}

Suppose $\underline w$ satisfies 
\[ 
\left\{
\begin{array}{cc}
0 < w_{k,l}^\tau < w_0 - 2 & \text{if } a_\tau \neq 0  \\
0 < w_{k,l}^\tau < w_0 - 1 &\text{if } a_\tau = 0 \text{ and } b_\tau \neq 0  \\
0 < w_{k,l}^\tau < w_0  & \text{otherwise}
\end{array}
\right.
\]

\begin{prop}
\label{prop78}
Let $f$ the the universal isogeny over $C_i$. Then
\[ (\widetilde{f}^*)^{-1}p_1^*\mathcal{IW}_{\underline w}^{0,+} \subset p_2^*\mathcal{IW}_{\underline w'}^{0,+},\]
with 
\[ \underline w^{'\sigma}_{k,l} = 
\left\{
\begin{array}{cc}
w_{k,l}^\sigma +2 & \text{if } k > a_\sigma + b_\sigma \text{ and } l \leq a_\sigma  \\
w_{k,l}^\sigma + 1 & \text{if } (b_\sigma + a_\sigma\geq k > a_\sigma  \text{ and }l \leq a_\sigma), \text{ or } (b_\sigma+a_\sigma \geq l > a_\sigma  \text{ and } k > a_\sigma + b_\sigma)\\
w_{k,l}^\sigma  & \text{otherwise}
\end{array}
\right.
\]
\end{prop}

\begin{proof}
This is similar to \ref{propdeltarad} and \cite{AIP} Proposition 6.2.2.2. Indeed, in the basis given by the "trivialisation" of $(H_{p_k}/H_{p_{k-1}})^D$ on $\mathcal X_0^+(p^n)(\eps)$, the dual of the morphism $H_{p_\tau} \fleche H_{p_\tau}'$ induced by $f$,
\[  f^D : (H_{p_\tau}')^D \fleche H_{p_\tau}^D,\]
is given by $\Diag(p^2,\dots,p^2,p,\dots,p,1,\dots1)$, where $p^2$ appears $a_\tau$-times, $p$ appears $b_\tau$-times and 1 $c_\tau$-times.
The rest follows exactly as in \cite{AIP} Proposition 6.2.2.2, as $\pi^*\mathcal F'_\tau \supset p\mathcal F_\tau$ is $a_\tau = 0$ and $b_\tau \neq 0$, $\pi^*\mathcal F'_\tau \supset p^2 \mathcal F_\tau$ if $a_\tau \neq 0$ and $\mathcal F_\tau = \pi^*\mathcal F'_\tau$ otherwise.
\end{proof}

We can thus define the operator $U_i^0$, 
\begin{eqnarray*} 
H^0(\mathcal X_{\Sigma''}(\underline\eps),\omega_{\underline w'}^{\kappa\dag}) \overset{p_2^*}{\fleche} H^0(C_i,p_2^*\omega_{\underline w'}^{\kappa\dag}) 
\overset{\widetilde{f}^{-1*}}{\fleche} H^0(C_i,p_1^*\omega_{\underline w}^{\kappa\dag}) \overset{\frac{1}{p^{n_i}}\Tr_{p_1}}{\fleche} H^0(\mathcal X_{\Sigma}(\underline\eps),\omega_{w}^{\kappa\dag})
\end{eqnarray*}
and also
\begin{eqnarray*} 
C_{cusp}(\underline\eps,\underline w',\kappa_{\mathcal W(w)}) \overset{p_2^*}{\fleche} R\Gamma(C_i\times \mathcal W(w),p_2^*\omega_{\underline w'}^{\kappa^{univ}\dag}) 
\overset{\widetilde{f}^{-1*}}{\fleche} R\Gamma(C_i\times \mathcal W(w),p_1^*\omega_{\underline w}^{\kappa^{univ}\dag}) \\\overset{\frac{1}{p^{n_i}}\Tr_{p_1}}{\fleche} R\Gamma(\mathcal X_{\Sigma}(\underline\eps)\times W(w),\omega_{w}^{\kappa^{univ}\dag}) = C_{cusp}(\underline\eps,w,\kappa_{\mathcal W(w)}),
\end{eqnarray*}
where $n_i$ is an integer defined in \cite{Bijmu} section 2.3 for example\footnote{For us, this integer will not be important as it is usefull to normalize the Hecke operators, and as our Hecke eigensystems are constructed on spaces where $p$ is inverted, this normalisation could be changed (we should change Theorem \ref{thr83} accordingly)}. It is related to the inseparability degree of the projection $p_1$.

\begin{rema}
\begin{enumerate}
\item Unfortunately it is not clear how to define the Hecke operator $U_i^0$ on the neighborhoods $\mathcal X(v)$ as we don't know how the Hasse invariants behaves with quotients... But we will solve this in the end of the paper.
\item 
We can define similarly $U_i^0$ over $\mathcal X(\underline\eps)\times \mathcal W_w$ with $\kappa^{un\dag}$.
\item Thus we can use the different operators $U_i^0$ to improve the radius of convergence in all directions $w_{k,l}^\tau$ with $k > l$. 
\end{enumerate}

\end{rema}

\subsection{A compact operator}

Using the previous construction, we can define a compact operator. Fix $w >0$ and $n$ sufficiently big such that $n-2-\eps > w$. 
Fix also $v$ sufficiently small such that $(\eps,n,v)$ is satisfied. 

%

Define $w' = (w^{'\sigma}_{k,l})_{\sigma, k > l}$ by,
\[ 
w^{'\sigma}_{k,l} = \left\{
\begin{array}{cc}
 w & \text{if } k=l   \\
w +1& \text{ otherwise } 
\end{array}
\right.
\]

\begin{rema}
We could be more precise about the precise values of $w'$ we can choose for what follows (summing over all $i$'s the previous proposition), but the previous will be sufficient.
\end{rema}

Denote by $\underline\eps' < \underline\eps$ the tuple given by Proposition \ref{propepsqc}. Then we have for each $\kappa \in \mathcal W_w$, the following operator,
\[ \prod_i U_i^0 : H^0(\mathcal X^{tor}(\underline\eps'),\omega_{\underline w'}^{\kappa\dag}) \fleche H^0(\mathcal X^{tor}(\underline\eps),\omega_{\underline w}^{\kappa\dag}),\]
and thus the operator,
\[ \prod_i U_i : H^0(\mathcal X^{tor}(\underline\eps),\omega_{\underline w}^{\kappa\dag}) \fleche H^0(\mathcal X^{tor}(\underline\eps'),\omega_{\underline w'}^{\kappa\dag}) \overset{\prod_i U_i^0}{\fleche} H^0(\mathcal X^{tor}(\underline\eps),\omega_{\underline w}^{\kappa\dag}),\]
is compact, as the first map is. Similarly, denote $U_i$ by precomposing $U_i^0$ by the map $H^0(\mathcal X^{tor}(\underline\eps),\omega_{{\underline w}}^{\kappa\dag})
 \subset H^0(\mathcal X^{tor}(\underline\eps),\omega_{\underline w'}^{\kappa\dag})$ of the previous subsection.
The same construction works also over $\mathcal X^{tor}(\underline\eps) \times \mathcal W_w$ with $\kappa^{univ}$.

\section{Classicity results}
\label{sect8}
Let $w$ and $\underline \eps$ such that  $0 < w < 1 - \underline\eps$. Up to reduce $\underline \eps$ this is always possible for some $w<1$. Then the map $\mathcal{IW}_w^{0} \fleche \mathcal X^{tor}(\underline\eps)$ has connected fibers. 

\begin{prop}
\label{prop81}
Let $\kappa$ be a classical weight. We have an exact sequences of sheaves over $\mathcal X_{}^{tor}(\underline\eps)$,
\[ 0 \fleche \omega^{\kappa'} \fleche \omega_w^{\kappa\dag} \overset{d_1}{\fleche} \bigoplus_{\alpha \in \Delta} \omega^{\alpha \cdot \kappa\dag}_w\]
which etale locally is isomorphic to the exact sequence (\ref{seqJones}).
\end{prop}

\begin{proof}
We construct the map $d_1$ as in \cite{AIP} (we don't need the hypothesis on $w$ here). Then we have a sequence,
\[ 0 \fleche V_{\kappa'}\hat{\otimes}\mathcal O_{\mathcal X^{tor}(\underline\eps)} \fleche V^{0,w-an}_{\kappa,L}\hat{\otimes}\mathcal O_{\mathcal X^{tor}(\underline\eps)} \overset{d_1}{\fleche} \bigoplus_{\alpha\in\Delta} V^{0,w-an}_{\alpha\cdot\kappa,L}\hat{\otimes}\mathcal O_{\mathcal X^{tor}(\underline\eps)},\]
which is exact by hypothesis on $w$ by Jones result (hypothesis implies that $N_1 \simeq \mathbb B$ as in \cite{Jones} section 8). Then as this sequence is etale locally isomorphic to the one of the proposition, we get the result.
\end{proof}

\begin{prop}
\label{prop82}
Let $\kappa = (k_{\sigma,j})_{\sigma,1\leq j \leq p_\sigma}$ be a classical weight. The submodule of $M_w^{\kappa\dag}(\mathcal X^{tor}(\underline\eps))$ on which each $U_i$ acts with slope strictly less than 
$\inf\{k_{\sigma,i} - k_{\sigma,i+1} : i < p_\sigma\}$ is contained in $H^0(\mathcal X^{tor}(\underline\eps),\omega^{\kappa'})$.
\end{prop}

\begin{proof}
By the previous proposition, and Proposition \ref{propepsqc} and Proposition \ref{prop78}, the proof is identical to \cite{AIP} Proposition 7.3.1. Indeed, let $f \in M_w^{\kappa\dag}(\mathcal X_{Iw}(\underline\eps))$ on which the $U_i$ acts with the said slope. Using Proposition \ref{prop78}, and that $f$ is finite slope for $U_i$, we can suppose that $\underline w < 1 - \underline\eps$. Thus, by proposition \ref{prop81}, because of the slope, we calculate as in \cite{AIP} that $d_1f = 0$. Thus $f \in H^0(\mathcal X^{tor}(\underline\eps),\omega^{\kappa'})$.
\end{proof}

The second result we need for classicity  is Bijakowski's result, \cite{Bijmu}. For each $\tau \in \mathcal T^+$, denote $A_\tau = \min(p_\tau,q_\tau)$.
\begin{theor}[Bijakowski]
\label{thr83}
Let $\kappa = (k_{\tau,j},\lambda_{\tau,l})_{\tau \in \mathcal T^+,1\leq j \leq p_\tau, 1 \leq l \leq q_\tau}$ be a classical weight and let 
$f \in H^0(\mathcal X^{tor}(\underline\eps),\omega^{\kappa})$. Suppose that $f$ is an eigenvector for the Hecke operators $U_{A_\tau}$ of eigenvalue $\alpha_\tau$ such that,
\[ n_{A_\tau} + v(\alpha_\tau) < \inf_{\tau}(k_{\tau,p_\tau} + \lambda_{\tau,q_\tau}),\]
for each $\tau \in \mathcal T^+$ verifying $A_\tau \neq 0$. Then $f$ is classical.
\end{theor}

\section{Projectiveness of the modules of overconvergent forms}
\label{sect9}
Recall that we work over $\mathcal X$, our Shimura variety with Iwahori level at $p$, and as explained in section \ref{sect63} we have defined for all $n$ and $w \in ]n-1,n-\eps_{0,\tau}[$ (where $\underline\eps_0$ was small enough), a sheaf $\omega_w^{\kappa\dag}$ for all $w$-analytic $\kappa \in \mathcal W$, and a universal one $\omega_w^{\kappa^{univ}\dag}$, both defined on sufficiently small strict neighborhoods of $\mathcal X^{\mu-can} = \mathcal X(\eps = 0)$, the $\mu$-ordinary canonical locus. We have two families of strict neighborhoods of this locus, each having their advantages. In this section, we prove that essentially we have all the advantages (action and compactness of $U = U_p = \prod_i U_i$, the operator of section \ref{sect7}, and vanishing of higher cohomology) on the finite slope part on both kind of strict neighborhoods.
In this section, we suppose that on the strict neighborhoods we consider we have the sheaves $\omega_w^{\kappa^{univ}\dag}$, which means concretely that $\underline \eps$ and $v$ are small enough (smaller than a constant which depends on $w$). Let $\mathcal U \subset \mathcal W(w) \subset \mathcal W$ an open affinoïd such that the universal character $\kappa_{\mathcal U}$ is $w$-analytic.

\begin{prop}
Let $w \leq w'$ and $\underline\eps \geq \underline \eps'$. Then the restriction maps,
\[ C_{cusp}(\underline \eps,w,\kappa_{\mathcal U}) \fleche C_{cusp}(\underline \eps',w',\kappa_{\mathcal U}),\]
are isomorphisms on the finite slope part for $U= \prod_i U_i$. In particular, the finite slope part for $U$ of $C_{cusp}(v,w,\kappa_{\mathcal U})$ and $C_{cusp}(\eps,w,\kappa_{\mathcal U})$ are the same, and thus are their cohomology groups.
\end{prop}

As explained in section \ref{sect7}, it is not clear that $C_{cusp}(v,w,\kappa_{\mathcal U})$ or any of its cohomology group is preserved by $U$. But by proposition \ref{propepsqc}, there exists $N > 0$ an integer (which depends on $v$ à priori) such that $U^N$ preserves $C_{cusp}(v,w,\kappa_{\mathcal U})$. We can see that when $U$ acts on a module $M$, the finite slope part for $U$ of $U^N$ are the same (see for example proof of proposition \ref{lemma93}). We thus define the finite slope part of $C_{cusp}(v,w,\kappa_{\mathcal U})$ for $U$ as the one for $U^N$. It is then a consequence of the equality that the finite slope part of $C_{cusp}(v,w,\kappa_{\mathcal U})$ is actually stable by $U$.

\begin{proof}
Indeed, it is enough to do it for $\eps'$ given by proposition \ref{propepsqc}, $w' = w-1$. We have the factorisation, 
\[H^i(C_{cusp}(\underline \eps',w',\kappa_{\mathcal U})) \overset{\widetilde U}{\fleche} H^i(C_{cusp}(\underline \eps,w,\kappa_{\mathcal U})) \overset{res}{\fleche}H^i(C_{cusp}(\underline \eps',w',\kappa_{\mathcal U}))\] 
Now for a finite slope section $f \in H^i(C_{cusp}(\underline \eps',w',\kappa_{\mathcal U}))$, by definition there exists a non zero polynomial $P$ with $P(0) = 0$ and $P(U)f = f$.
We can extend $f$ to $H^i(C_{cusp}(\underline \eps,w,\kappa_{\mathcal U}))$ by $P(\widetilde U)f$.
In particular, we can find for all $v$, an $\eps$ and $\eps' \leq \eps$ such that,
\[ \mathcal X(\eps') \subset \mathcal X(v) \subset\mathcal X(\eps),\]
and the composed restriction map,
\[ C(\eps,w,\kappa_{\mathcal U}) \fleche C(v,w,\kappa_{\mathcal U}) \fleche C(\eps',w,\kappa_{\mathcal U}),\]
is an isomorphism on the finite slope part, in particular, $C(v,w,\kappa_{\mathcal U})^{fs} = C(\underline \eps,w,\kappa_{\mathcal U})^{fs}$ and thus these spaces are stable by $U$.
\end{proof}

In particular, we get 
\begin{prop}
$C_{cusp}(v,w,\kappa_{\mathcal U})$ has cohomology concentrated in degree zero, and the finite slope part of the cohomology of $C_{cusp}(\eps,w,\kappa_{\mathcal U})$ is concentrated in degree zero.
\end{prop}

\begin{proof}
The first part is appendix Theorem \ref{thrvanishing}. Fix $i$, then we have restriction maps
\[ H^i(C_{cusp}(\underline\eps,w,\kappa_{\mathcal U})) \fleche H^i(C_{cusp}(v,w,\kappa_{\mathcal U})) \fleche H^i(C_{cusp}(\underline \eps',w,\kappa_{\mathcal U})),\]
(for well chosen $\underline\eps,\underline\eps',v$) whose composite is an isomorphism on finite slope parts, and the middle module vanishes for $i >0$.
\end{proof}

According to \cite{urb} section 2.3.10, we can form the alternated Fredholm determinant,
\[ \det(1-XU|C(\eps,w,\kappa_{\mathcal U})).\]
But, because of the results of the previous section, this alternated determinant should actually only be the one in degree 0. Moreover, we will be able to restrict (locally) to the classical construction on an eigenvariety as in \cite{ColBanach,Buz,AIP}. 

For this, fix $\eps,v,w$ and $\mathcal U$ accordingly. By Proposition \ref{pro617}, recall that $C_{cusp}(\eps,w,\kappa_{\mathcal U})$ and $C_{cusp}(v,w,\kappa_{\mathcal U})$ are perfect complexes (in the sense of Urban \cite{urb}), and the latter one can be represented by the projective (in the sense of Buzzard \cite{Buz} or \cite{urb}) 
module in degree 0 
$H^0(\mathcal X(v)\times\mathcal U,\omega_w^{\kappa_{\mathcal U}^{univ\dag}}(-D))$.
The compact operator $U$ acts on $C_{cusp}(\eps,w,\kappa_{\mathcal U})$, but not a priori on $C_{cusp}(v,w,\kappa_{\mathcal U})$, but by Proposition \ref{propepsqc}, there exists an integer $N$, which we fix, such that we have $\eps' < \eps$, inclusions,
\[ \mathcal X(\eps) \subset \mathcal X(v) \subset \mathcal X(\eps'),\]
and $U^N(\mathcal X(\eps)) \subset \mathcal X(\eps')$. In particular $U^N(\mathcal X(v)) \subset \mathcal X(\eps') \subset \mathcal X(v)$.
Thus, $U^N$ is a compact operator on both $C_{cusp}(\eps,w,\kappa_{\mathcal U})$ and $C_{cusp}(v,w,\kappa_{\mathcal U})$. We now need to explain how to construct the Eigenvariety. First, we have three Fredohlm series over $\mathcal O_{\mathcal U}$ $F_{U,\eps}$ and $F_{U^N,\eps}$ of $U$ and $U^N$ acting on $C_{cusp}(\eps,w,\kappa_{\mathcal U})$, and $F_{U^N,v}$ of $U^N$ acting on
$C_{cusp}(v,w,\kappa_{\mathcal U}) =  H^0(\mathcal X(v)\times\mathcal U,\omega_w^{\kappa_{\mathcal U}^{univ\dag}}(-D))$.

First, we need, as in the classical construction, to do things on a specific cover, so choose a slope covering covering for $\mathcal U$ and $F_{U,\eps}$, $(V,h)$ in the sense of definition 2.3.1. of \cite{JN} (this exist, see \cite{JN} Theorem 2.3.2 for example). Over $(V,h)$, we can thus decompose the Fredholm series,\[ F_{U,\eps} = GS, \] 
where $G \in \mathcal O_V[T]$ is a slope $\leq h$ polynomial and $S \in 1+T\mathcal O_V{{T}}$ is an entire series of slopes $> h$.
Accordingly, by \cite{ColBanach} Theorem A4.3/5 (or \cite{JN} Theorem 2.2.2), we have slope decompositions for $U$ of complexe,
 \[ C_{cusp}(\eps,w,\kappa_{\mathcal U}) = C_{cusp}(\eps,w,\kappa_{\mathcal U})_{U,\leq h} \oplus C_{cusp}(\eps,w,\kappa_{\mathcal U})_{U,> h}.\]
\todo{Est-ce possible de trouver un $h$ pour chaqu'un des $M^q$ ?}

\begin{lemma}
\label{lemma93}
This decomposition is a slope $Nh$ decomposition for $U^N$ acting on this module, and it induces a slope $Nh$ factorisation of 
\[ F_{U^N,\eps} = G'S'.\]
\end{lemma}

\begin{proof}
We can work on a single module, say $M$ and denote the associated decompositions associated to the slope decomposition of $F_{U,\eps} = GS$,
\[ M = M_{U,\leq h} \oplus M_{U,> h}.\]
As if $Q(T)$ is a slope $> Nh$ polynomial, the polynomial $Q(T^N)$ is of slope $> h$, we get
that $U^N$ is invertible on $M_{u,>h}$. Now let \[N = \{ m \in M | \exists Q \text { of slopes } \leq Nh, \text{ such that } Q^*(U^N)m = 0\}.\]
This is clearly a submodule of $M$, and has if $Q(T)$ has slopes $\leq Nh$, $Q(T^N)$ has slopes $\leq h$, we have $P \subset M_{U,\leq h}$.
We claim that $P = M_{U,\leq h}$. Denote the $\mathcal O_V$-module
\[ Q = M_{U,\leq h}/P.\]
Fix $x$ a point of $V$, and let $v \in (M_{U,\leq h})_x = (M_x)_{U,\leq h}$ (which is easily seen to be true, or see e.g. \cite{JN} Theorem 2.2.13), thus if we denote $N(v)$ the sub-$k(x)$-vector 
space generated by $v$ and its images by $u$ and its powers, $N(v)$ is finite dimensional (say of dimension $r$).
Denote $\mu_{u,v}$ and $\chi_{u,v}$ the minimal and characteristic polynomials of $u$ on $N(v)$. As there exists $Q$ of slopes $\leq h$ such that $Q^*(u)$ kills $u$, $\mu_{u,v}^*$, and thus $\chi_{u,v}^*$ have $\leq h$ slopes. Up to extending scalars, there is a basis of $N(v)$ such that the matrix of $u$ on $N(v)$ is given by
\[
\left(
\begin{array}{cccc}
\lambda_1  &   &  &\star \\
  &  \lambda_2 &&   \\
  &   &   \ddots & \\
0  &	&	& \lambda_r
\end{array}
\right),
\]
and we can thus calculate characteristic polynomial of $U^N$ : it is of slope $\leq Nh$. By the theorem of Cayley-Hamilton $v \in P$. Thus $Q\widehat{\otimes} k(x)$ is zero, and by Nakayama, $Q=0$. In particular we have that $M_{U,\leq h} \oplus M_{U,> h}$ is a slope $Nh$ decomposition for $U^N$, which is functorial with respect to localisations $\Spm(k(x))\fleche V$ as it comes from the slope decomposition of $F_{U^N,h}$, thus by \cite{JN} Theorem 2.2.13 it induces a slope $Nh$ decomposition
\[ F_{U^N,\eps} = G'S'.\]
\end{proof}

\begin{lemma}
The restriction map
\[ res : C_{cusp}(\eps,w,\kappa_{\mathcal U}) \fleche C_{cusp}(v,w,\kappa_{\mathcal U}),\]
induces an equality
\[ F_{U^N,\eps} = F_{U^N,v}.\] 
In particular, over $(V,h)$ we have a decomposition
\[C_{cusp}(v,w,\kappa_{\mathcal U}) = C_{cusp}(v,w,\kappa_{\mathcal U})_{U^N,\leq Nh} \oplus C_{cusp}(v,w,\kappa_{\mathcal U}))_{U,> Nh},\]
such that, res induces an isomorphism over $V$,
\[C_{cusp}(\eps,w,\kappa_{\mathcal U})_{U,\leq h} = C_{cusp}(v,w,\kappa_{\mathcal U})_{U^N,\leq Nh}.\]
\end{lemma}

\begin{proof}
The first part is because we a diagram
\begin{center}
\begin{tikzpicture}[description/.style={fill=white,inner sep=2pt}] 
\matrix (m) [matrix of math nodes, row sep=3em, column sep=2.5em, text height=1.5ex, text depth=0.25ex] at (0,0)
{ 
C_{cusp}(\eps) & C_{cusp}(v) 
\\
C_{cusp}(\eps) & C_{cusp}(v) 
\\
};
\path[->,font=\scriptsize] 
(m-1-1) edge node[auto] {$res$} (m-1-2)
(m-2-1) edge node[auto] {$res$} (m-2-2)

(m-1-1) edge node[auto] {$U^N$} (m-2-1)
(m-1-2) edge node[auto] {$U^N$} (m-2-2)
(m-1-2) edge node[auto] {$U^N$} (m-2-1)
(m-1-2) edge node[auto] {$U^N$} (m-2-1)
;
\end{tikzpicture}
\end{center}
Where the shortened notations speak for themselves, and thus $U^N : C_{cusp}(v) \fleche C_{cusp}(\eps)$ is a link in the sense of \cite{Buz}. Thus the two power series are equal.
\todo{Maybe represent $U^N$ by a map of complex on each $M^q$}
The rest follows by lemma \ref{lemma93}
\end{proof}

In particular, for  each $(w,\underline\eps)$, for $w$ big enough and $\underline\eps$ small enough such that for all $\tau$, $w \in ]n-1,n-\eps_\tau[$  (which determines a unique integer $n$), ($v$ doesn't play a role and can always be chosen so that $X(\eps) \supset X(v)$, which we do here), we can construct an Eigenvariety for the tuple
\[ (\mathcal O_{\mathcal W(w)},C_{cusp}(\eps,w,\kappa_{\mathcal W(w)}),\mathcal H^N\otimes\mathcal A(p),\prod_{\pi,i}U_{\pi,i}),\]
as if $C_{cusp}(\eps,w,\kappa_{\mathcal W(w)})$ were one projective module. Indeed, locally this can be replace $\pi(v)_*\omega_w^{\kappa_{\mathcal W(w)}\dag}(-D)$ where
\[\pi(v): X(v) \times \mathcal W(w) \fleche \mathcal W(w),\]
and this $\mathcal O_{\mathcal W(w)}$-module is indeed projective, and its finite slope part inherits the action of $U = \prod_{\pi,i}U_{\pi,i}$, and this constructions glue together.
Moreover, we have natural maps between them when $(w,\underline\eps),(w',\underline\eps')$ satisfies $w' \geq w$ and $\underline\eps' \leq \underline\eps$.

%
%

This is the main ingredient in all the constructions of Eigenvarieties. In particular, we get,

\begin{theor}
\label{thrHecke}
Let $p$ be a prime. Fix $S_p$ a set of primes over $p$ (see section 2) unramified in $\mathcal D$ and $(K_J,J)$ a type\footnote{Here by type we only mean, as in \cite{Her3}, a compact open subgroup $K_J$ of $G(\mathbb A_f^p)$ together with a finite dimensional representation} 
outside $S_p$, $K \subset \Ker J$, and $S^p$ the set of places away from $S_p$ where $K$ is not maximal. 
There exists an equidimensionnal rigid analytic space $\mathcal E_{S_p}$, together with a locally finite map,
\[ \mathcal E_{S_p} \overset{w}{\fleche} \mathcal W_{S_p},\]
and a Zariski dense subset $\mathcal Z$, such that for any $\kappa \in \mathcal W(L)$, $w^{-1}(\kappa)$ is in bijection with the eigensystems for $\mathcal H^S\otimes \mathcal A(p)$
acting on the space of overconvergent, locally analytic, modular forms of weight $\kappa$, type $(K_J,J)$, and finite slope for $\mathcal A(p)$.
Moreover, $w(\mathcal Z)$ consists of classical weights and $z \in \mathcal Z$ is an Hecke eigensystem for a \textit{classical} modular form of weight $w(z)$.
\end{theor}

\begin{proof}
The construction is classical as soon as we have the previous datum, see \cite{ColBanach} and \cite{Buz}. Just remark that cutting in the datum the piece of type $(K_J,J)$ is possible as we are in characteristic zero (see \cite{Her3}, Proposition 9.13). The equidimensionnality results follows from the fact that we locally reduce to a single projective module $H^0(\mathcal X(\underline \eps),\omega^{\kappa_U\dag}_w(-D))_{U^N,\leq h}$ and \cite{Che1} Lemme 6.2.10. The set $\mathcal Z$ is the set of points of $\mathcal E$ which map to a point in $\mathcal W$ satisfiying the hypotheses of proposition \ref{prop82} and theorem \ref{thr83}. This is (Zariski) dense by \cite{Che1} Corollaire 6.4.4. and using that every open of $\mathcal W$ contains a point satisfying the previous hypothesis. 
\end{proof}

\begin{rema}
We will always consider the space $\mathcal E_{S_p}$ with its reduced structure (see \cite{CheJL} section 3.6). But in turns out that $\mathcal E_{S_p}$ is almost always automatically reduced with the structure given by $\mathcal H \otimes \mathcal A(p)$. For the eigencurve this is \cite{CM}, Proposition 7.4.5., in the quaternionic case see \cite{CheJL} Proposition 4.8, and \cite{BC2}, section 7.3.6 for a unitary group, compact at infinity. In the next section, we will prove that in the case of $U(2,1)$ this is also true.
\end{rema}

\section{Some complements for Picard modular forms (especially when $p=2$)}
\label{sect10}

In a previous article (see \cite{Her3}), we constructed the Eigenvariety $\mathcal E$ for $U(2,1)_{E/\QQ}$ where $E$ is a quadratic imaginary field, under the hypothesis that $p$ was inert (if $p$ splits see \cite{Brasca}) so that the ordinary locus is empty, but also that $p \neq 2$, so that we can apply the main theorem of \cite{Her2} on the canonical filtration.
Theorem \ref{thrHecke} extends this construction also for $p=2$, and for $E/F$ a general CM-extension (but we only consider $F= \QQ$ in this section). Classical points on $\mathcal E$ correspond to classical forms for $(G)U(2,1)$ with \textit{classical} weights given by $\kappa = (k_1 \geq k_2,k_3) \in \ZZ^3$. The corresponding character in $\mathcal W$ is given by
\[ (x,y) \in \mathcal O^\times \times \mathcal O^1 \longmapsto \tau(x)^{k_1}\tau(y)^{k_2}\sigma\tau(x)^{k_3}.\]

\begin{prop}
\label{propred}
The curve $\mathcal E$ is reduced. This remains true if we had fixed the second weight to $k_2 \in \ZZ$ on the weight space.
\end{prop}

\begin{proof}
We will use \cite{CheJL} Proposition 3.9, and we only need to check assumption (SSG) there.
Thus, we need to find sufficiently many classical points $k \in \mathcal W$ for which the module $M_k^{class} \cap M_k^{\dag,\leq \alpha}$ is semi-simple as an 
$\mathcal H^N\otimes \mathcal A(p)$-module. We know already that the space of \textit{cuspidal} forms for a group $G$ is semi-simple for the action of $\mathcal H^N$ (spherical Hecke operators being auto-adjoint). Thus, we need to treat the action of $\mathcal A(p)$. But the action on an autormorphic form $\pi$ of $\mathcal A(p)$ determines its refinements. Thus, we only need to prove that we can assure that these refinements are distincts, leading that the action of $\mathcal A(p)$ on $\pi_p^{I}$ will be semi-simple ($I$ is an Iwahori subgroup). Let $k = (k_1,k_2,k_3) \in \mathcal W$ be a classical weight. As the space $H^0(X,\omega^k(-D))$ is finite dimensionnal, there is a finite number of classical points $f$ in $\mathcal E$ mapping to $k$ (and for varying $k$ these are strongly Zariski-dense in $\mathcal E$).
But the slopes of Hecke operators at $p$ are locally constant, thus for each of these points we can find an open $U_f$ (intersecting every component of $\mathcal E$ at $f$) on which the slope is actually constant.
Taking the intersection of the image of $U_f$ by $\pi$ in $\mathcal W$, we can find an open $V \ni k$ in $\mathcal W$ and for which every classical point $k' \in V$ and every classical $f'$ in the fiber of $k'$ has slopes equal the same as the one of some classical $f$ in the fiber of $k$. But the refinements are given in terms of eigenvalues of Frobenius by (see \cite{Her3} section 10.6 and proofs of propositions 11.1 and 11.2)
\[ (p^{-k_1-k_3}F_1,1,p^{k_1-k_3}F_1^{-1}) \quad \text{if $p$ is inert},\]
\[ (p^{1-k_3}F_1, p^{k_2-1} F_2, p^{k_1}F_3) \quad \text{if $p$ splits},\]
where $F_i \in \mathcal O(\mathcal E)$ corresponds to some operators in $\mathcal A(p)$. In particular, as the slopes of $F_i$ are constant on $V \ni k$, for any $k' \in V$ with sufficiently regular weights, the three Frobenius eigenvalues are distincts, thus as are the possible refinements. In particular for those $k'$ (which are Zariski dense in $\mathcal W$) the action of $\mathcal H\otimes \mathcal A(p)$ is semi-simple on classical forms in $M_{k'}^{class}$. The same proof works if $k_2$ is fixed.
\end{proof}

\begin{rema}
\label{remaCheJL}
\begin{enumerate}
\item In particular, by \cite{CheJL} this proves that for a classical $k_2 \in \ZZ$, $\mathcal E'$ given by the full eigencurve $\mathcal E$, base changed over 
\[ \mathcal W_{\mathcal O^\times} \hookrightarrow \mathcal W, \quad (k_1,k_3) \mapsto (k_1,k_2,k_3),\]
and the curve constructed as in the previous section, over $\mathcal W_{\mathcal O^\times}$ with a fixed value for $k_2$ coincide and are reduced.
\item Obviously, the same result where we would suppose $k_1 = k_3$ would not be true anymore as it could be that there isn't enough classical semi-simple points.
\end{enumerate}

\end{rema}

In \cite{Her3} (Theorem 1.3), we proved the following theorem,
\begin{theor}
\label{thrBK}
Let $E/\QQ$ be a quadratic extension, and 
\[ \chi : \mathbb A_E^\times/E^\times \fleche \CC^\times,\]
a algebraic Hecke character. We suppose $\chi$ polarized (i.e. $\chi^\perp := (\chi \circ c)^{-1} = \chi |.|^{-1}$
where $c$ is the complex conjugation on $E$). Let $p$ be a prime such that $p$ is unramified in $E$ and $p \notdivides \Cond(\chi)$, and $p\neq 2$ if $p$ is inert in $E$. Let
\[ \chi_p : G_E \fleche \overline{\QQ_p}^\times,\]
be its $p$-adic realisation. Then, if
$\ord_{s=0} L(\chi,s)$
is even and non-zero, the Bloch-Kato Selmer group $H^1_f(E,\chi_p)$ is non-zero.
\end{theor}

This results (actually a more general version of it) was almost entirely already proved by Rubin (\cite{Rub}) at least for CM elliptic curves when $p \neq 2$ (and $p \neq 3$ for $E = \QQ(i\sqrt{3})$).
In particular as $2$ is inert in $\QQ(i\sqrt{3})$ (and $3$ is ramified in this case), it does not prove anything new for inert primes (only for the split ones).

Fortunately, with Theorem \ref{thrHecke}, we will be able to remove the hypothesis $p \neq 2$ if inert. Moreover, we can also remove the hypothesis $p \notdivides \Cond(\chi)$
(as long as $p$ stays unramified in $E$). To do this we focus from now on to the case of $U(2,1)_{E/\QQ}$, and we will define sheaves of forms with nebentypus.

\subsection{A remark on $p = 2$}

To construct an integral model for the Picard modular surface, it is needed to choose a lattice for the group $(G)U(2,1)$, as it appeared in $\mathcal D$ in section 3. 
We do as we did in \cite{Her3} and choose the lattice 
$L = \mathcal O_E^3 \subset E^3$, stable for the form of matrix (used to define $(G)U(2,1)$) in the canonical basis given by,
\[
\psi = \left(
\begin{array}{ccc}
  &   &  1 \\
  & 1  &   \\
1  &   &   
\end{array}
\right).
\]
There is another natural choice, which would be the same lattice but the form
\[
\psi' = \left(
\begin{array}{ccc}
1  &   &   \\
  & 1  &   \\
  &   &   -1
\end{array}
\right).\]
These two forms are isomorphic over $\ZZ[1/2]$ but not modulo $2$. Moreover, see \cite{Bellthese} Section 3.1, any abelian scheme of type (2,1) $A/S$ will have a polarized Tate module $(T_\ell(A),q)$, together with the Weil pairing induced by the polarisation isomorphic either to
$(\mathcal O^3,J)$ or $(\mathcal O^3,J')$. Any of these form would give an integral model for the Picard modular surface, not isomorphic modulo $2$, and we choose $\psi$, the first one, to define $U(2,1)_{E/\QQ,\psi}$ over $\ZZ$ as in section \ref{sect3}. Apart to construct the Eigenvariety, this choice (for which the construction of the Eigenvariety can be checked to be independant afterwards, even if we don't need this result) will not appear in this section as we work in characteristic zero.

\subsection{Removing the hypothesis $p \notdivides \Cond(\chi)$.}
Recall that in \cite{Her3} Section 10, following \cite{BC1}, we introduced a type $(K_J,J)$ for $J = \Cond(\chi)$. Fix an auxiliary level $K^p \subset (\Ker J)^p$, and consider $\mathcal X_0(p^n)^{tor}/\Spm(K)$ the (rigid and compactified) Picard variety of Iwahori level $p^n$, over some $p$ adic field $K$. It is the analytic space of $X_0(p^n)^{tor}$ which away from the boundary its $S$-points parametrizes
\[ (A,\iota, \lambda,\eta, H_1 \subset H_2),\]
where \begin{itemize}
\item $A \fleche S$ is an abelian scheme of genus 3
\item $\iota : \mathcal O_E \hookrightarrow End_S(A)$ is a CM-structure of signature (2,1), i.e.
\[ \omega_A = \omega_{A,\tau} \oplus \omega_{A,\overline\tau}, \quad \omega_{A,\sigma} = \{ w \in \omega_A | \iota(x) w = \tau(x) w,\]
with $\omega_{A,\tau}$ and $\omega_{A,\overline\tau}$ are respectivelly locally free of rank $2$ and $1$, and where $\tau : E \fleche S$ is the canonical morphism and $\overline\tau$ its conjugate.
\item $\lambda : A \fleche {^t}A$ is a polarisation for which the rosatti involution on $\iota(x)$ coincides with $\iota(\overline x)$
\item $\eta$ is a level-$K^p$-structure,
\item $H_1 \subset H_2 \subset A[p^n]$ is a filtration by cyclic $\mathcal O_E\otimes \ZZ_p/p^n\ZZ_p$-modules such that $H_2^\perp = H_1$.
\end{itemize}
This is exactly the rigid space introduced in section \ref{sect5}.

The subgroups $H_1,H_2$ extends to $\mathcal X_0(p^n)^{tor}$, and we can also extend the polarisation of $H_2/H_1$ to the boundary. We will distinguish the cases $p$ inert (AU) and $p$ split (AL) in $E$.

In case (AL), i.e. $p = v\overline v$ is split, then $A[p^n] \simeq G^+ \times G^-$ (with $G^- = (G^+)^D$ and $\lambda$ exchange the two factors), and we can suppose that $G = G^+$, say,  is of dimension 2 and height $p^{3n}$. Under this decomposition, $H_i = H_i^+ \times H_i^-$ and $H_i^+$ is a cyclic rank $p^{in}$-subgroup of $G^+$ and $H_1^- = (H_2^+)^\perp = (G^+/H_2^+)^D \subset G^-$.
In this case \[\mathcal X_0^+(p^n) = \Isom(H_1^+,\ZZ/p^n\ZZ) \times \Isom(H_2^+/H_1^+,\ZZ/p^n\ZZ)\times \Isom(G^+/H_2^+,\ZZ/p^n\ZZ).\]
It is a $T_n = ((\ZZ/p^n\ZZ)^\times)^3$-etale torsor. Remark that $H_2^+$ is the canonical subgroup in this case.
In case (AL) we can also introduce a second space. Using the previous notation, denote by $\mathcal X_{P_n}^{tor}$ the analytic space associated to a toroïdal compactification of the 
following moduli space $X_{P_n}^{tor}$ over $\Spec(K)$. A $S$-point of $X_{P_n}^{tor}$ is a tuple $(A,\iota,\lambda,\eta,H_1^p,H_2^p,H)$ such that 
$(A,\iota,\lambda,\eta,H_1^p,H_2^p)$ is a $S$-point of $X_0(p)$ (Iwahori level, i.e. $H_1^p \subset H_2^p \subset G^+[p]$) together with a subgroup $H \subset G[p^n]$ locally isomorphic to $(\ZZ/p^n\ZZ)^2$ and $H[p] = H_2^p$. It is the Shimura variety of level $P_n \cap I(p)$ where $I(p)$ is the Iwahori subgroup of $GL_3(\ZZ_v)$ and $P_n$ is the subgroup of matrices of the form
\[ 
\left(
\begin{array}{ccc}
\star  & \star  & \star  \\
 \star & \star  & \star  \\
  &   &   \star
\end{array}
\right) \pmod{p^n}.
\]
In particular we have a map $\mathcal X_0(p^n) \fleche \mathcal X_{P_n}$.

In case (AU), i.e. $p$ inert, denote
\[ \mathcal X_0^+(p^n) = \Isom(H_1,\mathcal O/p^n\mathcal O) \times \Isom_{pol}(H_2/H_1,\mathcal O/p^n\mathcal O).\]
This is a $T_n = (\mathcal O/p^n\mathcal O)^\times \times (\mathcal O/p^n\mathcal O)^1$-etale torsor.

In both cases, if $\pi : \mathcal X_0^+(p^n)^{tor} \fleche \mathcal X_0(p^n)^{tor}$ and $\chi : T_n \fleche K^\times$ is a character, we can consider $\mathcal O_{\mathcal X_0(p^n)^{tor}}(\chi)$ to be the subsheaf of $\pi_*\mathcal O_{\mathcal X_0^+(p^n)^{tor}}$ of sections which vary like $\chi$. This is an invertible sheaf on $\mathcal X_0(p^n)^{tor}$. 

\begin{defin}
For all classical weight $\kappa$, we can consider the sheaf,
\[ \omega^\kappa(\chi) := \omega^\kappa\otimes_{\mathcal O_{\mathcal X_0(p^n)^{tor}}} \mathcal O_{\mathcal X_0(p^n)^{tor}}(\chi),\]
which is a locally free sheaf on $\mathcal X_0(p^n)^{tor}$, whose global sections are (classical) Picard modular form of weight $\kappa$ and nebentypus $\chi$.
Similarly,
\[ H^0(\mathcal X_0(p^n)^{tor},\omega^\kappa(\chi)(-D)),\]
is the set of cuspidal ones.
\end{defin}

\begin{prop}
\label{chiclass}
There is a natural injection \[\omega^\kappa(\chi) \hookrightarrow \omega^{\kappa'\chi\dag}_w,\]
for all $w \in ]n-1,n-\eps_\tau[$ and $\kappa\chi$ the product of the character $\kappa$ with the character 
\[ T(\ZZ_p) \fleche T_n \overset{\chi}{\fleche} K^\times,\]
which we still denote $\chi$.
\end{prop}

\begin{proof}
Indeed, a section $f$ of $\omega^\kappa(\chi)$ is a law which associate to $(A,x,w)$ where $A \in \mathcal X^{tor}(K)$, $x$ is a level $X_0^+(p^n)$-structure and $w$ an isomorphism
\[ St_{\mathcal O_E} \otimes \mathcal O_K \simeq \omega_A,\]
an element $f(A,x,w) \in \mathbb A^1(K)$, which moreover satisfies,
\[ f(A,tx,zw) = \chi(t)\kappa'(z)f(A,x,w).\]
In particular, this defines by restriction a section $g$ of $\mathcal{IW}_w^{0,+}$ which satisfies, for the induced action of $T(\mathcal O)$ on $\mathcal{IW}_w^{0,+}$ which sends $(A,x,w)$ to $(A,\overline t x, t w)$, such that $g(t i) = \chi(\overline t)\kappa'(t)g(i)$. Thus $g$ is a section of $\omega^{\kappa'\chi\dag}$.
\end{proof}

Over $\mathcal X_{P_n}^{tor}$ we also have a $I_{GL_2}(p)(\ZZ/p^n\ZZ)\times(\ZZ/p^n\ZZ)^\times$-torsor (where $I_{GL_2}(p)$ is the Iwahori subgroup of $GL_2(\ZZ_p)$), given by
\[ \Isom_{mod p}(H,(\ZZ/p^n\ZZ)^2) \times \Isom(G^+/H,\ZZ/p^n\ZZ),\]
where $mod p$ means that an isomorphism $\phi$ induces an isomorphism of $H^1_p$ inside $p^{n-1}\ZZ/p^n\ZZ e_1$.
Thus, for $\chi'$ a character of $I_{GL_2}(p)(\ZZ/p^n\ZZ)\times(\ZZ/p^n\ZZ)^\times$, i.e. of the form $(\chi_1 \circ \det,\chi_3)$, we have an invertible sheaf $\mathcal O(\chi')$ on $\mathcal X_{P_n}^{tor}$ and thus a sheaf $\omega^\kappa(\chi')$. The sheaf $\omega^\kappa(\chi)$ on $\mathcal X_0(p^n)^{tor}$ descend to $\mathcal X_{P_n}^{tor}$ if and only if $\chi = (\chi_1,\chi_1,\chi_3)$ and coincide with $\omega^\kappa(\chi')$ with $\chi' = (\chi_1\circ\det,\chi_3)$.

\begin{prop}
In the split case, if $v < \frac{1}{2p^{n-1}}$, the canonical subgroup induces an isomorphism,
\[ \mathcal X_0(p)^{tor}(v) \fleche \mathcal X_{P_n}^{tor}(v).\]
Here we really mean the $\mu$-canonical locus (and not the full $\mu$-ordinary locus).
\end{prop}

\begin{rema}
If $p > 2$, is inert, the results of \cite{Her2} give, for $v < \frac{1}{4p^{n-1}}$, an isomorphism $\mathcal X(v) \fleche \mathcal X_0(p^n)(v)$.
\end{rema}

Recall the following result of Rogawski (see \cite{BC2} section 6.9.6 and \cite{Bellnofam} section 2.7). Fix first a Hecke character $\mu$ as in \cite{BC2} Lemma 6.9.2(iii).

\begin{theor}[Rogawski]
Suppose $\ord_{s=0} L(\chi,s)$ is even and non zero. Then there exists a representation $\pi^n$, automorphic for $U(2,1)$ and cuspidal such that for every prime $x$ split in $E$,
\[ L(\pi^n_x) = \mu|.|^{-\frac{1}{2}}(L(\overline\chi\oplus 1 \oplus |.|).\]
\end{theor}

\begin{rema}
This representation $\pi^n$ is slightly different from the one of \cite{BC1} or \cite{Her3}, it is a twist of the latter by $L(\overline\chi)\mu|.|^{-\frac{1}{2}}$.
\end{rema}

\begin{prop}
\label{propkappa}
Suppose $\ord_{s=0} L(\chi,s)$ is even and non zero. Denote $n_0 = v_p(\Cond(\chi))$. 
Denote by $\chi'$ if $p$ is split the character $(1,1,\chi_p)$ and if $p$ is inert the character $(\chi_p,1)$ of 
$(\mathcal O/p^n\mathcal O)^\times \times (\mathcal O/p^n\mathcal O)^1$. Denote also $\kappa$ the classical weight corresponding to
\[ (a-1,0,1) \in \ZZ^3_{dom}. \]
Then the Hecke eigensystem (away from $p\Cond(\chi)$) of $\pi^n$ appears in $H^0(\mathcal X(\eps),\omega^{\kappa'\chi'\dag}_w(-D))$ for all $n \geq n_0$ and $w \in ]n-1,n-\eps[$ for $\eps$ small 
enough.
\end{prop}

\begin{proof}
Indeed we checked that $\pi^n$ contributes to the coherent first cohomology group in \cite{Her3} Proposition D.2. More precisely we checked that its restriction to $SU(2,1)$ appears with $K$-type corresponding to $\kappa$ restricted to $SU(2,1)$. As $\pi^n$ is a twist of the representation denoted $\pi^n(\chi)$ in \cite{Her3} by 
$\overline\chi \mu|.|^{-1/2}$, which is algebraic, we can calculate its algebraic weight $\kappa$ and check that $\kappa = (a-1,0,-1)$\footnote{We could also argue directly as in \cite{Her3} relating $\kappa$ to the Hodge-Tate weights of $\rho_{\pi^n}$ on the Eigenvariety $\mathcal E$}.
Moreover Bellaiche-Chenevier (\cite{BC1} Proposition 4.2) proved that $\pi^n(\chi)\otimes \chi_0^{-1}$ was of a certain type $(K_J,J)$  at ramified primes for $\chi$. As 
$\overline{\chi} = \overline{\chi_0}|.|^{\frac{1}{2}}=\chi_0^{-1}|.|^{\frac{1}{2}}$ and both $|.|$ and $\mu$ are unramified at $p$, we deduce that the twist $\pi^n$ is of the same type as $\pi^n(\chi)\otimes \chi_0^{-1}$ (which is obviously trivial if $\chi$ is unramified at $p$). Thus $\pi^n$ is of nebentype $\chi'$, and we deduce the previous result from proposition \ref{chiclass}.
\end{proof}

We need to take care of the action at $p$ of the Iwahori algebra $\mathcal A(p)$. This is well known in the case of $GL_2$ (see \cite{ColHigh}). Denote the higher-Iwahori subgroup
\[ I_n^+ = 
\left(
\begin{array}{ccc}
1+p^n\mathcal O  & \star  & \star  \\
p^n\mathcal O  & 1+p^n\mathcal O  &\star   \\
p^n\mathcal O  &p^n\mathcal O   &   1+p^n\mathcal O
\end{array}
\right) \cap G(\QQ_p),
\]
where $G(\QQ_p) = \GL_3(\QQ_p)$ if $p$ is split, and $U(2,1)(\QQ_p) = U(3)(\QQ_p)$ otherwise. We 
could do everything for $GU(2,1)$ or $\GL_3 \times \GL_1$ (if $p$ splits) but it doesn't change anything.
$I_n^+$ has a natural Iwahori decomposition $I_n^+ = \overline{N}_n \times T_n^+ \times N_n$ (and $N_n = N$), and thus if we denote $\Sigma^+$ the elements of the form 
\[
\left(
\begin{array}{ccc}
 p^{a_1} &   &   \\
  & p^{a_2}  &   \\
  &   &   p^{a_3}
\end{array}
\right) \quad \text{with } a_1 \geq a_2 \geq a_3,
\]
if $p$ splits, and 
\[
\left(
\begin{array}{ccc}
 p^{a_1} &   &   \\
  & p^{a_2}  &   \\
  &   &   p^{-a_1}
\end{array}
\right) \quad \text{with } a_1 \geq a_2,
\] if $p$ is inert. Denote by $\Sigma$ the group generated by $\Sigma^+$ and their inverse. \todo{Doit-on ajouter tout $T/T_n^+$? En même temps quel interet vraiment...? On sait que $\chi(p)$ est different de $1$ et $p^{-1}$}

\begin{prop} Denote by $\mathcal A_n^{+,0}(p)$ the sub-algebra of $\mathcal H(G(\QQ_p)//I_n^+)$ generated by the double class caracteristic functions
\[ 1_{I_n^+aI_n^+}, \quad a \in \Sigma^+.\]
$\mathcal A_n^{+,0}(p)$ is commutative. Denote by $\mathcal A_n^+(p)$ the algebra generated over $\QQ_p$ by $\mathcal A_n^{+,0}(p)$ and the inverse of the elements $1_{I_n^+aI_n^+}$. It is canonically isomorphic to $\Sigma$ and thus to $\mathcal A(p)$.
\end{prop}

\begin{proof}
$\mathcal A_n^+(p)$ is commutative by \cite{Casselman} Lemma 4.1.5.
\end{proof}

\begin{rema}
The canonical isomorphism $\Sigma \fleche \mathcal A_n^+(p)$ sends $a \in \Sigma^+$ to the corresponding double class, but this is not true for all $a \in \Sigma$, just like the case of $\mathcal A(p)$. The double class are not invertible in general (if $n > 1$ at least, see \cite{OggEv} Lemma 2 for (new) modular forms, but this is true if $n=1$, \cite{Vig}).
\end{rema}

There is thus an Hecke operator acting on $\mathcal X_0^+(p^n)$ corresponding to the double class $1_{I_n^+aI_n^+}$ where in the inert case 
\[
a = \left(
\begin{array}{ccc}
 p &   &   \\
  & 1  &   \\
  &   &  p^{-1}
\end{array}
\right)
\] 
and in the split case,
\[
a = \left(
\begin{array}{ccc}
 p &   &   \\
  & 1  &   \\
  &   &  1
\end{array}
\right) \quad \text{or} \quad a = \left(
\begin{array}{ccc}
 p &   &   \\
  & p  &   \\
  &   &  1
\end{array}
\right).
\]
We call respectively $U^n_p$, $U^n_{p,1},U^n_{p,2}$ the corresponding operators. These operators can be defined on the moduli problem $\mathcal X_0^+(p^n)$ and commutes with their counterparts on $\mathcal X_0(p)=:\mathcal X$ (see for example \cite{PS1} section 8.2), in the sense that for one of these operators, say $g$, if we denote the correspondance $C$ and $C_n = C \times_{\mathcal X} \mathcal X_0^+(p^n)$, with $\pi_g^n$ and $\pi_g$ the universal isogeny on $C_n$ and $C$, we thus have commutatives diagrams,
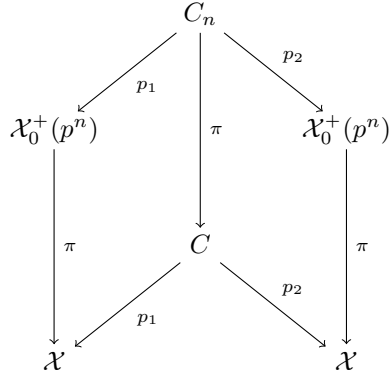
\begin{center}
\begin{tikzpicture}[description/.style={fill=white,inner sep=2pt}] 
\matrix (m) [matrix of math nodes, row sep=3em, column sep=2.5em, text height=1.5ex, text depth=0.25ex] at (0,0)
{ 
 &  C_n  & \\
 \mathcal X_0^+(p^n) & & \mathcal X_0^+(p^n)\\
 & C & \\
 \mathcal X & & \mathcal X\\
};
\path[->,font=\scriptsize] 
(m-1-2) edge node[auto] {$p_2$} (m-2-3)
(m-1-2) edge node[auto] {$p_1$} (m-2-1)
(m-1-2) edge node[auto] {$\pi$} (m-3-2)
(m-3-2) edge node[auto] {$p_1$} (m-4-1)
(m-3-2) edge node[auto] {$p_2$} (m-4-3)
(m-2-1) edge node[auto] {$\pi$} (m-4-1)
(m-2-3) edge node[auto] {$\pi$} (m-4-3)
;
\end{tikzpicture}
\end{center}
and a commutative diagram,
\begin{center}
\begin{tikzpicture}[description/.style={fill=white,inner sep=2pt}] 
\matrix (m) [matrix of math nodes, row sep=3em, column sep=2.5em, text height=1.5ex, text depth=0.25ex] at (0,0)
{ 
\mathcal IW^+_{\mathcal X_0^+(p^n)} \times C_n & & \mathcal IW^+_{\mathcal X_0^+(p^n)} \times C_n\\
\mathcal IW^+_{\mathcal X}\times C & &\mathcal IW^+_{\mathcal X}\times C \\
};
\path[->,font=\scriptsize] 
(m-1-1) edge node[auto] {$\pi_g^n$} (m-1-3)
(m-2-1) edge node[auto] {$\pi_g$} (m-2-3)
(m-1-1) edge node[auto] {$p$} (m-2-1)
(m-1-3) edge node[auto] {$p$} (m-2-3)
;
\end{tikzpicture}
\end{center}

The normalisation of the maps $\pi^n_g$ and $\pi_g$ can be done the same way, and we thus deduce that the operators $U_{p,\star}$ in level $\mathcal X$ and $U_{p,\star}^n$ 
in level $ \mathcal X_0^+(p^n)$ commutes with the pullback by $\pi$ (i.e. $U_{p,\star}^n(\pi^*f) = \pi^*(U_{p,\star}f)$). Thus, these operators defined on $\omega_w^{\kappa\dag}$ 
for any $\kappa \in \mathcal W$ $w$-analytic (with $w \in ]n-1,n-\eps_0[$) are the same once we identify (invariant by $T_n$) sections on some small neighborhood $\mathcal 
X_0^+(p^n)(v)$ of $\omega_w^{\kappa^0\dag}(-\kappa_{|T_n})$ with sections of $\omega_w^{\kappa\dag}$ on some small neighborhood $\mathcal X(v)$.

In particular to understand the action of $\mathcal A(p)$ on the forms corresponding to $\pi^n$ which appears in $H^0(\mathcal X(v),\omega_w^{\kappa\chi'\dag})$ for $v$ 
small enough, we need to understand the action of $\mathcal A_n^+(p)$ on $\pi^n_p$. 

\begin{defin}
If $\pi$ is a representation of $G(\QQ_p)$, denote by
\[ (\pi^{I_n^+})^{fs} := 1_{I_n^+aI_n^+}(\pi^{I_n^+}),\]
where $a$ is the diagonal element corresponding to $U_p^n$ if $p$ is inert, and $U_{p,1}^nU_{p,2}^n$ if $p$ is split (in other words $a$ is the double class corresponding to the compact operator $U_p$ in the text). This coincides with the space $V_{A^-}^{K_0}$ of \cite{Casselman}, Proposition 4.1.6.
\end{defin}

By \cite{Casselman} Lemma 4.1.7, this space $(\pi^{I_n^+})^{fs}$ is endowed with an action of $\mathcal A^+_n(p)$. 

\begin{prop}
Let $\pi$ be a representation of $G(\QQ_p)$. Write $I_n^+ = \overline{N}_nT_n^+N_n$ its Iwahori decomposition.
Then, as $\Sigma = \mathcal A_n^+(p)$-module,
\[ (\pi^{I_n^+})^{fs} = (\pi_{N_n})^{T_n^+} \otimes  \delta_B^{-1},\]
\end{prop}

\begin{proof}
As in \cite{BC2} Proposition 6.4.3, this is due to \cite{Casselman} Proposition 4.1.4. using the Iwahori decomposition.
\end{proof}

\begin{rema}
We could also extends a bit the previous isomorphism by adding the action of (the split part of) $T/T_n^+$ as in \cite{Casselman}.
\end{rema}

Moreover, as $\pi^n$ is a quotient of an induction (or the induction from a parabolic subgroup in the split case), we will use the same geometric lemma as \cite{BC2} Proposition 6.4.4. In particular we only need to calculate the admissible refinement using this lemma, and as this does not assume $\chi$ to be unramified, we find exactly the same (automorphic) refinements as if $p \notdivides \Cond(\chi)$ in $(((\pi_p^n)^{I_n^+})^{fs})^{ss}$. 

\begin{defin}
\label{raf}
Let $\sigma$ be the refinement corresponding to the one when $p \notdivides \Cond(\chi)$ used in \cite{Her3} when $p$ is inert (in which case it is unique, see \cite{Her3} Proposition 10.7), and to $(\mu|.|^{-1/2})(\overline\chi(p),1,p^{-1})$, see \cite{BC2} Lemma 8.2.1 when $p$ is split \footnote{For theorem \ref{corBK} when $p$ is split, we could also use the same refinement as in \cite{Her3}, i.e. $\mu|.|^{-1/2}(1,\overline\chi(p),p^{-1})$, but for theorem \ref{thrGEO}, we need to be at the point corresponding to $\sigma$. The reason is that the refinement need to be \textit{anti-ordinary}, as shown in proposition \ref{proppoids}. In particular, we choose a priori a different refinement as \cite{BC1,BC2}, but this is merely superficial : in our case as in \cite{BC1,BC2} our point (thus our choice of refinement) is anti-ordinary compared to the ordered Hodge-Tate weights, the reason is that when $\ord_{s=0}L(\chi,s)$ is even, the order of the Hodge-Tate weights at our point is not the same as in \cite{BC1,BC2}.}. More precisely, it corresponds to 
\[ 
(\mu|.|^{-1/2})(\overline\chi(p),1,p^{-1}) : \begin{array}{ccc}
T/T_n^+  & \fleche  & \CC^\times  \\
(a,b,c)  & \longmapsto  & (\mu|.|^{-1/2})(abc)\overline{\chi}(a)|c|
\end{array}
\]
in the case where $p$ splits, and to
\[ 
(\mu|.|^{-1/2})(1,\overline\chi(p),p^{-1}) : \begin{array}{ccc}
T/T_n^+ & \fleche  & \CC^\times  \\
(a,e)  & \longmapsto  & (\mu|.|^{-1/2})(a\overline{a}^{-1}e)\overline{\chi}(e)|a|
\end{array}
\]when $p$ is inert. Recall that $T \simeq (\mathcal O[1/p])^\times\times(\mathcal O[1/p])^1$ in this case.
\end{defin}

\subsection{Refinements of De Rham representations}

In this subsection, we slightly generalise the well-known notion of refinements (see e.g. \cite{BC2} section 2.4) to non-necessarily crystalline representations. This is especially usefull for us when $p | \Cond(\chi)$.

\begin{defin}
Let $V$ a $n$-dimensionnal, continuous $L$-representation of $G_K$, where $K$ is a $p$-adic field. Assume that $V$ is De Rham, and denote $WD(V)$ the Weil-Deligne representation associated to $V$ (see \cite{FonWD,BLGGTII}).
Assume that $L$ is big enough so that all eigenvalues of the Frobenius $\varphi$ on $WD(V)$ are defined on $L$. 
A \textit{Refinement} of $V$ is the datum $(\mathcal F_i)_{i=1,\dots,n}$ of a filtration 
\[ 0 \subsetneq \mathcal F_1 \subsetneq \dots \subsetneq \mathcal F_n = WD(V),\]
by Weil-Deligne representations.

Just as in the crystalline case, the previous definition more generally applies to a general De Rham $(\varphi,\Gamma)$-module $D$, to $WD(D)$ (see \cite{BergerAst}).
\end{defin}

\begin{rema}
Obviously when $V$ is crystalline, this definition coïncides with the one of \cite{BC2}.
\end{rema}

Let $D$ be a De Rham $(\varphi,\Gamma)$-module. Let $(\mathcal F_i)$ be a refinement of $D$, i.e. a filtration of $WD(D)$. Then we can associate to $(\mathcal F_i)$ a filtration of $D$ by
\[ \Fil_i(D) = (R[1/t]\mathcal F_i)\cap D.\]
This filtration is saturated, and thus defines a triangulation of $D$ (see \cite{BC2}, section 2.3).

\begin{prop}\label{prop1013}
The previous map $(\mathcal F_i) \mapsto (\Fil_i D)$ induces a bijection between the set of refinements of $D$ and the set of triangulations of $D$.
\end{prop}

\begin{proof}
This is \cite{BergerAst} Théorème A and Corollaire III.2.5.
\end{proof}

In the particular case of an automorphic representation $\pi$ of our unitary group $G$, with associated Galois representation $\rho_\pi$ (for example $\rho = 1 \oplus \chi_p^c \oplus \eps$ associated to the automorphic representation $\pi^n$ of the previous subsection), we have distinguished – we call them \textit{accessible}, (galois) refinements for $\rho_{\pi,v}$ which correspond to the (autormorphic) refinements for the action of $\mathcal A_n^+(p)$ on $\pi^n_v$ (for $v | p$). Such refinements exist only if $(\pi_v)^{I_n^+} \neq 0$ for some $n$. The association is explained in \cite{BC2} (when $G$ is split at $v$) for unramified representations, and for $U(3)(\QQ_p)$ (when $v$ is inert) in \cite{Her3} section 10.5. This can be generalized for non-necessarily unramified $\pi_v$, verbatim when there is no monodromy\todo{otherwise we need to check that such a refinement is stabilized by $N$ : probably true using the normalisation of LL and Bernstein-Zelevinsky (paying attention to the normalisation)}.
For example, to the refinement $\sigma$ of definition \ref{raf}, is associated the following refinement of $\rho = 1 \oplus \chi_p^c \oplus \eps$ :
\begin{equation}
\label{raftwist} 
\left\{\begin{array}{cc}
0 \subsetneq LL(1) \subsetneq LL(1) \oplus LL(\chi_p^c) \subsetneq WD(\rho_{G_p}) & \text{when $p$ is inert}. \\
0 \subsetneq LL(\chi_v^c) \subsetneq LL(1) \oplus LL(\chi_v^c) \subsetneq WD(\rho_{G_v}) &  \text{when $p$ is split}.
\end{array}
\right.
\end{equation}

Here $1$ is the trivial representation of $E_v^\times$, and $LL$ denotes the Local-Langlands correspondance. 

\subsection{Constructing the extension}

We thus take a prime $p$ unramified in $E$, which can be $2$ or not, and which can divide $\Cond(\chi)$ or not.
Let $\mathcal E$ be the eigenvariety of level $N = \Cond(\chi)^p$ (the prime-to-$p$-part of the conductor) associated to $U(2,1)_E$ and $p$ by Theorem \ref{thrHecke}.
It is equipped with a map $w : \mathcal E \fleche \mathcal W$, and there is a point $y \in \mathcal E$ which coincides with the representation $\pi^n$ together with its 
refinement $\sigma$ by definition \ref{raf} and proposition \ref{propkappa}. For all $\mathcal Z \subset \mathcal E$, we have associated to the automorphic form corresponding to $z$ a Galois representation
\[ \rho_z : G_E \fleche \GL_3(\overline{\QQ_p}),\]
which is moreover polarised in the following way :
\[ \rho_z^\perp \simeq \rho_z(-1) := \rho_z \eps^{-1},\]
where $\eps$ denote the cyclotomic character.
Let us be more precise : we will change a bit the convention used in \cite{Her3} to stick with the one of \cite{BC2} (this will make things easier to treat the case $p | \Cond(\chi)$). 
Denote for an automorphic representation $\pi$ of $U(2,1)$ of regular weight $\rho_\pi'$ the associated $p$-adic Galois representation by \cite{BC2} Conjecture 6.8.1, which is know to exists, see Remark 6.8.3, (vi) of \cite{BC2}. For $z \in \mathcal Z$ associated to a modular form $f_z$, denote by $\Pi$ any irreducible constituent of the representation of (the restriction to) $U(2,1)(\mathbb A)$ generated by $f_z$. Then we set
\[ \rho_z = \rho_{\pi}'\nu,\]
where $\nu$ is defined in \cite{BC2} Lemma 8.2.3, and is associated by class field theory to $\mu^{-1}|.|^{3/2}$\footnote{Carreful to the normalisation of the Local Langlands correspondance in \cite{BC2}
}. In particular it satisfies $\nu^\perp = \nu(-3)$.
Thus $\rho_z^\perp = \rho_z(-1)$. Moreover for $z \in \mathcal Z$ of classical (automorphic) weight $(k_1\geq k_2,k_3)$, the Hodge-Tate weights of $\rho_z$ are given by
\[ 
((1-k_3,k_2-1,k_1),(-1-k_1,-k_2,k_3-2)) = \left\{
\begin{array}{cc}
(\HT_\tau,\HT_{\overline\tau})(\rho_z)  & \text{ if  $p$ is inert } \\
(\HT_v,\HT_{\overline v})(\rho_z)  & \text{ if  $p$ splits}
 \end{array}
\right.
\]


\begin{prop}
There exists a pseudo character on $\mathcal E$,
\[ T : G_E \fleche \mathcal O(\mathcal E),\]
such that for all $z \in \mathcal Z$, $T_z$ is the trace of $\rho_z$. Moreover, $T^\perp = T(-1)$.
\end{prop}

\begin{proof}
This is \cite{Che1} Proposition 7.1.
\end{proof}

We need a particular point on $\mathcal E$. 

\begin{prop}\label{propx0}
Suppose $\ord_{s=0}L(\chi,s)$ is even and $L(\chi,0) = 0$. 
There exists a point $y \in\mathcal E$ corresponding to a (non-tempered) automorphic representation $\pi^n$.
The point $y$ is non-classical if $p |  \Cond(\chi)$\footnote{More precisely, it is non classical without level at $p$ as its system of Hecke eigenvalues doesn't appear in $H^0(\mathcal X,\omega^\kappa)$, but appears in $H^0(\mathcal X_0^+(p^n),\omega^{\kappa}(\chi))$.} but is classical otherwise. Moreover its $p$-adic weight $w(\chi)$ is of the form $w(\chi)^{alg}w(\chi)^{sm}$ where $w(\chi)^{sm}$ is the smooth (finite order) character of Proposition \ref{propkappa}, and $w(\chi)^{alg}$ is the algebraic character
\[ 
\begin{array}{ccc}
\mathcal O^\times \times \mathcal O^1  &  \fleche & \overline{\QQ_p}^\times  \\
 (x,y) & \mapsto  & \tau(y)^{1-a}\sigma\tau(x)^{-1}
\end{array} \quad \text{if $p$ is inert}
\]
\[ 
\begin{array}{ccc}
 \ZZ_p^3 & \fleche  & \overline{\QQ_p}^\times  \\
  (x,y,z) & \longmapsto  & y^{1-a}z^{-1}  
\end{array}
\quad \text{if $p$ is split}
\]
At the point $y$, the evaluation $T_{y}$ is given by the trace of $1 \oplus \eps \oplus \chi^c$ and the refinement is given by $\sigma$ of definition \ref{raf}, i.e. it is the refinement (\ref{raftwist}).
\end{prop}

\begin{proof}
This is a translation of Proposition \ref{propkappa} with the normalisation of $T$.
\end{proof}

We freely use the notation of \cite{KPX} concerning $\varphi,\Gamma$-modules. Denote $\delta_i$ for $i = 1,2,3$ the character,
\[ \delta_i : K^\times \fleche \mathcal O(\mathcal E)^\times,\]
such that $\delta_i(p) = F_i$\footnote{These $F_i \in \mathcal O(\mathcal E)^\times$ already appeared in proof of proposition \ref{propred}, see \cite{Her3} Proposition 11.1 and 11.2. These are the functions given by a basis of the Hecke operator in $\mathcal A(p)$.  }
 and, in the inert case, recall that we have on $\mathcal W$ two universal morphisms,
\[ \kappa_1 : x \in \mathcal O^\times \longmapsto \kappa_1(x) \in \mathcal O(\mathcal W)^\times, \quad \text{and} \quad \kappa_2 : y \in \mathcal O^1 \longmapsto \kappa_2(y) \in \mathcal O(\mathcal W)^\times,\]
such that at classical points $\kappa = (k_1,k_2,k_3) \in \ZZ^3$, we have
\[ \kappa_{1|\kappa}(x) = \tau(x)^{k_1}\overline\tau(x)^{k_3} \quad \text{and}\quad \kappa_{2|\kappa}(y)=\tau(y)^{k_2}.\]
We set
 \[\delta_{1|\mathcal O_K^\times} = (\overline \kappa_1)x_\tau^{-1} x_{\overline\tau},\]
\[ \delta_{2|\mathcal O_K^\times} : y \in \mathcal O_K^\times \longmapsto \kappa_2(\overline y/y) \tau(y),\]
\[ \delta_{3|\mathcal O_K^\times} = (\kappa_1)^{-1}x_{\overline\tau}^2 = (\overline{\delta_{1|\mathcal O_K^\times}})^{-1}x_\tau x_{\overline\tau}.\]
In particular we have $\delta_3 = \overline{\delta_1}^{-1}x$.
In the split case, we set
\[\delta_i : \QQ_p^\times \fleche \mathcal O(\mathcal E)^\times,\]
with $\delta_i(p) = F_i$ and as we have universal characters on $\mathcal W$,
\[ \kappa_i : \ZZ_p^\times \fleche \mathcal O(W)^\times,\]
such that for classical weights $(k_1,k_2,k_3) \in \ZZ^3$,
\[ \kappa_1(x) = x^{k_1}, \quad \kappa_2(x) = x^{k_2}, \quad \kappa_3(x) = x^{k_3},\]
We set \[ \delta_{1|\ZZ_p^\times} = \kappa_3x^{-1}, \quad \delta_{2|\ZZ_p^\times} = \kappa_2^{-1}x, \quad \delta_{3|\ZZ_p^\times} = \kappa_1^{-1}.\]
In particular, we can define $\wt_i := -\wt(\delta_i) \in \mathcal O(\mathcal W)^\Sigma$ the opposite of the derivative at 1 of $\delta_i$ (see \cite{KPX} Definition 6.1.6). In particular $\mathcal E, \mathcal Z$ and the functions $\delta_i$ satisfies the hypothesis of Corollary 6.3.10 of \cite{KPX} (excepts possibly the irreducibility condition).

Denote by $A = \mathcal O_{y}$ the rigid analytic local ring of $\mathcal E$ at $y$, and $K$ its total fraction ring. The pseudo character $T$ on $\mathcal E$ induces one on $A$, and denote by $I_{tot} \subset A$ its total reducibility ideal (see \cite{BC2} Proposition 1.5.1, Definition 1.5.2.) In particular for any $J \supset I_{tot}$ on $A/J$ we can write 
\[ T \otimes A/J = T_1 + T_{\overline\chi} + T_\eps.\]

\begin{prop}
\label{proppoids}
The reducibilty locus $\Spec(A/I_{tot})$ is a proper closed sub-scheme of $\Spec(A)$, i.e. $I_{tot} \neq \{0\}$. More precisely, if $a \geq 2$ and $p$ is split, we have that for all $i,j \in \{1,2,3\}$,
\[ 
\wt(\delta_i) - \wt(\delta_j) \equiv \wt(\delta_i)(y) - \wt(\delta_j)(y) \pmod{I_{tot}}
.\]
If $p$ is inert, \[ \wt_\tau(\delta_2) - \wt_\tau(\delta_3) \equiv \wt_\tau(\delta_2)(y) - \wt_{\tau}(\delta_3)(y) \pmod{I_{tot}},\]
and \[ \wt_{\overline\tau}(\delta_2) -\wt_{\overline\tau}(\delta_1) \equiv \wt_{\overline\tau}(\delta_2)(y) - \wt_{\overline\tau}(\delta_1)(y) \pmod{I_{tot}},\]
and similarly (with $\tau,\overline\tau$ changed by $v,\overline v$) is $p$ splits and $a=1$.
\end{prop}

\begin{proof}
Let $I \supset I_{tot}$ a finite length ideal of $\mathcal O_{y_0} = A$. We thus have for $j = \{1,\overline{\chi},\eps\}$,
\[ T_j : G_{E,S} \fleche A/I,\]
a (continuous) character, such that $T_j \pmod{\mathfrak m_A} = j$.

As $T_j$ is a character, by \cite{KPX} Theorem 6.2.14, there exists a character
\[ \delta_j' : K^\times \fleche (A/I)^{\times},\]
such that the $\varphi,\Gamma$-module associated to $T_j$, $D_{rig}(T_j)$ is isomorphic to $R_{A/I}(\pi_K)(\delta_j')$. From now on we just write this last space $R_{A/I}(\delta_j')$. We will determine $\delta_j'$. Recall that $j \in \{1,\overline\chi,\eps\}$ and $i \in \{1,2,3\}$. We choose the bijection between these two spaces, which 
corresponds to the refinement \ref{raf}, more precisely,
\[
\begin{array}{ccc}
1  & \mapsto  & \overline\chi  \\
2  &  \mapsto & 1  \\
 3 &  \mapsto &  \eps 
\end{array}
\text{if $p$ splits,} \quad 
\begin{array}{ccc}
1  & \mapsto  & 1  \\
2  &  \mapsto & \overline\chi  \\
 3 &  \mapsto &  \eps 
\end{array}
\text{if $p$ is inert.}
\]
Thus $T_2 = T_{\overline\chi}$ when $p$ is inert and $T_3 = T_\eps$ in any case, for example. 
By lemma \ref{lemmaB4}, we have a in particular a map
\[ R(\delta_i) \hookrightarrow D_{rig}(T_i)\simeq R(\delta_i'),\]
and we still want to determine the character $\delta_i'$.

To determine $\delta_i'$ the character of $T_i$, we still need to know the weight of $T_i$. We know by Lemma \ref{lemmaB2} that $T_i$ has its Sen operator killed by
\[ \prod_{i = 1}^3 (T - \wt_i) \in A/I[T]^\Sigma.\]
Moreover, at $y \in \mathcal E$, we have, if $p$ splits,
\[(\wt^v_1,\wt^v_2,\wt^v_3) = (0,-1,a-1) \quad \text{and} \quad  (\wt^{\overline v}_1,\wt^{\overline v}_2,\wt^{\overline v}_3) = (-a,0,-1) \]
and if $p$ is inert,
\[ (\wt^\tau_1,\wt^\tau_2,\wt^\tau_3) = (0,-1,a-1) \quad \text{and} \quad  (\wt^{\overline \tau}_1,\wt^{\overline \tau}_2,\wt^{\overline \tau}_3) = (-a,0,-1).\]
Thus, if $a \geq 2$ 
these weights are distincts at $y$.
Thus we can calculate the Hodge-Tate-Sen weight of $T_i$ : $T_1$ has weight $\wt_1$, $T_{\overline\chi}$ has weight $\wt_3$ and $T_\eps$ has weight $\wt_2$ \todo{Revoir, c'est pour $p$ inert ça !}.
If $a =1$, we can't a priori distinguish the two weights $h_1$, $h_3$ at $v$ and $\overline v$ (or $\tau$ and $\overline\tau$), but we know that $T_\eps = T_3$ has weight $\wt(\delta_2)$ at $v$, and 
that $T_1$ has weight $\wt(\delta_2)$ at $\overline v$.

Suppose $p$ is split, using the lemma \ref{lemmaB5} for $T_\eps$, we have (evaluating at $y$ to have the value of $t_\sigma,k_\sigma$),
\[ \wt_v(\delta_2) - \wt_v(\delta_3) - (\wt_v(\delta_2)(y)-\wt_v(\delta_3)(y)) \in I.\]
Using that $\delta_3 = \overline{\delta_1}^{-1}x$ and similarly for $\delta_2$ (or using lemma \ref{lemmaB5} for $T_1$ at $\overline v$), we get also
\[ \wt_{\overline v}(\delta_2) - \wt_{\overline v}(\delta_1) - (\wt_{\overline v}(\delta_2)(y)-\wt_{\overline v}(\delta_1(y))) \in I.\]
If $a \geq 2$, we can also use Lemma \ref{lemmaB5} for $T_1$ and $T_{\overline\chi}$ at $v$ and $\overline v$ to get the remaining,
\[\wt(\delta_i) - \wt(\delta_j) \equiv \wt(\delta_i)(y) - \wt(\delta_j)(y) \pmod{I}.\]
Similarly if $p$ is inert, applying the lemma \ref{lemmaB5} at $v$ for $T_\eps$ and $\delta_2$ we get
\[ \wt_\tau(\delta_2) - \wt_\tau(\delta_3) - (\wt_\tau(\delta_2)(y) - \wt_\tau(\delta_3)(y)) \in I,\]
and applying the previous lemma at $\overline \tau$ to $T_1$ and $\delta_2$ we have moreover,
\[ \wt_{\overline\tau}(\delta_2) - \wt_{\overline\tau}(\delta_1) - (\wt_{\overline\tau}(\delta_2)(y) - \wt_{\overline\tau}(\delta_1)(y)) \in I.\]
Thus if $a \geq 2$ and $p$ is split (and similarly but not for all $i,j$ in the remaining cases), we have 
\[ \wt(\delta_i) - \wt(\delta_j) - (\wt(\delta_i)(y) - \wt(\delta_j)(y)) \in I,\]
for all $i,j$ and all cofinite lenght ideal $I$ containing $I_{tot}$, thus, for all $i,j$, 
\[\wt(\delta_i) - \wt(\delta_j) - (\wt(\delta_i)(y) - \wt(\delta_j)(y)) \in I_{tot}.\] 
\end{proof}

We also need the following result, which is a corollary of theorem \ref{theorB5}.

\begin{cor}
\label{cor1018}
$Ext_T(1,i) \subset H^1_f(E,i)$, for $i = \overline\chi$ or $\eps$.
\end{cor}

\begin{proof}
Indeed, the theorem \ref{theorB5} gives that any extension in $\Ext_T(1,i)$ is crystalline at all place above $p$ (as the Frobenius eigenvalues of $i$ are different from 1).
At $v$ a place dividing $\ell \neq p$, if $v \notdivides \Cond(\chi)$, by hypothesis on the level of $\mathcal E$, the dense set of classical point $Z$ are unramified at $v$, thus $T(I_v) = 1$ on $\mathcal E$ (as $\mathcal E$ is reduced) and thus $\Ext_T(1,i)$ consist of unramified extension at $v$.

Now suppose $v | \Cond(\chi)$. If $i = \eps$, any extension is automatically unramified. Suppose $i = \overline\chi$. By choice of the type $J$ outside $p$ on $\mathcal E$, we know (\cite{BC1} Proposition 4.2 or \cite{Her3} proposition 10.21)
that for all $z \in  Z$, there exists a subgroup $I' \subset I_v$ such that $\rho_z(I') = \{1\}$. Thus, $T(I') = 1$ and for all $x \in \mathcal E, \rho_x(I') = 1$.
Thus, $T _{I_v}$ is locally constant, and the same for $\rho_{x|I_v}$ (as it is semi-simple as $I'$ acts trivially). Up to extending scalars, evaluating at $\rho_y$, we get
\[ T_{|I_v} = (1 \oplus 1 \oplus \overline\chi_{|I_v}) \otimes \mathcal O_U,\]
for some neighborhood $U$ of $x$. But as we have a morphism
\[ M_1/IM_1 \oplus T_i \fleche \rho_c \fleche 0,\]
we have that $\rho_c(I') = 1$, thus $\rho_c$ is semi-simple, thus $\rho_{c|G_v} \in H_f^1(G_v,1)$.
\end{proof}

We have the following improvement of Theorem \ref{thrBK} :

\begin{theor}
\label{corBK}
Let $\chi$ be a polarized algebraic Hecke character as in Theorem \ref{thrBK}. Suppose that $L(\chi,s)$ vanishes with even (non-zero) order at $s=0$. Let $p$ be unramified in $E$.
Then
\[ H^1_f(E,\chi_p) \neq \{0\}.\]
\end{theor}

\begin{proof}
Let $e_1,e_{\overline\chi},e_\eps$ be the indempotents as in Appendix \ref{sectAppB}, and denote $A_{i,j}$, for $i,j \in \{1,\overline\chi,\eps\}$ the corresponding $A$-modules.
Then as in \cite{BC2} Lemma 8.3.2, we get
\[ I_{tot} = A_{1,\overline\chi}A_{\overline\chi,1}.\]
But if $Ext_T(1,\overline\chi) = 0$, then $A_{\overline\chi,1} = A_{\overline\chi,\eps}A_{\eps,1}$ (\cite{BC2}, Theorem 1.5.5).
Thus $A_{\overline\chi,1}A_{1,\overline\chi} = A_{\overline\chi,\eps}A_{\eps,1}A_{1,\overline\chi}$. 
But as $H^1_f(E,\eps) = \{0\}$, we get by the same reasonning
\[ A_{\eps,1} = A_{\eps,\overline\chi}A_{\overline\chi,1}.\]
Thus, 
\[ I_{tot} = A_{\overline\chi,\eps}A_{\eps,\overline\chi}A_{\overline\chi,1}A_{1,\overline\chi} \subset \mathfrak mA_{\overline\chi,1}A_{1,\overline\chi} = \mathfrak mI_{tot}.\]
Thus $I_{tot} = 0$, contradicting proposition \ref{proppoids} for $2 = i \neq j = 3$.
\end{proof}

Actually we can be more precise than Theorem \ref{corBK} if $a \geq 2$ and $p$ splits. Denote from now on by $\mathcal E'$, the Eigenvariety on which $\kappa_2$ is constant, equal to $\kappa_2(y) = 0$

It is actually the base change of the three-dimensionnal $\mathcal E$ by $\mathcal W_{\mathcal O^\times} \hookrightarrow \mathcal W$ (see remark \ref{remaCheJL}).

\begin{theor}
\label{thrGEO} Suppose that $p$ is split.
Remember that we assumed that $L(\chi,0) = 0$ and $\ord_{s=0}L(\chi,s)$ is even. We denoted $y$ the point of $\mathcal E'$ with its refinement proposition \ref{raf}.
Let $t'$ be the dimension of the tangent space of $\mathcal E'$ at $y$, and let $h$ be the dimension of $\Ext_T(1,\overline\chi)$. Then if $a \geq 2$,
\[ t' \leq h\frac{h+3}{2}.\]
\end{theor}

In particular as $t' \geq 2$, we have that $h \neq 0$. 

\begin{rema}
\begin{enumerate}
\item If $t$ denotes the dimension of the tangent space at $y$ of $\mathcal E$, we have $t \leq t' +1$.
\item If $h=1$ (in particular if $\dim H^1_f(E,\chi_p) = 1$) and $a > 1$ then $\mathcal E$ is regular at $y$, and thus there is a unique component of $\mathcal E$ passing through $y$. Unfortunately, because of the Bloch-Kato conjecture, we never expect this to be the case as the order of vanishing is even.
\item In case $a=1$, unfortunately we don't have that 
\[ H_f^1(E,\overline\chi(-1)) = \{0\}\]
(which is a special case of the Bloch-Kato conjecture) nor $\dim H^1(E,\overline\chi(-1)) \leq 1$ (which is a special case of Jannsen conjecture). Indeed, in this case $V = \overline\chi(-1)$ is of motivic weight 1 (and automorphic) thus the $L$-function of $V^*(1)$ doesn't vanish at $s = 0$. 
But even if we had these results, the main obstruction to get the previous theorem \ref{thrGEO} when $a=1$ is Proposition \ref{proppoids}, as we can't control along the reducibility locus
 the Hodge-Tate weights which have same value at $y$  (as when $p$ is inert).

\end{enumerate}\end{rema}

For now on suppose $a \geq 2$ and $p$ splits (unless otherwise stated). Denote $A = \mathcal O_{\mathcal E',y}$. The GMA associated to $e_\eps,e_{\overline\chi},e_1$ is of the form
\[
\left(
\begin{array}{ccc}
A  & A_{\eps,\overline\chi}  & A_{\eps,1}  \\
A_{\overline\chi,\eps}  &  A & A_{\overline\chi,1}  \\
A_{1,\eps}  & A_{1,\overline\chi}  &   A
\end{array}
\right)
\]

\begin{prop}
$I_{tot}$ is the maximal ideal of $A = \mathcal O_{\mathcal E',y}$.
\end{prop}

\begin{proof}
By proposition \ref{proppoids} we already have that $\Spec(A/I_{tot})$ is included in the fiber of $w: \mathcal E' \fleche \mathcal W$ at $\kappa$, i.e. $I_{tot} \supset \mathfrak m_\kappa\mathcal O_E$.
Let $\psi : A/I_{tot} \fleche k[\eps]$ (where $\eps^2 = 0$) any morphism, and for $i \in \{1,\overline\chi,\eps\}$, denote
\[ T_{i,\psi} = T_i \otimes_{A/I_{tot},\psi} k[\eps].\]
We will show that these $T_{i,\psi}$ are constant deformation of $i$ to $k[\eps]$.
Indeed, they are all unramified outside $p\Cond(\chi)$ by construction (as $T$ is). If $v \Cond(\chi)$ is prime to $p$, then $T_{i,\psi}$ can be seen as an extension of $i$ by itself, thus gives a class in $H^1(E_v,1 = ii^{-1})$. But such a class is trivial (as $\Hom(i,i(-1)) = \{0\}$).
If $v |p$, we need to check that the induced class in $H^1(\QQ_p,\QQ_p)$ is crystalline. For all $i$, it suffices to show it for $T_{i,\psi}i^{-1}$, which reduces to $1$. Thus, in every case we have a deformation of $1$ to $k[\eps]$. But we know moreover that its Hodge-Tate weight is an integer as we are in the 
fiber $w^{-1}(\kappa)$. If $p$ is split, then this is exactly Proposition 2.3.4. If $p$ is inert, the same proof applies, using again Berger's result (\cite{BergerAst}).
Thus, for every $i$, $T_{i,\psi} \in H^1_f(E,\QQ_p)$. But this space is trivial.

Recall that $\mathcal O_{\mathcal E}$ is generated over $\mathcal O_{\mathcal W}$ by $F_i, i \in \{1,2,3\}$ and by (unramified) Hecke operators at $v$ prime to $p$ not dividing 
$\Cond(\chi)$.
But these Hecke operators are determined by the conjugacy class of $\rho_{z}(Frob_v)$ for $z \in \mathcal Z$ by Zariski density and because $\mathcal O_{\mathcal E}$ is 
reduced, and in turn this conjugacy class is determined by $T(\Frob_v^n), n \in \NN$. Thus as we showed that the $T_{i,\psi}$ are trivial, we have for all $v \notdivides \Cond(\chi)$ and different from $p$, that $T_{i,\psi}(\Frob_v^n) \in k$ thus $T_\psi(\Frob_v^n) \in k$. Moreover, modulo $I_{tot}$ the image of $\mathcal O_{\mathcal W}$ in $A = \mathcal O_{y}$ is just $k = A/\mathfrak m_A = \mathcal O_{\mathcal W}/\mathfrak m_\kappa$ as $I_{tot} \supset \mathfrak m_\kappa$. Moreover, we have seen in Lemma \ref{lemmaB4} that $D_{rig}(T_i(\delta_{i,|\Gamma}^{-1}))$ has slope $F_i$, but as $T_{i,\psi}$ is in the fiber of $\kappa$, this means that $D_{rig}(T_{i,\psi}(\delta_{i,\Gamma}^{-1}(y)))$ as slope $F_i \pmod I_{tot}$, but as $T_{i,\psi}$ is constant, $F_i \in k$. Thus $\psi(A/I_{tot}) \subset k$. As this is true for all $\psi$, we get the result.
\end{proof}

Now the rest of the proof is the same as \cite{BC2} Chapter 9. The previous results prove that $I_{tot} = A_{1,\rho}A_{\rho,1}$ (see \cite{BC2} Lemma 9.3.1), thus $A_{\rho,1}A_{\rho,1} = \mathfrak m_A$ by the previous proposition. Denote for $\rho$ a $p$-adic representation of $G_{E,S}$, $H^1_{f'}(E,\rho) \subset H^1(E,\rho)$ the subspace generated by extensions 
\[ 0 \fleche \rho \fleche U \fleche 1 \fleche 0,\]
satisfying, for all $v \notdivides p$,
\[ \dim U^{I_v} = 1 + \dim \rho^{I_v}.\]

\begin{lemma} If $a\geq 2$
 \[\dim H^1_{f'}(E,\overline{\chi}(-1)) \leq 1.\]
If $a = 1$, 
 \[\dim H^1_{f'}(E,\overline{\chi}(-1)) \leq 2 + \dim H^1_f(E,\overline{\chi}(-1)).\]

\end{lemma}

\begin{proof}
This is the same proof as \cite{BC2} Proposition 5.2.7. As $H^1_f(E,\overline\chi(-1)) = 0$ by \cite{BC2} Proposition 5.2.5 if $a \neq 1$, it is enough to calculate
\[ \dim H^1(E_p,\overline\chi(-1))/H_f^1(E_p,\overline\chi(-1)).\]
Thus, by the local Euler Characteristic formula, it is enough to calculate $H^1_f(E_p,\overline\chi(-1))$. But this is given by \cite{BK} Corollary 3.8.4 as $\overline\chi$ is De Rham. We get, when $p$ is inert,
\[ \dim H_f^1(E_p,\overline\chi(-1)) = 
\left\{
\begin{array}{cc}
1  & \text{if } a \geq 2 \\
 0 &  \text{if } a =1
  \end{array}
\right.
 \]
and if $p = v\overline v$ is split,
\[ \dim H_f^1(E_v,\overline\chi_v(-1)) =  0 \quad \text{and} \quad \dim H_f^1(E_v,\overline\chi_{\overline v}(-1)) =  \left\{
\begin{array}{cc}
1  & \text{if } a \geq 2 \\
 0 &  \text{if } a =1
  \end{array}
\right.\]
\end{proof} 

Using Corollary \ref{cor1018} (and the analog outside $p$ for $h_{\overline\chi,1}$) we get exactly as \cite{BC2} Lemma 9.4.2 :

\begin{lemma}
Denote $h_{i,j} = \dim \Ext_T(i,j)$. Then
\begin{enumerate}
\item $h_{1,\overline\chi} = h_{\overline\chi,\eps} = h$
\item $h_{1,\eps} = 0$ and $h_{\eps,1} \leq 1$.
\item $h_{\overline\chi,1} = h_{\eps,\overline\chi} \leq 1$.
\end{enumerate}
\end{lemma}

But we know (see \cite{BC2} lemma 9.4.2) that $h$ is the minimal number of generators of $A_{\overline\chi,1}$, and similarly, 
$A_{1,\overline\chi}A_{\overline\chi,1} = \mathfrak m$, which has by hypothesis $t'$ as minimal number of generators, can be generated (see \cite{BC2} Lemma 9.4.3, by
\[ h + s\]
generators, where $s$ is the minimal number of generators of $A_{\eps,\overline\chi}A_{\overline\chi,1} \simeq A_{\overline\chi,1}A_{\overline\chi,1}$, thus 
\[ s \leq h\frac{h+1}{2}.\]
Putting everything together, we get
\[ t' \leq h + h\frac{h+1}{2} = h\frac{h+3}{2}.\]

\appendix
\section{Cohomology of cuspidal automorphic sheaves}
\label{AppA}
\begin{prop}(Lan, \cite{LanIMRN} Theorem 6.1) 
Let $\mathfrak X_1(p^n)^*$ the minimal compactification of $\mathfrak X_1(p^n)$, defined by normalisation of the minimal compactification with our fixed auxiliary level, as in \cite{LanSigma}, Proposition 6.1. There is a proper surjection $p : \mathfrak X_1(p^n)^{tor} \fleche \mathfrak X_1(p^n)^*$.
\end{prop}

\begin{defin}
The ($\mu$-ordinary) Hasse invariant ${^\mu}\Ha$ descend to $\mathfrak X_1(p^n)^*$ (modulo $p$), and we can thus define $\mathfrak X_1(p^n)^{\mu-full*}(v)$ to be the 
normalisation in its generic fiber of the greatest open in the blow up of $(^\mu\Ha,p^v)$ where this ideal is generated by ${^\mu\Ha}$. Its generic fiber is 
$\mathcal X_1(p^n)^{\mu-full*}(v)$, 
a strict neighborhood of the (full) $\mu$-ordinary locus. Denote $\mathcal X_1(p^n)^*(v)$ the (union of) connected component which contains a point of maximal degree, and as $\mathfrak X_1(p^n)^{\mu-full*}$ is normal in its generic fiber, there is an associated open $\mathfrak X_1(p^n)^*(v)$. We thus have a map,
\[ \pi(v) : \mathfrak X_1(p^n)^{tor}(v) \fleche \mathfrak X_1(p^n)^*(v).\]
\end{defin}


For all this section, except the last two results (Corollary \ref{corA5}, Theorem \ref{thrvanishing}), we forgot the notation concerning the level at $p$, and denote $\mathfrak X_1(p^n)^{tor}(v)$ by $\mathfrak X^{tor}(v)$, and similarly for $\mathfrak X(v),\mathfrak X^*(v),\mathfrak X^{tor},\mathfrak X,\mathfrak X^*$.
We thus have the previous map,
\[ \pi(v) : \mathfrak X^{tor}(v) \fleche \mathfrak X^*(v).\]

%
%

We have the following vanishing result.

\begin{prop}
Denote by $D(v)$ the boundary in $\mathfrak X^{tor}(v)$. Then, for all $q > 0$,
\[ R^q\pi(v)_*\mathcal O(-D(v)) = 0.\]
\end{prop}

\begin{proof}
This is essentially Lan's result (see \cite{LanIMRN} Proposition 8.6), slightly modified because of the neighborhood we chose.
First, note that we can prove it for $\mathfrak X^{\mu-full}(v)^{tor}$ and $\mathfrak X^{\mu-full}(v)^*$ and then localise (as the schemes are normal and thus have same connected component as their rigid fiber) to $\mathfrak X^{tor}(v)$ and $\mathfrak X^*(v)$. From now on and until the end of this proof, we denote $\mathfrak X(v)^?$ the neighborhood of the full $\mu$-ordinary locus in $\mathfrak X^?$. By the formal functions theorem we can work on formal completions of points of $x \in \mathfrak X^*(v)$. Let us describe the completions at $x$ of $\mathfrak X(v)^{tor}$. 
Let $\mathfrak Z$ be a stratum of $\mathfrak X^*$ (\cite{LanSigma} Theorem 12.1), and denote $\mathfrak Z(v)$ be the base change of $\mathfrak Z$ to $\mathfrak X^*(v)$. 
As $\mathcal X^*(v)$ is normal in its generic fiber $\mathcal X^*(v)$ by construction, by pullback $\mathfrak Z(v) = \mathfrak X^*(v) \times_{\mathfrak X^*}\mathfrak Z$ is normal in its generic fiber.


Let $C \fleche \mathfrak Z$ be the proper scheme, normal over $\mathcal O$ (\cite{LanSigma} Proposition 8.4). 
Then $C(v):= C \times_{\mathfrak X^*} \mathfrak X^*(v) = C \times_\mathfrak Z \mathfrak Z(v)$ is normal in its rigid fiber (again by pullback). Define analogoulsy the local models (see \cite{LanKoecher} section 4.) $\mathfrak Y_{[\sigma]}(v)$ and $\mathfrak Y(v)$.

We can describe locally $\mathfrak X^{tor}$ over $\mathfrak X^*$ by $\mathfrak Y$ (see \cite{LanSigma} Theorem 10.3, \cite{LanIMRN} Theorem 6.1 (4)) and
also for $\mathcal X^{tor}(v)$ over $\mathcal X^*(v)$, i.e. in rigid fiber, as this is just localisation over an open subset. Denote by  $\mathfrak X^*(v)^0$ denote the open of the blow-up where the ideal is generated by ${^\mu}\Ha$ (i.e. before taking the normalisation in its rigid fiber), and similarly for $\mathfrak X^{tor}(v)^0$. Then $\mathfrak X^{tor} \fleche \mathfrak X^*$ is not flat à priori, but as $({^\mu}\Ha,p^v)$ is in both cases a regular sequence, this implies that the admissible formal blow-up in both cases is given by the closed subset of equation $(X{^\mu}\Ha-Yp^v)$ in $Proj(\mathcal O_{\mathfrak X^?}[X,Y])$ (see e.g. \cite{Bosch}, Proposition 7. (iii)). Thus this admissible blow-up commutes with the base change $\mathfrak X^{tor} \fleche \mathfrak X^*$.
In particular, $\mathfrak X^{tor}(v)^0 = \mathfrak X^{tor} \times_{\mathfrak X^*} \mathfrak X^{*}(v)^0$. Thus $\mathfrak X^{tor}(v)$ is the normalisation of $\mathfrak X^{tor} \times_{\mathfrak X^*}\mathfrak X^{*}(v)^0$ in its rigid fiber. The etale, local isomorphism
\[ \mathfrak Y \fleche \mathfrak X^{tor}_{\hat{Z}},\]
can thus be base changed over $\mathfrak X^*$ by $\mathfrak X^*(v)^0$, and as etale localisation commutes with normalisation, we only need to check that 
our chart $\mathfrak Y(v)$ is the normalisation of $\mathfrak Y(v)^0 := \mathfrak Y \times_{\mathfrak X^*} \mathfrak X^*(v)^0$ in its rigid fiber. But 
$\mathfrak Y(v) \fleche \mathfrak Y(v)^0$ is pulled-back from $\mathfrak X^*(v) \fleche \mathfrak X^*(v)^0$, thus $\mathfrak Y(v)$ is normal in its rigid fiber, and moreover the morphism
\[ \mathfrak X^*(v) \fleche \mathfrak X^*(v)^0,\]
is finite by \cite{PS4} Proposition 1.1. Thus, again by \cite{PS4} Proposition 1.1, $\mathfrak Y(v)$ is the normalisation and we get a local isomorphism
\[ \mathfrak Y(v) \fleche \mathfrak X^{tor}(v).\]


Thus, for 
$x \in \mathfrak Z(v)$ we have, \[(\mathfrak Y(v)/\Gamma)^{\widehat{~}}_x \simeq (\mathfrak X(v)^{tor})_x^{\widehat{~}}\]

Then according to \cite{LanKoecher} Theorem 3.9 (and especially section 7), and \cite{LanIMRN} Theorem 8.6 it is sufficient to prove the analog of Proposition 8.3 (of \cite{LanIMRN}) for $p(v) : C(v) \fleche \mathfrak Z(v)$. But $p(v)$ is also proper, and the pullback of the sheaf $\Psi(\ell)$ relatively ample over $\mathfrak Z(v)$, thus the same proof applies.
\end{proof}

%

In the following, we denote for a object $X$  over $\Spec(\mathcal O)$ or $\Spf(\mathcal O)$ and $n \in \NN^*$, $X_n$ the base change to $\Spec(\mathcal O/p^n)$.
We also denote, as in \cite{AIP}, $\mathfrak W(w)^0$ the analogous weight space, but forgetting the torsion part when constructing $\mathfrak W(w)$. This can be seen for example as characters in $\mathfrak W(w)$ being trivial on the torsion part of $T(\ZZ_p)$, but we don't fix such an identification.

\begin{prop}
Consider the following diagram, for $m \geq n$,
\begin{center}
\begin{tikzpicture}[description/.style={fill=white,inner sep=2pt}] 
\matrix (m) [matrix of math nodes, row sep=3em, column sep=2.5em, text height=1.5ex, text depth=0.25ex] at (0,0)
{ 
\mathfrak X^{tor}(v)_m & & \mathfrak X^{tor}(v)_n \\
\mathfrak X^{*}(v)_m & &\mathfrak X^*(v)_n\\
};
\path[->,font=\scriptsize] 
(m-1-1) edge node[auto] {$i$} (m-1-3)
(m-2-1) edge node[auto] {$i'$} (m-2-3)
(m-1-1) edge node[auto] {$\pi_m$} (m-2-1)
(m-1-3) edge node[auto] {$\pi_n$} (m-2-3)
;
\end{tikzpicture}
\end{center}
We have the equality,
\[ i{'*}\pi_{n*}\mathfrak w_{w,n}^{\kappa^0\dag}(-D) = \pi_{m,*}i^*\mathfrak w_{w,m}^{\kappa^0\dag}(-D).\]
In particular, $\pi_*\mathfrak w_w^{\kappa^0\dag}(-D)$ is a small formal Banach sheaf on $\mathfrak X^*(v) = \mathfrak X_1(p^n)^*(v)$. Similarly for 
$(\pi\times1)_*\mathfrak w_w^{\kappa^{0,univ}\dag}(-D)$ on $\mathfrak X^*(v) \times \mathfrak W(w)^0$. Moreover $H^i(\mathfrak X^*(v), \pi_*\mathfrak w_w^{\kappa^0\dag}(-D))[1/p]$ vanishes for $i \geq 1$ (similarly for the higher direct image of $(\pi\times1)_*\omega_w^{\kappa^{0,univ}\dag}(-D)$ on $\mathcal W(w)$).
\end{prop}

\begin{proof}
The proof is the same as in \cite{AIP} or \cite{Brasca}, except that we stay at level $\mathfrak X_1(p^n)(v)$ (which is easier), as the map $\mathfrak X_1(p^n)(v) \fleche \mathfrak X(v)$ 
is not finite in our situation. We can prove as in \cite{AIP} that $\mathfrak w_{w}^{\kappa^0\dag}(-D)$ is a direct limit of sheaves whose cokernel is a successive extension of the sheaf $\mathcal O_{\mathfrak X_1(p^n)^{tor}(v)}(-D)$. Thus, it is enough to show that
\[ R^1\pi_*\mathcal O_{\mathfrak X^{tor}(v)}(-D) = 0,\]
but this is the previous proposition. This implies also that $R^i\pi_*\mathfrak w_w^{\kappa^0\dag}(-D) = 0$ for $i>0$. Moreover, as $\pi_*\mathfrak w_w^{\kappa^0\dag}(-D)$ is small on $\mathfrak X^*(v)$ which is generically affinoid, Theorem A.1.2.2 of \cite{AIP} implies its higher cohomology vanishes after inverting $p$.
\end{proof}

Exactly as in \cite{AIP},  section 8.2, we deduce the following two results. Now we go back to the notation $\mathfrak X_1(p^n)^{tor}(v)$ to denote the (integral toroïdal 
compactification of the) Shimura variety with level "$\Gamma_1(p^n)$" at $p$, and $\mathcal X^{tor}(v)$ denote the rigid analog, with Iwahori level at $p$, as in the rest of the 
text. 

\begin{cor}
\label{corA5}
The module
\[M^{0,un}_{v,w,cusp} := H^0(\mathfrak X_1(p^n)^{tor}(v) \times \mathfrak W(w)^0,\mathfrak w_w^{\kappa^{0,un}\dag}(-D))[1/p]\] is a projective $\mathcal O_{\mathfrak W(w)^0}[1/p]$-module, and for all $\kappa \in \mathcal W(w)^0$, the specialisation
\[ M^{0,un}_{v,w,cusp} \fleche H^0(\mathfrak X_1(p^n)^*(v),\mathfrak w_w^{\kappa^{0}\dag}(-D))[1/p],\]
is surjective.
\end{cor}

\begin{theor}
\label{thrvanishing}
For all $v,w$ the module $M^{un}_{v,w} := H^0(\mathcal X^{tor}(v) \times \mathcal W(w),\omega_w^{\kappa^{un}\dag}(-D))$ is a projective $\mathcal O_{\mathcal W(w)}$-module, and for all $\kappa \in \mathcal W(w)$, the specialisation map
\[ M^{un}_{v,w,cusp} \fleche M^{\kappa}_{v,w,cusp},\]
is surjective. Moreover $H^i(\mathcal X^{tor}(v) \times \mathcal W(w),\omega_w^{\kappa^{un}\dag}(-D))$ vanishes for $i > 0$.
\end{theor}

\section{Families and triangulations}
\label{sectAppB}
In this appendix we generalise the tools used in \cite{BC2} to prove the theorem in section \ref{sect10}. Fortunately, this is mainly a matter of reformulation, as most of the work is done in \cite{KPX}. From now on, we take $\mathcal E$ to be the eigenvariety for $(G)U(2,1)_{E/\QQ}$ and $p$ a prime unramified in $E$, constructed in section \ref{sect9} (see also \cite{Brasca} for $p$ split in $E$ and \cite{Her3} for $p$ inert not equal to 2), which is 3-dimensionnal or its variant with weight $k_2 \in \ZZ$ fixed, which coincide with the base change by 
\[ \mathcal W_{\mathcal O^\times} \hookrightarrow \mathcal W,\]
which is $2$-dimensionnal. Automorphically, the second construction "fixes the central character" (which can "move" in the three dimensionnal eigenvariety, but keeping its polarisation ; in particular even in the $3$-dimensionnal eigenvariety, we can't twist automorphic forms by a power of the norm character). In any case we always have $Z \subset \mathcal E$ a strongly Zariski-dense subset consisting of classical automorphic forms of integral (= algebraic) weight. This space is not dense for the analytic topology, as it is already the case in $\mathcal W$. We can define $Z^{la}$ the subset of $\mathcal E$ of classical automorphic forms possibly with level at $p$, and locally algebraic weight-character $\kappa \in \mathcal W$. $Z$ doesn't accumulate at $Z^{la}$, and as if $p | \Cond(\chi)$, we will only have a point $y \in Z^{la}$ corresponding to the automorphic representation $\pi^n(\chi)$ of section \ref{sect10}, we first need to enlarge a bit $Z$\footnote{We could actually prove directly the following result on all $Z^{la}$, and even the crystabellianity of these representations, by extending results of \cite{BPS,Bijmu} for all classical modular forms with Nebentypus, as it is done in \cite{PS1}. But the following will be enough for us.}. 

\begin{prop}
\label{propB1}
There exists $Z' \subset Z^{la}$, which accumulates at every point of $Z^{la}$, such that for all $z \in Z'$, we have that the Sen polynomial of $\rho_z$ is killed by
\[ \prod_{i=1}^3(T-wt_i(z)).\]
\end{prop}

\begin{proof}
Let $z \in Z^{la}$. In particular, there exists $w,n$ such that $z \in \mathcal E_{w,\eps}$ in the notations before Theorem \ref{thrHecke}. This $\mathcal E_{w,\eps}$ is affinoid.
Thus, by \cite{BC2} Lemma 7.8.11, there exists $g : \mathcal E'_{w,\eps}\fleche \mathcal E_{w,\eps}$ on which we have an actual representation on a coherent torsion free of $G = G_{E,S}$.
 We can then apply \cite{KPX} Definition 6.2.11 or \cite{BC2} p125 to have a Sen operator in family over $\mathcal E'_{w,\eps}$. But $Z^{}$ is Zariski dense in $\mathcal E_{w,\eps}$ thus as is its pullback $Z^{'alg}$ in $\mathcal E'_{w,\eps}$. 
Moreover, there is $Y \subset \mathcal E'_{w,\eps}$ Zariski open and dense, on which $Z_Y' = Y \cap Z^{'alg}$ is Zariski dense, with
$\rho_{z} = \rho_{g(z)}$ for all $z \in Y$. Thus for all $z \in Y \cap Z^{'alg}$, we have that the Sen operator is killed by 
\[ \prod_{i=1}^3(T-wt_i(z)).\]
By density, this is true for all $x \in \mathcal E'_{w,\eps}$. Thus, for all $y \in Z'_Y = g^{-1}(Z^{la})\cap Y$, the Sen polynomial of $\rho_y = \rho_{g(y)}$ is killed by the same polynomial. Using $Z' = g(Z'_Y)$ we get the result.
\end{proof}


By proposition \ref{propx0} there exists a point $y \in \mathcal E$, whose (semi-simplified) Galois representation is $1 \oplus \overline\chi \oplus \eps$ and its refinement is $\sigma$ (see definition \ref{raf}). Let $A = \mathcal O_{\mathcal E,y}$ be the rigid analytic local ring at $y$. We want to study this ring and the pseudo-character $T$ at $A$. By \cite{BC2} Theorem 1.4.4 and Lemma 1.8.3 for $S = A[G]/\Ker T$, we choose idempotents $e_\eps,e_{\overline\chi},e_1$ that are compatible with the involution $\tau$ given by $i \mapsto i^\perp(1)$. We thus have a generalized matrix algebra (GMA) of the form
\[
\left(
\begin{array}{ccc}
A  & A_{\eps,\overline\chi}  & A_{\eps,1}  \\
A_{\overline\chi,\eps}  &  A & A_{\overline\chi,1}  \\
A_{1,\eps}  & A_{1,\overline\chi}  &   A
\end{array}
\right)
\]
This defines $\Ext_T(i,j)$ and $h_{i,j} = \dim \Ext_T(i,j)$ for all $i \neq j \in \{1,\eps,\overline\chi\}$. In the end, we want to study $I_{tot}$ the total reducibility locus and this GMA.

On $A/I_{tot}$, we have pseudo-characters of dimension 1 (i.e. actual characters)
\[ T_j : G = G_{E,S} \fleche A/I_{tot}, \quad j \in \{\eps,1,\overline\chi\},\]
such that $T_j \otimes A/\mathfrak m_A = T_j \otimes k(y) = j$. From now on fix $I \supset I_{tot}$ a cofinite lenght ideal.
 
\begin{lemma}
\label{lemmaB2}
The Sen operator of $T_j$ is killed by the polynomial
\[ \prod_{i = 1}^3 (T - \wt_i) \in A/I[T]^\Sigma.\]
\end{lemma}

\begin{proof}
Let $y \in \mathcal E$. As remarked, the set $Z'$ of \ref{propB1} accumulates at $y$. Fix $j \in \{1,\overline\chi,\eps\}$ and denote $S = A[G]/\Ker T$. 
There exists $M$ a $S$-module, of finite type as $A$-module such that $MK = K^3$ and with an exact sequence
\[  0 \fleche K \fleche M/IM \fleche T_j \fleche 0,\]
such that $K$ as a Jordan-Holder sequence with all subquotient isomorphic to $T_i$ for $i \neq j$ (see \cite{BC2} Theorem 1.5.6 and Lemma 4.3.9).
Thus, it suffices to prove that $M/IM$ as its Sen operator killed by the previous polynomial. But by \cite{BC2} Lemma 4.3.7 (and because $Z'$ accumulates at $y$) we can find 
$U \subset \mathcal E$ an affinoïd open containing $z$, in which $Z'$ is Zariski dense, together with $\mathcal M$ a coherent torsion-free $\mathcal O_U$-module endowed 
with an action of $G$ such that $\mathcal M(U) \otimes A \simeq M$ as $A[G]$-module, and $\mathcal M \otimes_{\mathcal O(U)} \Frac(\mathcal O(U))$ is free of rank 3, semisimple as $G$-module and 
trace $T \otimes_{\mathcal O(X)}\mathcal O(U)$. By generic semi-simplicity and generic flatness, there exists $F \subset U$ a Zariski closed subspace such that for all $x \in U \backslash F$, $\mathcal M_y = \mathcal M_{y}^{ss} = \rho_y$. We can change $Z'$ by $Z' \cap (U \backslash F)$, which is still Zariski dense in $U$.
Denote by $\varphi$ the Sen operator of $D_{Sen}(\mathcal M)$ (or $\mathcal B = \End_{\mathcal O(U)}(\mathcal M(U))$ see \cite{KPX} Definition 6.2.11 or \cite{BC2} proof of lemma 4.3.3). For all $z \in Z'$, $\varphi_z$ is killed by
\[ P = \prod_{i=1}^3(T - \wt_i(z)),\]
by proposition \ref{propB1}, and as $Z'$ is Zariski dense, and $\mathcal O(U)$ is reduced we get that $P$ kills $\varphi$ on $U$, and reducing to $A/I$ we get the result.
\end{proof}

Fix the bijection between $\{1,2,3\}$ and $\{1,\overline\chi,\eps\}$ corresponding to the refinement \ref{raf}, i.e.
\[
\begin{array}{ccc}
1  & \mapsto  & \overline\chi  \\
2  &  \mapsto & 1  \\
 3 &  \mapsto &  \eps 
\end{array}
\text{if $p$ splits,} \quad 
\begin{array}{ccc}
1  & \mapsto  & 1  \\
2  &  \mapsto & \overline\chi  \\
 3 &  \mapsto &  \eps 
\end{array}
\text{if $p$ is inert.}
\]
Thus it makes sense to speak about $T_i, i \in \{1,2,3\}$.

\begin{lemma}
\label{lemmaB4}
For all $i$, the $A/I$-module
\[ H^0_{\varphi,\Gamma}(D_{rig}(T_i)(\delta_i^{-1})\]
is free of rank 1.\todo{revoir, ça dépends de $p$ inert ou non?}
\end{lemma}

\begin{proof}
We will consider inductively the pseudocharacters $T$, $\Lambda^2 T$ and $\det T$ whose reduction is respectively $1 \oplus \overline \chi \oplus \eps$, $\overline\chi \oplus\eps\oplus \eps\overline\chi$ and $\eps\overline\chi$. In particular they are multiplicity free. Recall that for $I \supset I_{tot}$, $T$ splits, thus also $\Lambda^2T$, we denote $T_i' = T_1\dots T_i$ for $i = 1,2,3$. By induction on $i$, it is enough to prove the result for $T_i'$. In  particular for all $i$, we can find $M$ a $S$-module, finite type over $A$, of generic rank 3 (if $i = 1,2$, rank 1 and $M = T_3'$ if $i =3$) such that (\cite{BC2} Theorem 1.5.6 and Lemma 4.3.9).
\[ 0 \fleche K \fleche M/I \fleche T_i' \fleche 0,\]
with $K^{ss}$ reducing to a direct sum of $\prod_{k = 1}^i T_{j_k} \neq T_i'$. As $\delta_i(p) = F_i$ and at $y$ these values are
\[ (\overline\chi(p),1,p^{-1}) \text{ if $p$ splits,} \quad (1,\overline\chi(p),p^{-1}) \text{ if $p$ is inert},\]
which are distincts ($|\overline\chi(p)| = p^{-1/2}$), the slope of $\delta_1...\delta_i$ is distinct from the one appearing in $K$. In particular
\[ H^0_{\varphi,\Gamma}(D_{rig}(K(\delta_1...\delta_i)^{-1})) = \{0\}.\]
Thus, it suffices to show that $H^0_{\varphi,\Gamma}(D_{rig}(M(\delta_1...\delta_i)^{-1})))$ is free of rank 1 for every cofinite ideal $J$ of $A = \mathcal O_y$. But this is assured by \cite{KPX} Theorem 6.3.9 and \cite{BC2} Theorem 3.3.3 and Lemma 3.3.9. Indeed, first, by \cite{BC2} Lemma 4.3.7 we can find $U \subset \mathcal E$ containing $y$ an affinoid together with a coherent torsion free module $\mathcal M$ with an action of $G = G_{E,S}$ reducing to $M$ on $A = \mathcal O_y$, which is generically free of rank 3, and such that the trace of $G$ on $\mathcal M$ coincide with $T \otimes_{\mathcal O_{\mathcal E}} \mathcal O(U)$.
Denote $\delta_1^{(i)} = \delta_1\dots\delta_i$, and $H^0_{\varphi,\Gamma}(D_{rig}(-))$ is a functor as in \cite{BC2} Section 3.2.2.
Moreover, by \cite{BC2} Lemma 3.4.2 and \cite{KPX} Theorem 6.3.9 (applied to $\mathcal M'^\vee$ and $\delta = \delta_1^{(i),-1}$) there exists a birational morphism (see \cite{BC2} section 3.2.3)
\[ \pi :  U' \fleche  U,\]
such that the strict transform $\mathcal M'$ of $\mathcal M$ on $U'$ is locally free, and moreover we have a map
\[ D_{rig}(\mathcal M'^\vee) \fleche R_{U'}(\delta_1^{(i),-1})\otimes \mathcal L,\]
whose kernel is a $\varphi,\Gamma$-module of rank 2 (is trivial if $i =3$) and which is generically surjective.
Moreover it is proven in the course of the proof of \cite{KPX} Theorem 6.3.9 that $H^0_{\varphi,\Gamma}(D_{rig}((\mathcal M'^\vee)^\vee)(\delta_1^{-1}))$ is locally free of rank 1.
In particular, as these sheaves are coherent, for all $y' \in \pi^{-1}(y)$, and all cofinite length ideal $J'$ of $\mathcal O_{y'}$,
\[ H^0_{\varphi,\Gamma}(D_{rig}(\mathcal M' (\delta_1^{(i),-1})\otimes \mathcal O_{y'}/J')),\]
is free of rank 1. Indeed, we have the commuting diagram
\begin{center}
\begin{tikzpicture}[description/.style={fill=white,inner sep=2pt}] 
\matrix (m) [matrix of math nodes, row sep=3em, column sep=2.5em, text height=1.5ex, text depth=0.25ex] at (0,0)
{ 
H^0_{\varphi,\Gamma}(D_{rig}(\mathcal M')(\delta_1^{-1})) \otimes \mathcal O_{y'}/I' & H^0_{\varphi,\Gamma}(D_{rig}(\mathcal M')(\delta_1^{-1})\otimes \mathcal O_{y'}/I')  \\
H^0_{\varphi,\Gamma}(D_{rig}(\mathcal M')(\delta_1^{-1})) \otimes \mathcal O_{y'}/\mathfrak m_{y'} & H^0_{\varphi,\Gamma}(D_{rig}(\mathcal M')(\delta_1^{-1}) \otimes \mathcal O_{y'}/\mathfrak m_{y'}) \\
};
\path[->,font=\scriptsize] 
(m-1-1) edge node[auto] {$$} (m-1-2)
(m-2-1) edge node[auto] {$i$} (m-2-2)
(m-1-1) edge node[auto] {$f$} (m-2-1)
(m-1-2) edge node[auto] {$red$} (m-2-2)
;
\end{tikzpicture}
\end{center}
where the map $i$ is injective (\cite{KPX} eq 6.3.9.1). As the map $f$ is non-zero, the map $red$ is also non-zero. Thus by \cite{BC2} Lemma 3.3.9,
\[ H^0_{\varphi,\Gamma}(D_{rig}(\mathcal M')(\delta_1^{-1})\otimes \mathcal O_{y'}/I')\]
is free of rank one over $\mathcal O_{y'}/I'$.
Thus by \cite{BC2} Proposition 3.2.3 and Lemma 3.3.9, for all cofinite lenght ideal $J$ of $\mathcal O_y = A$, we have that 
\[ H^0_{\varphi,\Gamma}(D_{rig}(\mathcal M (\delta_1^{(i),-1})\otimes \mathcal O_{y}/J)),\]
is free of rank 1 over $A/J$.
\end{proof}

\begin{lemma}
\label{lemmaB5}
Suppose that  $D$ is a $\varphi,\Gamma$-module of rank 1 on an artinian ring $A$, and with Hodge-Tate weight $k = (k_\sigma)_{\sigma\in\Sigma} \in \ZZ^\Sigma$.
Fix \[ \delta : K^\times \fleche A^\times,\] and denote $(t_\sigma)_{\sigma \in \Sigma} \in \ZZ^{\Sigma}$ its Hodge-Tate weights. 
Suppose that
\[ H^0_{\varphi,\Gamma}(D(\delta^{-1})),\]
is free of rank 1 over $A$. Then $D = R_A(\delta')$ with $\delta' = \delta \prod_\sigma x_\sigma^{k_\sigma - t_\sigma}$.
\end{lemma}

\begin{proof}
Let $D = R_A(\delta')$ and by hypothesis we have a injective morphism of $R_A$-modules
\[ R(\delta) \hookrightarrow D= R_A(\delta').\]
Let $v$ be the image of a basis of $R(\delta)$, and denote by $e$ a basis of $D$.
Thus, $D' = R_Av$ is a sub-$\varphi,\Gamma$-module of $D$, isomorphic to $R_A(\delta)$.
Reducing modulo $\mathfrak m_A$, by \cite{KPX} corollary 6.2.9 we have that $\overline{D'} = \prod_\sigma t_\sigma^{l_\sigma}\overline D$ for some $l_\sigma \in \ZZ$.
But $\Gamma$ acts on $v$ as $\delta(\gamma)$. Moreover, using the previous equality, it also acts on $\overline v$ by
\[ \gamma \overline v = \prod_{\sigma} LT_\sigma(\gamma)^{l_\sigma}\delta'(\gamma)\overline v.\]
Thus, $\overline \delta_{|\Gamma} = (\prod_{\sigma}x_\sigma \overline\delta')_{|\Gamma}$, which by hypothesis gives
\[ l_\sigma = t_\sigma - k_\sigma.\]
Consider $M = \prod_\sigma t_\sigma^{-l_\sigma} R_Av$. Then $M$ is saturated in $D'$, thus $D' = M$.
But as $R_Av \simeq R_A(\delta)$, $M \simeq R(\prod_\sigma x_\sigma^{-l_\sigma} \delta)$, thus, by \cite{KPX} Lemma 6.2.13,
\[ \delta' = \delta\prod_\sigma x_\sigma^{k_\sigma - t_\sigma} .\]
\end{proof}

Recall (\cite{BC2} Lemma 8.27, that we have an injective map
\[ \iota_{T,i,j} : \Ext_T(i,j) \hookrightarrow \Ext_{k[G_{E,S}]}(i,j).\]
\begin{theor}
\label{theorB5}
Let $\rho : G \fleche \GL_{d_i+d_j}(A/I)$ an extension of $T_1$ by $T_i$ inside the image of $\iota_{T,i,1}$. Then, if $p$ splits, for $\star = v,\overline v$
\[ D_{crys,\star}(\rho_c(\delta_{1,|\Gamma}^{-1}))^{\varphi = F_1}\]
is free of rank 1 over $A/I$.  If $p$ is inert,
\[ D_{crys,\tau}(\rho_c(\delta_{1,|\Gamma}^{-1}))^{\varphi^2 = F_1}\] is free of rank 1 over $A/I$. 

\end{theor}

\begin{proof}
Recall that $1$ is the only constituent of $\rho_y$ which has $1 = p^{\wt_1}F_1$ as eigenvalue for its Frobenius. 
By \cite{BC2} Theorem 1.5.6 (2), there is an exact sequence,
\[ 0 \fleche K \fleche (M_j/IM_j \oplus \rho_i) \fleche \rho_c \fleche 0,\]
with $K^{ss}$ being a direct sum of $\overline{T_k}$, $k \neq 1$.
Thus, $D_{crys}(K(\delta_1^{-1}))^{\varphi = F_1} = D_{crys}(T_i(\delta_1^{-1}))^{\varphi = F_1} = \{0\}$. In particular, it is enough to prove that
\[ D_{crys}(M_j(\delta_1^{-1}))^{\varphi = F_1}\]
is free of rank 1 over $A$. We will use the same devissage as in \ref{lemmaB4}. By \cite{BC2} Lemma 4.3.9, there exists $M = M_j \oplus N_j$ such that $MK = K^3$ a sub-$A[G]$ module of $K^3$ of finite type over $A$. Extending this module to an affinoid $U \subset \mathcal E$ containing $y$, and using the accumulation of $Z'$ at $y$ (Proposition \ref{propB1}), we can find a birational morphism $\pi : U' \fleche U$ and $\mathcal M'$ the strict transform of $\mathcal M$, locally free on $U'$, for which the conclusion of \cite{KPX} Theorem 6.3.9 for $(\mathcal M')^\vee$ and $\delta_1^{-1}$ applies. In particular 
\[H^0_{\varphi,\Gamma}(D_{rig}(\mathcal M')(\delta_1^{-1}))\]
is locally free of rank one on $U'$. 

As in Lemma \ref{lemmaB4} we can specialize 
at $\mathcal O_{y'}$ for every $y'$ above $y \in U$. But we have the commuting diagram
\begin{center}
\begin{tikzpicture}[description/.style={fill=white,inner sep=2pt}] 
\matrix (m) [matrix of math nodes, row sep=3em, column sep=2.5em, text height=1.5ex, text depth=0.25ex] at (0,0)
{ 
H^0_{\varphi,\Gamma}(D_{rig}(\mathcal M')(\delta_1^{-1})) \otimes \mathcal O_{y'}/I' & H^0_{\varphi,\Gamma}(D_{rig}(\mathcal M')(\delta_1^{-1})\otimes \mathcal O_{y'}/I')  \\
H^0_{\varphi,\Gamma}(D_{rig}(\mathcal M')(\delta_1^{-1})) \otimes \mathcal O_{y'}/\mathfrak m_{y'} & H^0_{\varphi,\Gamma}(D_{rig}(\mathcal M')(\delta_1^{-1}) \otimes \mathcal O_{y'}/\mathfrak m_{y'}) \\
};
\path[->,font=\scriptsize] 
(m-1-1) edge node[auto] {$$} (m-1-2)
(m-2-1) edge node[auto] {$i$} (m-2-2)
(m-1-1) edge node[auto] {$f$} (m-2-1)
(m-1-2) edge node[auto] {$red$} (m-2-2)
;
\end{tikzpicture}
\end{center}
where the map $i$ is injective (\cite{KPX} eq 6.3.9.1), the map $f$ is non-zero, thus the map $red$ is also non-zero. By \cite{BC2} Lemma 3.3.9,
\[ H^0_{\varphi,\Gamma}(D_{rig}(\mathcal M')(\delta_1^{-1})\otimes \mathcal O_{y'}/I')\]
is free of rank one over $\mathcal O_{y'}/I'$ for all $y' \in \pi^{-1}(y)$ and $I'$ of cofinite length.
Thus the hypothesis of \cite{BC2} Proposition 3.2.3 are satisfied, and
by \cite{BC2} Lemma 3.3.9 again,
\[ H^0_{\varphi,\Gamma}(D_{rig}(M \otimes A/I)(\delta_1^{-1}))\]
is free of rank 1. In particular we have an injection of $R_{A/I}$-modules,
\[ 0 \fleche R_A \fleche D_{rig}(M\otimes A/I)(\delta_1^{-1}) \fleche Q \fleche 0.\]
Moreover, as the reduction to $A/\mathfrak m_A$ of $D_{crys}(M(\delta_1^{-1}))^{\varphi = 1}$ is of rank 1,
using the functor $D_{cris}$, we have that $D_{cris}(1) \subset D_{cris}(M(\delta_1^{-1}))^{\varphi = 1}$ and thus $D_{crys}(Q)^{\varphi = 1} = \{0\}$.
In particular $D_{crys}(M(\delta^{-1}))^{\varphi = 1} = D_{crys}(1)$ is free of rank 1 over $A$, and thus
\[ D_{crys}(M(\delta_{1|\Gamma})^{-1})^{\varphi = F_1}\]
is free of rank 1 over $A$. The same proof remains valid in the case where $p$ is inert.
\end{proof}

\bibliographystyle{smfalpha} 
\bibliography{biblio} 

\end{document}